\RequirePackage[table]{xcolor}
\documentclass{article}
\usepackage{geometry}
\usepackage{xcolor}
\usepackage{authblk}
\usepackage{ytableau}
\usepackage{bm}
\usepackage{yfonts}
\usepackage[english]{babel}
\usepackage[T1]{fontenc}
\usepackage[latin1]{inputenc}
\usepackage{dsfont}
\usepackage{mathrsfs}
\usepackage{amsmath}
\usepackage{bbm}
\usepackage{amssymb}
\usepackage{verbatim}
\usepackage{amsthm}
\usepackage{graphicx}
\usepackage{xcolor}
\usepackage{fancybox}
\usepackage{algpseudocode}
\usepackage{enumitem}
\usepackage{eurosym}
\usepackage{amsthm}
\usepackage{xr}

\usepackage[colorinlistoftodos,prependcaption,textsize=tiny]{todonotes}
\DeclareMathOperator*{\argmin}{arg\,min} 

\allowdisplaybreaks
\newtheorem{theorem}{Theorem}[section]

\newtheorem{lemma}[theorem]{Lemma}

\newtheorem{definition}{Definition}

\newtheoremstyle{named}{}{}{\bfshape}{}{\bfseries}#1{\thmnote{#3}}
\theoremstyle{named}

\newtheoremstyle{named1}{}{}{\bfshape}{}{\bfseries}#1{\thmnote{#3}}
\theoremstyle{named1}

\def\ra{\rangle}
\def\la{\langle}
\def\bb{\mathbb}
\def\bm{\mathbf{m}}
\def\cal{\mathcal}
\def\frak{\mathfrak}
\def\Tr{\mathrm{Tr}}
\def\ijr{1\leq i\leq r,i<j\leq d}

\newcommand{\Vb}{{f_{\bf b}}}
\newcommand{\Vm}{{f_{\bf m}}}
\newcommand{\Vzero}{{f_{\bf 0}}}
\newcommand{\Vl}{{f_{\bf l}}}
\newcommand{\Va}{{f_{\bf a}}} 
\newcommand{\Vbp}{{f_{\bf b'}}}
\newcommand{\Vap}{{f_{\bf a'}}}
\newcommand{\lc}{{l(c)}}
\newcommand{\tcb}{{t_{\bf b}^c}}
\newcommand{\tca}{{t_{\bf a}^c}}
\newcommand{\ssc}{{s_c}}

\newcommand{\bz}{\mathbf{z}}

\newcommand{\sumtwo}[2]{\sum_{\substack{#1 \\ #2}}} 
\makeatletter
\newcommand*{\addFileDependency}[1]{
  \typeout{(#1)}
  \@addtofilelist{#1}
  \IfFileExists{#1}{}{\typeout{No file #1.}}
}
\makeatother

\newcommand*{\myexternaldocument}[1]{%
    \externaldocument{#1}%
    \addFileDependency{#1.tex}%
    \addFileDependency{#1.aux}%
}
\myexternaldocument{low_LAN_supp_rev_v6}
\newcommand\blfootnote[1]{%
  \begingroup
  \renewcommand\thefootnote{}\footnote{#1}%
  \addtocounter{footnote}{-1}%
  \endgroup
}
\begin{document}

\title{\textbf{Minimax estimation of low-rank quantum states and their linear functionals}}
\date{}

\author[1]{Samriddha Lahiry \footnote{Corresponding author. Email: slahiry@fas.harvard.edu}}
\author[2]{Michael Nussbaum \footnote{ Supported in part by NSF Grant DMS-1915884. Email: nussbaum@math.cornell.edu}}
\affil[1]{Department of Statistics, Harvard University,}
\affil[2]{Department of Mathematics, Cornell University}

\renewcommand\Authands{ and }

\maketitle

\vspace{0.1in}

\begin{abstract}
~ In classical statistics, a well known paradigm consists in establishing asymptotic equivalence between an experiment of i.i.d. observations and a Gaussian shift experiment, with the aim of  obtaining optimal estimators in the former complicated model from the latter simpler model. In particular, a statistical experiment consisting of $n$ i.i.d observations from d-dimensional multinomial distributions can be well approximated by an experiment consisting of $d-1$ dimensional Gaussian distributions. In a quantum version of the result, it has been shown that a collection of $n$ qudits (d-dimensional quantum states) of full rank can be well approximated by a quantum system containing a classical part, which is a $d-1$ dimensional Gaussian distribution, and a quantum part containing an ensemble of $d(d-1)/2$ shifted thermal states. In this paper, we obtain a  generalization of this result when the qudits are not of full rank. We show that when the rank of the qudits is $r$, then the limiting experiment consists of an $r-1$ dimensional Gaussian distribution and an ensemble of both shifted pure and shifted thermal states. For estimation purposes, we establish an asymptotic minimax result in the limiting Gaussian model. Analogous results are then obtained for estimation of a low rank qudit from an ensemble of identically prepared, independent quantum
systems, using the local asymptotic equivalence result. We also consider the problem of estimation of a linear functional of the quantum state. We construct an estimator for the functional, analyze the risk and use quantum local asymptotic equivalence to show that our estimator is also optimal in the minimax sense.
\end{abstract}

\section{\textbf{Introduction}}
\blfootnote{\textit{Keywords and phrases}:  low rank states, quantum LAN,, quantum minimax estimation, functional estimation}
 Recent breakthroughs in quantum technology, such as quantum computing, communication, and metrology (cf. \cite{MR4071358}) have created renewed interest in the probabilistic and statistical problems that arise from quantum information theory. In particular, the age old tools of mathematical statistics have found their way into the toolbox of scientists working with quantum information.
 
 One of the fundamental problems in quantum statistics is quantum estimation, i.e. estimation of a quantum state or functionals of a state. In the case of state estimation one deals with a statistical inverse problem of inferring unknown state parameters from the measurement data obtained by probing a large number of individual quantum systems. On the other hand, in the case of functional estimation, one is interested in measuring some particular property of the state instead of the whole state. In analogy to the classical decision theoretic approach, one can develop a \textit{quantum decision theoretic framework} for such problems where the ``best estimate'' also involves optimizing over the measurements. In this direction, a \textit{quantum Cramer-Rao bound} has been obtained in \cite{Braunstein&Caves, Helstrom,Hol-2nd-ed-11} for the covariance matrix of an unbiased estimator, as a first step in uncovering this framework. Another closely related problem pertains to the discrimination between  quantum states, in which case one can devise optimal testing procedures. In the asymptotic setup,  the \textit{quantum Stein's lemma} (cf. \cite{MR1806811}) and the \textit{quantum Chernoff bound} have been established (cf. \cite{Audenaert,Audenaert&Nussbaum,Li,Nussbaum&Szkola}). For an overview of the evolving literature on
quantum statistics cf.  \cite{Gill&Guta,MR2017871,MR2761295,GutaJencova2007,BGN-QAE} and
references therein, and the monographs  \cite{MR3558531,MR2363070,Hol-2nd-ed-11}. Statistical aspects of quantum
algorithms and simulation are  discussed in \cite{MR2840170,MR3012432,MR3552740}.

In classical statistics, the first step for obtaining optimal decision theoretic procedures often consists in approximating complicated experiments (families of laws, or models) by simpler ones (cf.  \cite{LeCam} for details). An important example under this paradigm is called local asymptotic normality (LAN), where one establishes asymptotic equivalence between i.i.d. models indexed by a local parameter and a Gaussian shift model (with the shift given by the same local parameter). To construct an optimal estimator, one constructs a preliminary estimate of the parameter first and then uses LAN in the neighborhood of the estimated value. The optimal procedure in the Gaussian model then ensures an (asymptotically) optimal procedure in the i.i.d. model. Global equivalence has also been shown between several nonparametric estimation problems like nonparametric regression (cf. \cite{BrownLow}), density estimation (cf. \cite{Nussbaum96}) and the Gaussian white noise model. The theory of LAN can also be used to show optimality in the estimation of functionals by restating the functional estimation problem as a smooth parametric estimation problem (cf. \cite{MR620321}). In the quantum setup, \textit{quantum LAN theory}  for parametric models, established in  \cite{GutaJencova2007,Guta_Kahn_qubit,Kahn&Guta,YFG13}, shows that a model given by a large collection of identically prepared finite dimensional states can be approximated by a quantum Gaussian shift model in a local neighborhood. An extension of  quantum LAN theory towards a  quantum version of  local asymptotic equivalence in nonparametric models can be found in \cite{BGN-QAE} where it is shown  that an ensemble of pure states in infinite dimensional Hilbert space can be approximated by coherent states, which constitute a quantum counterpart of the Gaussian white noise model.

Our contribution in this paper is threefold; as a first result we establish LAN for low-rank quantum states. We restrict ourselves to finite dimensional states, in particular to a setup similar to \cite{Guta&Janssens&Kahn,Guta_Kahn_qubit,Kahn&Guta}. In \cite{Guta_Kahn_qubit} the authors show that a quantum statistical experiment consisting of a large number of qubits (two dimensional quantum states) with rotational shift can be approximated by a quantum Gaussian state, specifically a shifted thermal state, while in \cite{Guta&Janssens&Kahn} the authors improved the result to include diagonal perturbations of the qubits. They showed that the limiting experiment in this case contains a classical part which is Gaussian and a quantum part which is a shifted thermal state. Later in \cite{Kahn&Guta} the authors extended the results for qubits to qudits (d-dimensional quantum systems) and showed that a large number of full rank qudits can be well approximated by a quantum experiment consisting of a 
$d-1$ dimensional ``classical'' Gaussian part (corresponding to the diagonal elements of the original qudit model) and a quantum part which is a tensor product of $d(d-1)/2$ shifted thermal states (corresponding to the off-diagonal elements of the original qudit model). The latter paper used technical results on symmetric and general linear groups but was crucially dependent on the assumption that the original qudit was of full rank. We extend this result to the case of rank-r qudits and obtain an interesting limiting experiment. While the classical part contains an $r-1$ dimensional classical Gaussian as expected from the degeneracy of the eigenvalues, the quantum part contains a tensor product of both shifted thermal states ($r(r-1)/2$ many) and shifted pure states ($r(d-r)$ many). We observe that for the full rank case there are only shifted thermal states in the quantum part of the limit as in \cite{Guta&Janssens&Kahn} and \cite{Kahn&Guta}, while for the rank $r=1$ case we only have shifted pure states (without a classical part). On the other hand, one can easily modify the proof of local asymptotic equivalence in \cite{BGN-QAE} to show that the limiting experiment of an ensemble of finite dimensional pure states gives rise to a tensor product of $d-1$ shifted pure states (without a classical part) in the limit. Our result for rank $1$ qudits matches this result. Thus we obtain a unifying picture that suitably generalizes previous results on finite dimensional quantum LAN.

Secondly, we address the question of optimal estimation of low-rank qudits using the LAN result. Recently there has been considerable interest in the study of low-rank quantum models, although in most cases the authors study them in the high dimensional low-rank setup and obtain finite sample bounds in contrast to our asymptotic setup. In the framework of low-rank quantum state tomography, the problem is treated in \cite{Carpentier,Koltchinskii-von-Neumann,Kolt-Xia-MLearn} where the authors consider a trace regression model with the observables being random and then establish minimax bounds in the estimation of the low-rank state. A similar problem has been studied in \cite{Cai-Zhou-et-al} but with a sparsity assumption different from the low-rank structure. The noiseless case (also called quantum compressed sensing) has been studied in \cite{Flammia&Gross,Gross_compressed,Liu2011UniversalLM} and similar methods of quantum tomography (using a small number of randomized measurements to reconstruct a low rank state) have been considered in \cite{Acharya_2017,Acharya_2016}. A spectral thresholding and a rank penalized method are used for the low-rank quantum state tomography problem in \cite{ButuceaGutaKypraios} and \cite{Butucea_rank} respectively, while \cite{MR4093475} uses a projected least square approach. A comparative study of these approaches can be found in \cite{Acharya_2019}. However, most of these methods are only optimal in the rate sense and fail to obtain sharp constants. We employ the methods of \cite{ButuceaGutaKypraios} and \cite{MR4093475} to obtain a preliminary estimator which is both rank consistent and lies close to the original state and then use our LAN result to obtain the minimax risk up to sharp constants. The optimal minimax risk for the full rank case is obtained in \cite{GutaKahnproc}; our result for the rank $r$ case  matches with former for $r=d$.

Finally, in the spirit of \cite{MR620321,KosLev76}, we use LAN to address the question of optimal estimation of linear functionals of low-rank quantum states. While property estimation is a well-established concept in quantum information theory (for example see \cite{wagner_issa} for estimation of von-Neumann entropy), the optimal procedures for estimation of a linear functional of a quantum state are less well known. In this direction, the quantum analog of the concept of a least favorable sub-family is treated in \cite{tsang_semipara}. We construct an estimator of the linear functional and use the latter concept  and LAN to show optimality (in the sense of sharp constants) of our estimator. It should be noted that while \cite{tsang_semipara} gives a pointwise optimality result, we establish a minimax result; the classical analog of the latter can be found in \cite{KosLev76}.

\subsection{Outline of results} The paper is organized as follows. In Section 2, we review the basic quantum mechanical concepts of states, measurements, observables, and quantum channels. Section 3 reviews classical parametric LAN and quantum LAN for the full rank case as obtained in
\cite{Kahn&Guta}. It also contains our main theorem (Theorem 3.3) where we establish LAN for the low-rank model.

Section 4 describes the Bayes estimator in the Gaussian case from \cite{Holevo1973,Hol-78}, which will be useful in our construction of the optimal estimator and also in establishing the minimax lower bound.
Analogously to the classical case we also observe a shrinkage phenomenon for this Bayes estimator.

In Section 5 we construct the optimal estimator for rank $r$ qudits. First, in Theorem 5.1, using a part of the sample, we obtain an estimator which lies close to the original state and has rank $r$ with high probability. We then use the quantum channel on the remaining sample to transfer the estimation problem of qudits to the estimation of parameters in a limiting Gaussian model which contains a classical Gaussian law and a tensor product of both shifted and pure thermal states in its quantum part. An asymptotic minimax result for the limiting Gaussian model is given in Theorem 5.2. To prove this theorem we use the result from Section 4 to get a lower bound of the minimax risk by the Bayes risk and then use a covariant measurement to give a matching upper bound. Finally, we observe that the Hilbert-Schmidt norm between the two qudits is locally quadratic approximately and then transfer the risk from the Gaussian model to obtain the optimal risk in the low-rank qudit model in Theorem 5.4.

Estimation of a linear functional is treated in Section 6. We construct an estimate of the functional in question using the appropriate observable to give an upper bound. The lower bound is given by constructing a least parametric subfamily and then using LAN to obtain a lower bound in the limiting Gaussian model. The latter is indexed by a one-dimensional parameter and we use a Bayesian result from \cite{Hol-78} to give the lower bound.
 
 Some representation theoretic tools are needed for proving Theorem 3.3;  they are reviewed in Appendix A. The Bayesian result for a one dimensional parameter is discussed in Appendix B. Proofs of the main theorems are given in Appendix C while the proofs of the more technical lemmas are deferred to Appendix D. The appendices are included in \cite{low_LAN_supp}.
 \subsection{Notation}
 In physics, the vectors of a Hilbert
space $\mathcal{H}$ (assumed separable) are written as ``ket'' $|v\rangle$, $v^*$ (a vector in the dual space $\mathcal{H}^*$) as ``bra'' $\langle v|$ and the
inner product of two vectors as the ``bra-ket'' $\langle u|v\rangle\in
\mathbb{C}$ which is linear with respect to the right entry and anti-linear
with respect to the left entry. Similarly, $M:=|u\rangle\langle v| $ is the
rank one operator acting as $M  : |w\rangle \mapsto M|w \rangle = \langle
v |w\rangle |u\rangle$.  For an operator $A$ the expression $\langle u|Av\rangle$ will sometimes be denoted as $\langle u|A|v\rangle$. The space of
bounded linear operators on $\mathcal{H}$ is denoted by $\mathcal{L}(\mathcal{H})$. Of particular interest are the following two subspaces of $\mathcal{L}(\mathcal{H})$.
\begin{enumerate}
    \item $\mathcal{T}_1(\mathcal{H})\subset\mathcal{L}(\mathcal{H})$ - the trace class defined by $\mathcal{T}_1(\mathcal{H})=\{A:\mathcal{H}\rightarrow \mathcal{H}:\mathrm{Tr}(A^*A)^{1/2}<\infty\}$. Operators in $\mathcal{T}_1(\mathcal{H})$ are equipped with the norm $\mathrm{Tr}(A^*A)^{1/2}$.\\
    \item $\mathcal{T}_2(\mathcal{H})\subset\mathcal{L}(\mathcal{H})$ - the Hilbert Schmidt operators defined by $\mathcal{T}_2(\mathcal{H})=\{A:\mathcal{H}\rightarrow \mathcal{H}:\mathrm{Tr}(A^*A)<\infty\}$. Operators in $\mathcal{T}_2(\mathcal{H})$ are equipped with the norm $(\mathrm{Tr}(A^*A))^{1/2}$. The class $\mathcal{T}_2(\mathcal{H})$ is a Hilbert space with respect to the inner product $%
(A, B) := \mathrm{Tr}(A^* B)$.
\end{enumerate}

It is well known that $\mathcal{T}_1(\mathcal{H})\subset\mathcal{T}_2(\mathcal{H})$. For any Hilbert space, the usual norm will be denoted by $||.||$ and the identity operator on that space by $\mathbf{1}$ where the particular space will be understood from the context. We will denote by $||\mu-\nu||_{\mathrm{TV}}$ the total variation norm between two measures $\mu$ and $\nu$. By $a\vee b$ and $a\wedge b$ we will denote $\max(a,b)$ and $\min(a,b)$ respectively and $a_+$ will be used to denote $a \vee 0$. By $\lfloor a\rfloor$ and $\lceil a \rceil$, we will denote the largest integer less than or equal to $a$ and the smallest integer greater than or equal to $a$ respectively. We will use the notation $a_n\asymp b_n$ whenever $c<\liminf_n (a_n/b_n) \leq \limsup_n (a_n/b_n) <C$ for some constants $c,C>0$. Throughout the paper, $c$ and $C$ will denote arbitrary constants.

\section{Quantum mechanics preliminaries}
The outline for this section is as follows. In Subsection 2.1 we describe the concepts of quantum states, measurements and observables. We also discuss  quantum channels which are essential for exchanging information between two quantum systems. In Subsection 2.2 we consider qudits or d-dimensional quantum states;  the problem of finding optimal  measurements for these states and their linear functionals  is the main objective of this article. For optimal estimation of qudits (which will also be called an i.i.d. model) it is convenient  to analyze a quantum Gaussian model first and then relate it to the i.i.d. model. Following this path, we describe quantum Gaussian states in Subsection 2.3, in particular we describe a limiting Gaussian model. In Section 3 we state the theorem (Theorem 3.3) which guarantees that the i.i.d. model can be approximated by the Gaussian model described in Subsection 2.3. In subsection 2.4, we describe the general problem of quantum statistical inference and the concepts of Bayes and minimax risk in quantum models, which are used as benchmarks for optimality. We also describe the concept of quantum asymptotic equivalence and how risks can be transferred between two asymptotically equivalent models using quantum channels. In Section 5 we  compute the risk in the limiting Gaussian model (Theorem 5.2) and then use the concept of risk transfer together with Theorem 3.3 to obtain the optimal risk in the i.i.d. model. A similar method of risk transfer will also be employed to establish optimality of an estimator of a linear functional in Section 6.
\subsection{States, measurements, and observables}
A \textit{state of a quantum system} is described by a self-adjoint trace class operator
$\rho$ on a complex Hilbert space {$\mathcal{H}$, which is positive }%
$(\rho\geq0$) and normalized to $\mathrm{Tr}\left(  \rho\right)  =1$ (a
density operator). A state is called \textit{pure} if it is of the form $\rho=|\psi\rangle\langle\psi|$, otherwise it is called a \textit{mixed state}. We denote the set of states by $\mathcal{S}(\mathcal{H})$. 

Data on a quantum system are obtained from
\textit{observables} which are self-adjoint operators $S$ in the Hilbert space $\mathcal{H}$. If  $S$ has spectral decomposition $S=\sum_{j}\lambda_{j}\Pi_{j}$ where $\Pi_{j}$s
are projectors, then a measurement generates a \textit{discrete random variable} $X_{S}$ taking
values in the set of eigenvalues $\left\{  \lambda_{1},\lambda_{2}%
,\ldots\right\}  $ with probabilities $p_{j}=\mathrm{Tr}(\rho\cdot\Pi_{j})$. The expectation of
$X_{S}$ under the state $\rho$ is then given by the \textit{Born-von Neumann postulate:}%
\begin{equation*}
E_{\rho}X_{S}=\sum_{j}\lambda_{j}\mathrm{Tr}\left(  \rho\Pi_{j}\right)
=\mathrm{Tr}\left(  \rho S\right)  . \label{trace-rule}%
\end{equation*}
In quantum mechanics, one needs generalized versions of the above definitions of observables and measurements because the spectral decomposition of self-adjoint operators in the form of a weighted sum of projectors may fail to hold when the Hilbert space is infinite dimensional. If a measurement has outcomes in a measurable space $(\Omega,\mathfrak{B})$, it is determined by a positive operator-valued measure.
\begin{definition}\label{def.POVM}
A positive operator valued measure (POVM) is a map $M:\mathfrak{B}\to \mathcal{L}(%
\mathcal{H})$ having the following properties
\begin{itemize}
\item[1)] positivity: $M(B) \geq 0$ for all events $B\in\mathfrak{B}$ (hence M(B) is self-adjoint)
\item[2)] $\sigma$-additivity: $M(\cup_i B_i) = \sum_i {M}(B_i)$ for
any countable set of mutually disjoint events $B_i$ (here the convergence is in the weak operator topology of $\mathcal{L}(\mathcal{H})$)
\item[3)] normalization: $M(\Omega) = \mathbf{1}$.
\end{itemize}
\end{definition}

If the operators $M(B)$ are also orthogonal projections, i.e. $M(A)^2=M(A)$ and $M(B)M(A)=0$ when $A\cap B=\emptyset$, then it is called a \textit{simple measurement}. The collection of projectors $\{\Pi_j\}$ in the spectral decomposition $S=\sum_j \lambda_j\Pi_j$ is an example of a simple measurement. The outcome of the measurement has probability
distribution
\begin{equation}
P_\rho(B) = \mathrm{Tr}(\rho M(B)), \qquad B\in \mathfrak{B}.
\label{prob_measurement}
\end{equation}
 The spectral theorem shows that any self-adjoint operator $S:\mathcal{H}\rightarrow \mathcal{H}$ can be diagonalized as follows:
\begin{equation*}S=\int_{\sigma(S)}xM(dx),\label{general-spectral}\end{equation*}
where $\sigma(S)$ is the spectrum of $S$ and $M$ is a POVM, also called spectral measure associated with the operator $S$. When $S$ is an observable with a continuous spectrum, it generates a \textit{continuous random variable} $X_S$ with probabilities given by (\ref{prob_measurement}). Also, it easily follows that
$$E[X_S]=\mathrm{Tr}(S\rho).$$
The expected value of an observable $S$ is often denoted as $\langle S \rangle$, when the state dependence is not explicitly shown. There are POVMs (called \textit{generalized measurements}) where the orthogonality does not hold, but these can be extended to a POVM in a larger Hilbert space where the extended version is orthogonal.
Let $\mathit{POVM}(\Omega,\mathcal{H})$ be the set of POVMs with values in $\mathcal{L}(\mathcal{H})$ and outcome space $\Omega$ and let  $\mathcal{H}_0$ be another Hilbert space with a density operator $\rho_0$. Then any simple measurement $M'$ in $\mathit{POVM}(\Omega,\mathcal{H}\otimes \mathcal{H}_0)$ induces a measurement $M$ in $\mathit{POVM}(\Omega,\mathcal{H})$ which is determined by
$$\mathrm{Tr}(\rho M(B))=\mathrm{Tr}((\rho\otimes \rho_0)M'(B)),B\in \mathfrak{B},$$
for all states $\rho$ on $\mathcal{H}$. The pair $(\mathcal{H}_0,\rho_0)$ is called an \textit{ancilla} and it is known that (cf. \cite{Hol-2nd-ed-11}, Section 2.5) given any measurement $M$ in $\mathcal{H}$, there exists an ancilla $(\mathcal{H}_0,\rho_0)$ and a simple measurement $M'$ such that the above equation holds. The triple $(\mathcal{H}_0,\rho_0,M')$ is called a \textit{realization} of the measurement $M$  and the notion of adding an ancilla before taking simple measurements is called \textit{quantum randomization} in \cite{MR2017871}.

In many cases, it is convenient to perform a measurement after ``changing'' the state of the original system by interacting with other systems. The
maps describing such transformations are called quantum channels.
\begin{definition}
A quantum channel between systems with Hilbert spaces $\mathcal{H}_{1}$ and
$\mathcal{H}_{2}$ is a mapping $\ T$ which assigns to every state $\rho$ on
$\mathcal{H}_{1}$ the state $T(\rho)$ on $\mathcal{H}_{2}$ given by
\begin{equation*}
T(\rho)=\sum_{i=1}^{\infty}{K}_{i}\rho K_{i}^{\ast},%
    \label{eq.kraus}
\end{equation*}

where $\left\{  {K}_{i}\right\}  $ are bounded  operators $K_{i}%
:\mathcal{H}_{1}\rightarrow\mathcal{H}_{2}$ such that  $\sum_{i=1}^{\infty
}K_{i}^{\ast}K_{i}=\mathbf{1}$ (the series converging  in the strong
operator topology of $\mathcal{L}(\mathcal{H}))$. 
\end{definition}
It can be shown that the map $T$ is trace preserving and \emph{completely positive}, i.e. $\mathrm{Id}_m \otimes T$ is
positive for all $m\geq 1$, where $\mathrm{Id}_m$ is the identity map on the space of $m$ dimensional matrices. The simplest example of a quantum channel is a transformation $%
\rho\mapsto U\rho U^*$, where $U$ is a unitary operator on $\mathcal{H}$.
More generally, if $|\varphi \rangle\in \mathcal{K}$ is a pure state of an
ancillary system, and $V$ is a unitary on $\mathcal{H}\otimes \mathcal{K}$,
then
\begin{equation*}
\rho\mapsto T(\rho):= \mathrm{Tr}_\mathcal{K} ( V(\rho\otimes
|\varphi\rangle \langle\varphi |)V^* ) 
\end{equation*}
is a quantum channel where $\mathrm{Tr}_\mathcal{K} $ is the partial trace over $%
\mathcal{K}$ (with respect to an orthonormal basis $\{|f_i\rangle \}_{i=1}^{%
\mathrm{dim} \mathcal{K}}$). If we define operators $K_i$ on $\mathcal{H}$ such that $\langle \psi |K_i |
\psi^\prime\rangle :=\langle \psi \otimes f_i | V| \psi^\prime\otimes
\varphi \rangle$, then it can be seen that  $T(\rho)$ can be written as in the form given in Definition 2.
We define a dual map $T^{\ast}$ of a quantum channel as follows
\begin{align*}
T^{\ast}   :\mathit{POVM}(\Omega,\mathcal{H}_{2})&\rightarrow \mathit{POVM}(\Omega
,\mathcal{H}_{1})\\
T^{\ast}(M)(B)  & =\sum_{i=1}^{\infty}K_{i}^{\ast}M(B)K_{i},
\end{align*}
where the $\sum_{i=1}^{\infty}K_{i}^{\ast}M(B)K_{i}$ is a strongly convergent sum. From the definition, it can be easily verified that $T^{\ast
}(M)$ is indeed an element of $\mathit{POVM}(\Omega,\mathcal{H}_{1})$
(i.e. a POVM satisfying properties 1,2 and 3 of Definition 1) and that it satisfies the following duality relation
 $$\mathrm{Tr}(\rho T^*(M)(B))=\mathrm{Tr}(T(\rho)M(B)), \quad\forall B \in \mathfrak{B}$$

and all states $\rho$ on $\mathcal{H}_1$ (cf. 29.9 of \cite{Partha-book}).

For estimation purposes, we will need to define the following distances between two quantum states. The \emph{trace-norm} distance between two states $\rho_{0}%
,\rho_{1}\in\mathcal{S(\mathcal{H})}$ is given by
\[
\Vert\rho_{0}-\rho_{1}\Vert_{1}:=\mathrm{Tr}(|\rho_{0}-\rho_{1}|),
\]
where $|\tau|:=\sqrt{\tau^{\ast}\tau}$ denotes the absolute value of $\tau$. An interpretation of this metric in terms of quantum testing can be found in \cite{MR3558531}.
In the special case of pure states, the trace-norm distance is given by
\begin{equation}
\Vert|\psi_{0}\rangle\langle\psi_{0}|-|\psi_{1}\rangle\langle\psi_{1}%
|\Vert_{1}=2\sqrt{1-|\langle\psi_{0}|\psi_{1}\rangle|^{2}}.
\label{eq.trace.norm.pure}%
\end{equation}

Similarly one can define the $L^2$ distance between two states induced by the Hilbert-Schmidt norm as 
\[
\Vert\rho_{0}-\rho_{1}\Vert_{2}:=[\mathrm{Tr}((\rho_{0}-\rho_{1})^*(\rho_{0}-\rho_{1}))]^{1/2}.
\]
\subsection{Qudits under local parametrization}
 Consider a qudit or a $d$-dimensional density matrix, i.e.
$\rho\in M_d(\bb{C})$ (the space of $d\times d$ complex matrices), with $\rho\geq 0$ and $\mathrm{Tr}(\rho)=1$. A natural way to parametrize qudits is to write it in the form $\mathscr{U}(\zeta)\rho_0 \mathscr{U}^*(\zeta)$ with $\rho_0=$diag$(\mu_1,\mu_2,\ldots,\mu_d)$ and $\sum_{i=1}^d\mu_i=1$, where $\mathscr{U}(\zeta)$ are unitary matrices, i.e. elements of the Lie group $SU(d)$ parametrized by $\zeta\in \mathbb{C}^{d(d-1)/2}$. We describe the unitaries in more detail.
Consider the generators of the Lie algebra $\mathfrak{su}(d)$ (cf. \cite{MR1153249}):
\begin{align*}
    H_j&=E_{jj}-E_{j+1,j+1} \text{ for }1\leq j\leq d-1\\
    T_{j,k}&=iE_{j,k}-iE_{k,j}\text{ for }1\leq j<k\leq d\\
    T_{k,j}&=E_{j,k}+E_{kj}\text{ for }1\leq j<k\leq d,
\end{align*}
where $E_{i,j}$ is the matrix with $(i,j)^{th}$ entry equal to $1$, and all other entries equal to $0$. 

Assume $\mu_1>\mu_2>\ldots>\mu_d>0$ and define $$\mathscr{U}(\zeta)=\exp\left[i\left(\sum_{1\leq j<k\leq d}\frac{Re(\zeta_{j,k})T_{j,k}+Im(\zeta_{j,k})T_{k,j}}{\mu_j-\mu_k}\right)\right].$$

Next, we consider local models by first perturbing the eigenvalues only:
$$\rho_{0,u}=diag(\mu_1+u_1,\mu_2+u_2,\ldots,\mu_d+u_d),$$
where $\sum u_i=0$; that gives a state in a local neighborhood of $\rho_0$. To describe all possible states in the local neighborhood of $\rho_0$, we should also consider rotations using unitaries, i.e.
\begin{equation}\rho_{\vartheta}=\mathscr{U}(\zeta)\rho_{0,u}\mathscr{U}^{*}(\zeta),\label{local_qudit_full}
\end{equation}
where $\vartheta=(u,\zeta)$. Similarly in the low rank case we assume $\mu_1>\mu_2>\ldots>\mu_r>\mu_{r+1}=\ldots=\mu_d=0$ and define 
$$\mathscr{U}^r(\zeta)=\exp\left[i\left(\sum_{\substack{1\leq j\leq r\\j<k\leq d}}\frac{Re(\zeta_{j,k})T_{j,k}+Im(\zeta_{j,k})T_{k,j}}{\mu_j-\mu_k}\right)\right].$$

Next we consider a low-rank state in the local neighborhood of $\rho_{0,r}=diag(\mu_1,\mu_2,\ldots,\mu_r,0,\ldots,0)$:
\begin{align}\rho_{0,u,r}&=diag(\mu_1+u_1,\mu_2+u_2,\ldots,\mu_r+u_r,0,\ldots,0)\label{local_unrotated_qudit}\\\rho_{\vartheta,r}&=\mathscr{U}^r(\zeta)\rho_{0,u,r}\mathscr{U}^{r*}(\zeta).\label{local_qudit_r}\end{align}

We note that it is enough to parametrize low rank qudits using $\mathscr{U}^r(\zeta)$ instead of $\mathscr{U}(\zeta)$ since the first order terms in the Taylor expansion  of $\mathscr{U}^r(\zeta)\rho_{0,u,r}\mathscr{U}^{r*}(\zeta)$ and $\mathscr{U}(\zeta)\rho_{0,u,r}\mathscr{U}^*(\zeta)$ are identical and it is this term that determines the limiting model (cf. \cite{Acharya_2016}). 

Let $\mathfrak{z}_{ij}=\frac{\zeta_{ij}}{\sqrt{\mu_i-\mu_j}}$. For notational convenience we will denote $\theta=(u,\mathfrak{z})$ and redefine the unitaries $\mathscr{U}(\zeta)$ and $\mathscr{U}^r(\zeta)$ as $U(\mathfrak{z})$ and $U^r(\mathfrak{z})$ respectively. We will denote the state by $\rho_{\vartheta}$ ($\rho_{\vartheta,r}$ respectively) or $\rho_{\theta}$ ($\rho_{\theta,r}$ respectively) depending on whether it is indexed by $\vartheta$ or $\theta$.

In classical LAN we are interested in the limit for $n$ i.i.d. copies of the experiment with the local parameter lying in an $n^{-1/2}$ neighborhood of $0$ (see Section 3 for more details). In the quantum setup, this amounts to studying the following operators 
$$\rho^{\theta,n}=\rho^{\otimes n}_{\theta/\sqrt{n}},\quad\rho^{\theta,r,n}=\rho^{\otimes n}_{\theta/\sqrt{n},r}.$$

\subsection{Gaussian states and Fock spaces}

To obtain Gaussian random variables, in the
space {$\mathcal{H}$}$=L^{2}\left(  \mathbb{R}\right)  $ one considers two
special observables $Q,P$ with continuous spectrum:
\[
\left(  Qf\right)  \left(  x\right)  =xf\left(  x\right)  \text{, }\left(
Pf\right)  \left(  x\right)  =-i\frac{df}{dx}\left(  x\right)  ,\;\;\;\;f\in D \subset
L^{2}\left(  \mathbb{R}\right)
\]
(defined on an appropriate domain D) often associated to position ($Q)$ and momentum ($P$) of a particle. These operators satisfy the Heisenberg commutation relations
$$[Q,P]=i\textbf{1}.$$
It can be shown that $Z_{u}:=u_{1}Q+u_{2}P$, $u\in\mathbb{R}^{2}$ are
observables (called the \textit{canonical observables}). In this context we define the \textit{quantum characteristic function} as $\tilde{W}_{\rho}(u_1,u_2)=\mathrm{Tr}(\rho\exp\left(  iZ_{u}\right))$.
 If the following relation holds
\[
E_{\rho}\exp\left(  iZ_{u}\right)  =\mathrm{Tr}(\rho\exp\left(  iZ_{u}\right))=\exp\left(  iu^{T}\mu-\frac{1}%
{2}u^{T}\Sigma u\right)  \text{, }u\in\mathbb{R}^{2},
\]
 then $\rho$ is called a Gaussian state with mean $\mu$ and covariance matrix $\Sigma$.
 For such quantum Gaussian states in $L^{2}\left(  \mathbb{R}\right)  $ we
adopt a compact notation, resembling the one for the
$2$-variate normal law:
\begin{equation}
\rho=\mathbb{N}_{2}\left(  \mu,\Sigma\right)  . \label{notation-q-Gauss}%
\end{equation}
Here $\Sigma$ is a $2\times 2$ real matrix such that
$$\Sigma\geq \pm \frac{i}{2} \left(\begin{array}{cc}
    0 &  -1\\
     1& 0
\end{array}\right).$$
To define the simplest
Gaussian state, let $\psi_{0}=\sqrt{\varphi_{1/2}}$ be the square root of the
density function of the normal $N\left(  0,1/2\right)  $ distribution and
consider the operator $\rho_0$ acting by $\rho_0 f=\psi_{0}\left\langle \psi
_{0},f\right\rangle $, $\;f\in L^{2}\left(  \mathbb{R}\right)  $. Since
$\psi_{0}$ is a unit vector in $L^{2}\left(  \mathbb{R}\right)  $, the
operator $\rho_0$ (henceforth called the vacuum state) is a projection (written $\rho_0=\left\vert \psi_{0}%
\right\rangle \left\langle \psi_{0}\right\vert $ in Dirac notation) and it can be
shown that $\rho_0=\mathbb{N}_{2}\left(  0,I_2/2\right)$ in the notation described above.

An important class
is the collection of \textit{coherent states} $\mathbb{N}_{2}\left(  \mu,I_2/2\right)  $;
these are pure states which can be interpreted as a vacuum shifted by $\mu
\in\mathbb{R}^{2}$ (similar to the Gaussian shift model in classical statistics).
Consider the operators $a^*=(Q-iP)/\sqrt{2}$ (the \textit{creation operator}), $a=(Q+iP)/\sqrt{2}$ (the \textit{annihilation operator}) and $N=a^*a$ (the number operator). It is well known that the \textit{Hermite basis} $\{|0\rangle,|1\rangle\,\ldots\}$ forms an eigenbasis of the number operator,  i.e. $N|k\rangle=k|k\rangle$. For any $z \in \mathbb{C}$ define the displacement operator as
$$D(z)=\exp(za^*-\bar{z}a)$$ and the \textit{coherent state} as
\begin{equation}
|G(z)\rangle =D\left(  z\right)  \left\vert 0\right\rangle=\exp(-|z|^2/2)\sum_{k=0}^\infty \frac{z^k}{\sqrt{k!}} |k\rangle.
\label{basic-coherent-state-def}%
\end{equation}
In the density operator notation this pure Gaussian state is $|G(z)\rangle \langle G(z)|$. The
expectations of the canonical observables $Q$ and $P$ under the state
$|G(z)\rangle \langle G(z)| $ are%
\[
\left\langle Q\right\rangle =\sqrt{2}\operatorname{Re}z\text{, }\left\langle
P\right\rangle =\sqrt{2}\operatorname{Im}z
\]
and the characteristic function of $|G(z)\rangle \langle G(z)|$ is
\begin{equation*}
\varphi\left(  t\right)  =\exp\left(  i\left(  t_{1}\sqrt{2}\operatorname{Re}%
z+t_{2}\sqrt{2}\operatorname{Im}z\right)  -\frac{1}{4}\left(  t_{1}^{2}%
+t_{2}^{2}\right)  \right)  \text{, }t\in\mathbb{R}^{2}. \label{char-func-0}%
\end{equation*}
The presence of the factor $\sqrt{2}$ motivates us to adopt a modified
notation for the coherent vector: setting $\mu=\left(  \sqrt{2}%
\operatorname{Re}z,\sqrt{2}\operatorname{Im}z\right)  $, we will write
$|G(z)\rangle =\left\vert \psi_{\mu}\right\rangle $ so that now
the expectations are $\left(  \left\langle Q\right\rangle ,\left\langle
P\right\rangle \right)  =\mu$. The characteristic function $\varphi\left(  t\right)$ is that of $\mathcal{N}_2(\mu,I/2)$ and hence in the notation of (\ref{notation-q-Gauss})
\begin{equation}|G(z)\rangle \langle G(z)|=|\psi_{\mu}\rangle\langle\psi_{\mu}|=\mathbb{N}_{2}\left(  \mu,I/2\right).
\label{coherent-gaussian}
\end{equation}
Other important classes of Gaussian states are thermal states   and shifted thermal states; these are mixed states unlike the vacuum and coherent states. We define below the thermal state with temperature $\beta$ as
\[\phi_{\beta}=(1-e^{-\beta})\sum_{k=0}^{\infty}e^{-k\beta}|k\ra\la k|.\]
Shifted thermal states are defined using the shift operator $D(z)$ as follows:
\[
\phi^{z}_{\beta}=D(z)\phi_{\beta}D^*(z).
\]
One can show that the quantum characteristic function $\Tr(\phi^{z}_{\beta}\exp(iu_1Q+iu_2P))$ of the shifted thermal state is given by
\[
\Tr(\phi^{z}_{\beta}\exp(iu_1Q+iu_2P))=\exp(i(u_1\sqrt{2}Re(z)+u_2\sqrt{2}Im(z))-\frac{\sigma^2_{\beta}}{2}(u_1^2+u_2^2)),
\]
where $\sigma^2_{\beta}=\frac{\coth(\beta/2)}{2}$. Defining $\mu=(\sqrt{2}Re(z),\sqrt{2}Im(z))$,
we write that
\begin{equation}\phi_{\beta}^{z}=\bb{N}_2(\mu,\sigma^2_{\beta}I_2).\label{two_notations_thermal}
\end{equation}
To define a k-mode Gaussian state one considers the space $\bigotimes_{i=1}^kL^2(\mathbb{R})$ and identifies the number basis as follows
\begin{equation}|\mathbf{m}\ra=\otimes_{1\leq i\leq k}|m_{i}\ra,\quad \mathbf{m}=\{m_i\in \mathbb{N}:1\leq i\leq k\}.\label{number_basis}
\end{equation}
We consider the collection of operators $\{Q_1,P_1,\ldots,Q_k,P_k\}$ where each $Q_i$ and $P_i$ are position and momentum operators of a particular mode (i.e. acting on a particular $L^2(\mathbb{R})$). These operators satisfy joint commutation relations as follows:
\begin{equation}
    \label{joint_commutation}
    [Q_i,P_j]=i\delta_{i,j}\mathbf{1},\quad[Q_i,Q_j]=0,\quad[P_i,P_j]=0.
\end{equation}
One can then define the creation and annihilation operators $a_i^*$ and $a_i$ for each mode and proceed to define the displacement operator as 
$$D(\mathbf{z})=\exp(\mathbf{z}.\mathbf{a}^*-\mathbf{\bar{z}}.\mathbf{a}),$$
where $\mathbf{z}.\mathbf{a}^*=\sum_{i=1}^kz_ia^*_i$ and similarly for $\mathbf{\bar{z}}.\mathbf{a}$. Then we have 
\begin{equation}|G(\mathbf{z})\rangle =D\left(  \mathbf{z}\right)  \left\vert \mathbf{0}\right\rangle=\exp(-|\mathbf{z}|^2/2)\sum_{m_1,\ldots,m_k=0}^\infty \prod_{i=1}^k\frac{z_i^{m_i}}{\sqrt{m_i!}} |\mathbf{m}\rangle.\label{multimode coherent}\end{equation}
To describe our limiting model we need the following multimode Fock spaces:
\begin{align}
    \mathcal{F}&:= \bigotimes_{1\leq i<j \leq d}L^2(\mathbb{R})\label{full_fock}\\
    \mathcal{F}^r&:= \bigotimes_{\ijr}L^2(\mathbb{R})\label{low_fock}.
\end{align}
Recall the shifted thermal state $\bb{N}_2(\mu,\sigma^2_{\beta}I_2)$;  in the complex notations it is denoted by $\phi^{z}_{\beta}$ (see equation (\ref{two_notations_thermal})). Similarly the  shifted pure state in (\ref{coherent-gaussian}) can be written as $\phi^{z}_{\infty}$ (noting that the case $\beta=\infty$ corresponds to the pure case). The limiting model considered in \cite{Kahn&Guta} is as follows
\begin{equation}\Phi^{\theta}=\cal{N}_{d-1}(u,V_{\mu})\otimes \bigotimes_{1\leq i<j \leq d}\phi_{\beta_{ij}}^{\mathfrak{z}_{ij}},\label{Gaussian_model_full}
\end{equation}
where $\beta_{ij}=\ln(\mu_i/\mu_j)$ and $\mathfrak{z}_{ij}=\frac{\zeta_{ij}}{\sqrt{\mu_i-\mu_j}}$ i.e. the diagonal perturbation only appears in the classical part which is a $d-1$ multivariate normal experiment with mean $u$ and covariance matrix $V_{\mu}$. The latter is the covariance matrix of a multinomial random variable with probabilities $\mu_i$, while the rotation perturbations determine the $d(d-1)/2$ shifted thermal states. Since $\beta_{ij}$ are constants we will use the abbreviated notation
\begin{equation}\phi^{\mathbf{\mathfrak{z}}}=\bigotimes_{1\leq i<j \leq d}\phi_{\beta_{ij}}^{\mathfrak{z}_{ij}}\quad\in \mathcal{T}_1(\mathcal{F}),\label{tensor_thermal}\end{equation}
where $\mathbf{\mathfrak{z}}=(\mathfrak{z}_{ij})_{1\leq i<j\leq d}$ is a vector in $\bb{C}^{d(d-1)/2}$. 

We also consider the following model which is a tensor product of both thermal and pure quantum states along with a classical part which is given by a multivariate normal random variable.
\begin{equation}
\Phi^{\theta,r}=\cal{N}_{r-1}(u,V_{\mu})\otimes \bigotimes_{1\leq i<j \leq r}\phi_{\beta_{ij}}^{\mathfrak{z}_{ij}}\otimes\bigotimes_{\substack{1\leq i\leq r\\ r+1\leq j \leq d}}\phi_{\infty}^{\mathfrak{z}_{ij}}=\cal{N}_{r-1}(u,V_{\mu})\otimes\phi^{\mathfrak{z},r},\label{Gaussian_model_low}
\end{equation}
where 
\begin{equation}
    \phi^{\mathbf{\mathfrak{z},r}}= \bigotimes_{1\leq i<j \leq r}\phi_{\beta_{ij}}^{\mathfrak{z}_{ij}}\otimes\bigotimes_{\substack{1\leq i\leq r\\ r+1\leq j \leq d}}\phi_{\infty}^{\mathfrak{z}_{ij}}\quad \in \mathcal{T}_1(\mathcal{F}^r)\nonumber.\label{tensor_mixed}
\end{equation}

In Section 3 we show that the states $\Phi^{\theta,r}$ arise as limiting models for low-rank qudits.

The classical-quantum limiting state also motivates us to adopt a notation (similar to \cite{Gill&Guta}) that allows us to describe the commutation relations in the hybrid system in a compact fashion. Note that for the $k$-mode system we have defined operators $\{Q_i,P_i\}_{i=1}^k$ given by (\ref{joint_commutation}). In addition to this, we allow $l$ classical random variables $C_1,\ldots,C_l$ that commute with each other and with all $(Q_i,P_i)$. We can denote the $m=2k+l$ variables as
\begin{equation}
    \label{hybrid_notation}
    (X_1,\ldots,X_m)\equiv(C_1,\ldots,C_l,Q_1,P_1,\ldots,Q_k,P_k)
\end{equation}

and the commutation relations as
$$[X_i,X_j]=iS_{ij}\mathbf{1},$$
where $S=0_{l\times l}\oplus{\bigoplus_{i=1}^{2k}}\Omega$ and
$$\Omega=\left(\begin{array}{cc}
     0&  1\\
     -1& 0
\end{array}\right).$$
We can define a state $\varrho=f\otimes \rho$ in the space $L^1(\bb{R}^l)\otimes \mathcal{T}_1(\mathcal{F}^k)$, with $\rho$ a quantum state and $f$ a probability density. Now define the hybrid characteristic function of the state as
\begin{equation*}
    \label{hybrid_cf}
    E_{\varrho}(e^{i\sum_{j=1}^m u_jX_j}):=\int\mathrm{Tr}(\rho e^{i\sum_{j=l+1}^{2k+l}u_iX_i})f(y)e^{i\sum_{j=1}^lu_jy_j}dy_1\ldots dy_l.
\end{equation*}

\begin{definition}
A hybrid state $\varrho$ is called classical-quantum Gaussian if the characteristic function has the following form:
$$E_{\varrho}(e^{i\sum_{i=1}^m u_jX_j})=e^{iu^T\tau-u^T\Sigma u/2},$$

where $\tau\in\bb{R}^m$ and the covariance matrix $\Sigma$ is a $m\times m$ real matrix such that
$\Sigma\geq \pm \frac{i}{2}S$.
\end{definition}
Note that each shifted thermal state can be denoted by $\phi_{\beta_{ij}}^{\mathfrak{z}_{ij}}=\bb{N}_2(\nu_{ij},\sigma^2_{\beta_{ij}}I_2)$ where $$\nu_{ij}=(\sqrt{2}Re(\mathfrak{z}_{ij}),\sqrt{2}Im(\mathfrak{z}_{ij})).$$  Since the classical part in the limiting model is also a Gaussian (given by $\cal{N}_{r-1}(u,V_{\mu})$), we adopt the following alternate notation for the classical-quantum Gaussian:
\begin{equation}
\Phi^{\theta,r}=\mathfrak{N}(\tau,\mathscr{S}),
    \label{classical_quantum}
\end{equation}
where
\begin{align*}
    \Sigma&= V_{\mu}\oplus{\bigoplus_{\ijr}}
\sigma^2_{\beta_{ij}}I_{2}\\
S&=0_{r-1\times r-1}\oplus{\bigoplus_{\ijr}}\Omega\\
\mathscr{S}&=\Sigma+ \frac{i}{2}S,\quad \tau=u\oplus{\bigoplus_{\ijr}}
\nu_{ij}.\\
\end{align*}
We use $\mathscr{S}$ instead of $\Sigma$ in $\mathfrak{N}(\tau,\mathscr{S})$ that captures the underlying non-commutative structure via $S$. We note that a similar notation incorporating the commutation relations into a  complex covariance matrix has been used in \cite{YFG13}. This notation will be useful in computing a Bayes risk for a one dimensional parameter (see Appendix B) which will subsequently be used in establishing a minimax lower bound for the estimation of a linear functional of the state.

\subsection{Quantum statistical inference}

In this section we formalize the quantum counterparts of the basic notions of optimality in classical statistical inference. 
In classical statistics, an experiment is defined to be a family of probability measures on a sample space and denoted by $\mathcal{E}=\{P_{\theta},\theta\in \Theta\}$ where $\Theta$ is the parameter space.
\begin{definition}
A quantum statistical model over a parameter space $\Theta$ consists of a
family of quantum states $\mathcal{Q} = \{\rho_\theta :\, \theta\in \Theta \}
$ on a Hilbert space $\mathcal{H}$, indexed by an unknown parameter $%
\theta\in \Theta$.
\end{definition}
Inference in quantum models generally involves two steps.
In the first step one performs a measurement on the state $\rho_\theta$ and generates data, while in the second step, one uses standard statistical tools to solve the specific decision problem using data from the first step.
If one performs a measurement $M$
on the system in state $\rho_\theta$, a random outcome is obtained with distribution $P^M_\theta(E) := \mathrm{Tr}(\rho_\theta
M(E))$ (cf. Subsection 2.1). The measurement data is therefore
described by the classical model $\mathcal{P}^M := \{ P^M_\theta
:\, \theta\in \Theta\}$ and the estimation problem can be treated using
``classical'' statistical methods. However, in many scenarios, the optimal estimators for individual components of a parameter are incompatible with each other and the optimal joint estimator for the two components can be entirely different from the optimal estimators of the individual components.

In the classical setup, a randomized decision function is given by a Markov kernel $\nu$. If $L(\theta,u)$ is the loss function then the risk is given by
\begin{equation}
 R(\theta,\nu)=\int\int L(\theta,u)\nu_x(du)\mu_{\theta}(dx)=\int L(\theta,u)\int\nu_x(du)\mu_{\theta}(dx)=\int L(\theta,u)\tilde{\nu}_{\theta}(du), 
 \label{class-risk}
\end{equation}
where $\tilde{\nu}_{\theta}(A)=\int\nu_x(A)\mu_{\theta}(dx)$.

Section 2.2.4 of  \cite{Hol-01} discusses the quantum counterpart of this classical formulation. Let $\rho_{\theta}$ be the quantum state and $\mu^M_{\theta}(B)=Tr(\rho_{\theta}M(B))$ be the probability measure generated by the POVM $M$. Then the risk is given by
\begin{equation*}
R(\theta,M)=\int L(\theta,u)\mu^M_{\theta}(du).
\label{operator-risk}
\end{equation*}
By using the fact that every affine map $\rho_{\theta}\rightarrow \mu_\theta()$ can be associated with a POVM, we see that $M$ is an analog of the classical randomized decision function $\nu$ given in (\ref{class-risk}).
We can easily define the Bayes and minimax problems for quantum estimation.

\textbf{Minimax problem} 
$$\inf_M\sup_{\theta\in \Theta}R(\theta,M)=\inf_M\sup_{\theta\in \Theta}\int L(\theta,u)\mu^M_{\theta}(du)=\inf_{\hat{m}}\sup_{\theta \in \Theta}E_{\theta}[L(\theta,\hat{m})]$$

\textbf{ Bayes problem} 
$$\inf_M\int_{\Theta}R(\theta,M)\pi(d\theta)=\inf_M\int_{\Theta}\int L(\theta,u)\mu^M_{\theta}(du)\pi(d\theta)=\inf_{\hat{m}}\int E_{\theta}[L(\theta,\hat{m})]\pi(d\theta).$$
The notations $\inf_{\hat{m}}\sup_{\theta \in \Theta}E_{\theta}[L(\theta,\hat{m})]$ and $\inf_{\hat{m}}\int E_{\theta}[L(\theta,\hat{m})]\pi(d\theta)$ will be called \textit{condensed notations} and will be used henceforth. Note that the infimum is over all POVM and the notation $\hat{m}$ should not be confused with a deterministic estimator seen in the classical setup. We will also denote the Bayes risk as $\inf_{\hat{m}} E[L(\theta,\hat{m})]$ where the expectation is also taken over the parameter $\theta$.

In classical statistics, a well-known paradigm is using asymptotic equivalence of experiments to transfer risk bounds from one experiment to another. Suppose we have two experiments $\mathcal{E}=\{P_{\theta},\theta\in \Theta\}$
on a  sample space $(\Omega_1,\mathcal{A}_1)$ and   $\mathcal{F}=\{Q_{\theta},\theta\in \Theta\}$ on a  sample space $(\Omega_2,\mathcal{A}_2)$. Also let the loss function satisfy the condition $0 \leq L(\theta,u) \leq 1$. If there exists a Markov kernel $K$ such that
$$\sup_{\theta \in {\Theta}}||KP_{\theta}-Q_{\theta}||_{TV}\leq \epsilon,$$
then for any randomized decision function $\mu$ of $\theta$ in the model $Q_{\theta}$, the randomized decision function $\nu=\mu \circ K$ (composition of two Markov kernels) satisfies
$$R^1({\theta,\nu})\leq R^2(\theta,\mu)+\epsilon,$$
where $R^1({\theta,\nu})$ and $R^2(\theta,\mu)$ are the risks in the models $\mathcal{E}$ and $\mathcal{F}$ respectively. We discuss the generalization of this paradigm to the quantum setup and also generalize it to the case of unbounded loss. 

The quantum equivalent of a Markov kernel is the transformation by quantum channels. The quantum model $\mathcal{Q}$ can be transformed into
another quantum model $\mathcal{Q}^\prime := \{ \rho_\theta^\prime :\theta
\in \Theta \}$ on a Hilbert space $\mathcal{H}^\prime$ by applying a quantum channel
\begin{eqnarray*}
T&:& \mathcal{T}_1(\mathcal{H}) \to \mathcal{T}_1(\mathcal{H}^\prime) \\
T&:& \rho_\theta \mapsto \rho_\theta^\prime.
\end{eqnarray*}
In this context, we define the \textit{quantum Le Cam distance} between two models from \cite{BGN-QAE}.
\begin{definition}
Let $\mathcal{Q}$ and $\mathcal{Q}^\prime$ be two quantum models over $\Theta
$. The deficiency of $\mathcal{Q}$ with respect to $\mathcal{Q}^\prime$ is defined
by
\begin{equation*}
\delta \left(\mathcal{Q},\mathcal{Q}^\prime\right) := \inf_T
\sup_{\theta\in\Theta} \| T(\rho_\theta) - \rho_\theta^\prime\|_1,
\end{equation*}
where the infimum is taken over all channels $T$. The \emph{Le Cam} distance
between $\mathcal{Q}$ and $\mathcal{Q}^\prime$ is defined as
\begin{equation*}  \label{eq.q.LeCam}
\Delta \left(\mathcal{Q},\mathcal{Q}^\prime\right):= \mathrm{max}\left(
\delta \left(\mathcal{Q},\mathcal{Q}^\prime\right)\,, \, \delta \left(%
\mathcal{Q}^\prime,\mathcal{Q}\right)\right).
\end{equation*}
\end{definition}
Its interpretation is that models which are ``close'' in the Le Cam distance
have similar risk bounds. Suppose we have two sequences of quantum models (or experiments) $\mathcal{E}^{(n)}=\{\rho^{(1,n)}_{\theta}:\theta\in \Theta\}$  and $\mathcal{F}^{(n)}=\{\rho^{(2,n)}_{\theta}:\theta\in \Theta\}$ with associated sequences of Hilbert spaces $\mathcal{H}^{1,n}$ and $\mathcal{H}^{2,n}$. Assume that $\Delta \left(\mathcal{E}^{(n)},\mathcal{F}^{(n)}\right)\rightarrow 0$; this implies $\delta \left(\mathcal{E}^{(n)},\mathcal{F}^{(n)}\right)\rightarrow 0$ and in particular there exists a sequence of quantum channels $T_n$, such that
$$||T_n(\rho^{(1,n)}_{\theta})-\rho^{(2,n)}_{\theta}||_1= o(1).$$
Let the loss function also change with $n$ and satisfy the relation $0\leq L_n(\theta,u) \leq c_n$. Also, assume that the sequence of quantum channels $T_n$ is such that
\begin{equation}c_n\sup_{\theta \in \Theta}||T_n(\rho^{(1,n)}_{\theta})-\rho^{(2,n)}_{\theta}||_1= o(1).
\label{equivalence-rate}
\end{equation}

Recall the dual map $T^*$ of a quantum channel $T$. It follows that for any $M\in \mathit{POVM}(\Omega,\mathcal{H}^{2,n})$
\begin{align}
    R^1_n(\theta,T_n^*(M))= &\int L_n(\theta,u)\mathrm{Tr}(\rho_{\theta}^{(1,n)}T_n^*(M(du)))\nonumber\\
     = &\int L_n(\theta,u)\mathrm{Tr}(\rho_{\theta}^{(2,n)}M(du))\nonumber\\
    &+\int L_n(\theta,u)[\mathrm{Tr}(\rho_{\theta}^{(1,n)}T_n^*(M(du)))-\mathrm{Tr}(\rho_{\theta}^{(2,n)}M(du))]\nonumber\\
     = & R^2_n(\theta,M)+\int L_n(\theta,u)[\mathrm{Tr}(T_n(\rho_{\theta}^{(1,n)})(M(du)))-\mathrm{Tr}(\rho_{\theta}^{(2,n)}M(du))]\nonumber\\
    \leq & R^2_n(\theta,M)+c_n||T_n(\rho^{(1,n)}_{\theta})-\rho^{(2,n)}_{\theta}||_1\nonumber\\
    \leq & R^2_n(\theta,M)+o(1),\label{risk-transfer}
\end{align}
the term $o(1)$ tending to $0$ uniformly over all $\theta$. Thus we can compare the risks of the two models $\mathcal{E}^{(n)}$ and $\mathcal{F}^{(n)}$ if (\ref{equivalence-rate}) holds. Note that we have similar relations for minimax risks and Bayes risks, by taking a supremum over $\Theta$ or integrating with respect to a prior, respectively, and then taking an infimum over all estimators:
\begin{align}
    \inf_M\sup_{\theta\in \Theta}R^1_n(\theta,M) & \leq\inf_M\sup_{\theta\in \Theta}R^2_n(\theta,M)+o(1).\label{minimax-risk-transfer}\\
    \inf_M\int_{\Theta}R^1_n(\theta,M)\pi(d\theta) & \leq \inf_M\int_{\Theta}R^2_n(\theta,M)\pi(d\theta)+o(1).\nonumber \label{Bayes-risk-https://www.overleaf.com/project/5ea3798bc4096c0001f3b837transfer}
\end{align}

\section{Local asymptotic normality in low-rank systems}
\subsection{Classical LAN}
Consider a collection of i.i.d random variables $\{X_{1},\ldots,X_{n}\}$
taking values in a measurable space $(\mathcal{X},\Sigma_{\mathcal{X}})$ with
$X_{i}\sim P_{\theta}$ where $\theta$ belongs to $\Theta$, which is an open
subset of $\mathbb{R}^{d}$. We can consider a local perturbation around a
fixed point $\theta_{0}$ and if we denote $\theta=\theta_{0}+u/\sqrt{n}$ (with
$u$ bounded), then we can represent the aforementioned collection of random
variables by a statistical experiment $\mathcal{E}_{n}=\{P_{\theta_{0}%
+u/\sqrt{n}}^{n},||u||\leq C\}$ on a sample space $(\mathcal{X}^{n}%
,\Sigma_{\mathcal{X}}^{n})$ where $P_{\theta_{0}+u/\sqrt{n}}^{n}$ is an
$n$-fold product of $P_{\theta_{0}+u/\sqrt{n}}$. Under some regularity
assumptions (see theorem below) $\mathcal{E}_{n}$ can be approximated by a
Gaussian shift experiment $\mathcal{F}=\{N(u,I_{\theta_{0}}^{-1}),||u||\leq
C\}$ where $I_{\theta_{0}}$ is the Fisher information matrix at $\theta_{0}$.
The following result is well known (cf. \cite{MR1791434}, Theorem
2.9 along with \cite{Strasser}, 79.3).

\begin{theorem}
{\label{classical_LAN}} Assume $\left(  \mathcal{X},\Sigma_{\mathcal{X}%
}\right)  $ is a Polish (complete separable metric) space with its Borel
$\sigma$-algebra. Assume further that\medskip\newline(i) the experiment
$\mathcal{E}_{n}$ is dominated: $P_{\theta}\ll\mu$, $\theta\in\Theta$ where
$\mu$ is a $\sigma$-finite measure on $\Sigma_{\mathcal{X}}$, \newline(ii) the
densities $p_{\theta}\left(  x\right)  =\left(  dP_{\theta}/d\mu\right)  (x)$
are jointly measurable in $\left(  x,\theta\right)  $ and differentiable in
quadratic mean at $\theta=\theta_{0}$, i.e. for some measurable function
$\ell_{\theta}:\mathcal{X}\rightarrow\mathbb{R}^{d}$
\[
\int\left[  p_{\theta+u}^{1/2}-p_{\theta}^{1/2}-u^{T}%
\ell_{\theta}p_{\theta}^{1/2}\right]  ^{2}d\mu=o\left(  \left\Vert
u\right\Vert ^{2}\right)  \text{ as }u\rightarrow0,
\]
(iii) the Fisher information matrix $I_{\theta}=4E_{\theta}[\ell_{\theta}%
\ell_{\theta}^{T}]$ is nonsingular at $\theta=\theta_{0}$. \medskip\newline
Then the experiments $\mathcal{E}_{n}$ and $\mathcal{F}$ are asymptotically equivalent.
\end{theorem}

In other words, there exist sequences of Markov kernels $T_n$ and $S_n$, such that:
\begin{align*}
    &\lim_{n\rightarrow \infty}\sup_{||u||\leq C}||T_n(P^n_{\theta_0+u/\sqrt{n}})-N(u,I^{-1}_{\theta_0})||_{TV}=0\\
    &\lim_{n\rightarrow \infty}\sup_{||u||\leq C}||P^n_{\theta_0+u/\sqrt{n}}-S_n(N(u,I^{-1}_{\theta_0}))||_{TV}=0.\end{align*}

\subsection{Quantum LAN}
Consider the following domain of the local parameters
$$\Theta_{n,\beta,\gamma}=\{(u,\mathfrak{z}):|u_k|\leq n^{\gamma},|\mathfrak{z}_{ij}|\leq n^{\beta}, \ \forall 1\leq k\leq d-1,\ 1\leq i<j\leq d\},$$

for some $\beta>0,\gamma>0$.

Recall the $d$-dimensional state $\rho_{\theta}$ given in equation (\ref{local_qudit_full}) and the corresponding Gaussian state given in (\ref{Gaussian_model_full}) indexed by the same local parameter $\theta$. Note that we have used the alternate notation using $\theta$ instead of $\vartheta$ (see the discussion after (\ref{local_qudit_r})). We consider the following two experiments:
$$\mathcal{Q}_n=\{\rho^{\theta,n}:\theta\in \Theta_{n,\beta,\gamma}\},\quad \mathcal{R}_n=\{\Phi^{\theta}:\theta\in \Theta_{n,\beta,\gamma}\}.$$
We state the LAN theorem for full rank states (proved in \cite{Kahn&Guta}) which shows that these local models are asymptotically equivalent.
\begin{theorem}{\label{full_rank_LAN}}
Recall the Fock space described in (\ref{full_fock}). Then for $0<\gamma<1/4$ and $0<\beta<1/9$ there exist quantum channels $T_n$ and $S_n$
\begin{align*}
    &T_n:M(\mathbb{C}^d)^{\otimes n}\rightarrow L^1(\mathbb{R}^{d-1})\otimes \mathcal{T}_1(\mathcal{F})\\
    &S_n:L^1(\mathbb{R}^{d-1})\otimes \mathcal{T}_1(\mathcal{F})\rightarrow M(\mathbb{C}^d)^{\otimes n} \\
\end{align*}
such that 
\begin{align*}
    \sup_{\theta\in \Theta_{n,\beta,\gamma}}||\Phi^{\theta}-T_n(\rho^{\theta,n})||_1=O(n^{-\kappa})\\
   \sup_{\theta\in \Theta_{n,\beta,\gamma}}||S_n(\Phi^{\theta})-\rho^{\theta,n}||_1=O(n^{-\kappa}),\\
\end{align*}   
  for some $\kappa>0$ which depends on $\beta$ and $\gamma$.

\end{theorem}
Now we are ready to state the low-rank version of the above theorem. Recall that a qudit of rank $r$ can be parametrized as $\rho_{\theta,r}$ (see equation (\ref{local_qudit_r}) and use the alternate notation, i.e. $\theta$ instead of $\vartheta$) and consider the corresponding Gaussian state given in (\ref{Gaussian_model_low}) indexed by a local parameter $\theta$. Define 
$$\Theta_{n,r,\beta,\gamma}=\{(u,\mathfrak{z}):|u_{k}|\leq n^{\gamma},|\mathfrak{z}_{ij}|\leq n^{\beta} \ \forall \ 1\leq k\leq r-1,\ijr\},$$

for some $\beta>0,\gamma>0$.
We now have the low-rank versions of the earlier models:
$$\mathcal{Q}^r_n=\{\rho^{\theta,r,n}:\theta\in \Theta_{n,r,\beta,\gamma}\}\quad \mathcal{R}^r_n=\{\Phi^{\theta,r}:\theta\in \Theta_{n,r,\beta,\gamma}\}.$$
The following theorem shows that the above models are asymptotically equivalent.

\begin{theorem}{\label{channels_for_low_rank}}
Recall the Fock space described in (\ref{low_fock}). Then for $0<\gamma<1/4$ and $0<\beta<1/9$ there exist quantum channels $T^r_n$ and $S^r_n$
\begin{align*}
    &T^r_n:M(\mathbb{C}^d)^{\otimes n}\rightarrow L^1(\mathbb{R}^{r-1})\otimes \mathcal{T}_1(\mathcal{F}^r)\\
    &S^r_n:L^1(\mathbb{R}^{r-1})\otimes \mathcal{T}_1(\mathcal{F}^r)\rightarrow M(\mathbb{C}^d)^{\otimes n} \\
\end{align*}
such that 
\begin{equation}
    \sup_{\theta\in \Theta_{n,r,\beta,\gamma}}||\Phi^{\theta,r}-T^r_n(\rho^{\theta,r,n})||_1=O(n^{-\kappa})\label{forward_channel}
    \end{equation}
    \begin{equation}
   \sup_{\theta\in \Theta_{n,r,\beta,\gamma}}||S^r_n(\Phi^{\theta,r})-\rho^{\theta,r,n}||_1=O(n^{-\kappa}),\label{reverse_channel}
   \end{equation}
   
  for some $\kappa>0$ which depends on $\beta$ and $\gamma$.
\end{theorem}

We observe that the classical part corresponds to a low dimensional ($r-1$- variate) normal experiment while the quantum part contains two subparts. When $\mu_j$ and $\mu_k$ are both positive (i.e. $1\leq j<k \leq r$), we get shifted thermal states with temperatures given as before ($\beta_{jk}=\ln(\mu_j/\mu_k)$). When $\mu_j>0$ and $\mu_k=0$, we get shifted pure states. We can compare our result with other available results in quantum LAN. 

\textit{Comparison with other LAN results}
\begin{enumerate}
    \item The diagonal case: In the absence of the rotation by unitaries the diagonal state represents a local multinomial model and the limiting model contains only the classical Gaussian part, i.e. $\mathcal{N}_{r-1}(u,V_{\mu})$, and we recover the classical LAN result of Theorem \ref{classical_LAN}. A global version of this approximation (with possibly increasing dimension of the multinomial) was obtained in \cite{MR1922539}.\\ 
    \item $d=2,r=1$: In this case, we observe that the limiting model consists of a single shifted pure Gaussian state and no classical component. This case was discussed in  \cite{Guta_Kahn_qubit} using a heuristic argument.\\
    \item $d=r>2$: The limiting model consists of a $d-1$ dimensional normal distribution in the classical part and only shifted thermal states in the quantum part. We recover the full rank case of \cite{Kahn&Guta}, i.e. Theorem \ref{full_rank_LAN}.\\
    \item $d>r=1$: The limiting model consists of $d-1$ shifted pure states and no classical component. The case when $d=\infty$ and $r=1$,  i.e. the case of infinite dimensional pure states was treated in \cite{BGN-QAE} under a slightly different parametrization. The limiting model in \cite{BGN-QAE} is quantum white noise, or equivalently a tensor product of infinitely many shifted pure states. A slight modification of the proof of Theorem 4.1 in \cite{BGN-QAE-SUPP} (proving it for finite dimensional pure states) shows that the limiting model is indeed a tensor product of $d-1$ shifted pure states and agrees with our current result.\\
\end{enumerate} 
\section{Measurement of the shift parameter for a Gaussian state}

In this section, we consider measurement of the shift parameter $\mu$ of the model $\rho=\mathbb{N}_2(\mu,\sigma^2I_2)$. We describe the particular generalized measurement (the \textit{covariant measurement}) that is used to measure the shift parameter. It can be shown that a Bayes estimator can be constructed by appropriately ``shrinking'' the covariant measurement.

An essential fact is that the coherent vectors $\left\{  \left\vert \psi
_{m}\right\rangle ,m\in\mathbb{R}^{2}\right\}  $ form an ``overcomplete system''
(if multiplied by a factor $1/\sqrt{2\pi}$), i.e. fulfill
\begin{equation}
\frac{1}{2\pi}\int_{\mathbb{R}^{2}}\left\vert \psi_{m}\right\rangle
\left\langle \psi_{m}\right\vert dm=\mathbf{1}, \label{completeness-relation}%
\end{equation}
where $\mathbf{1}$ is the identity operator in the Hilbert space $\mathcal{H}%
=L^{2}\left(  \mathbb{R}\right)  $ (see equation 3.5.45, p. 101 of  \cite{Hol-2nd-ed-11}, 
with proof after Proposition 3.5.1). A complete orthonormal system $\left\{  \left\vert
\psi_{m}\right\rangle \right\}  $ is an example of an overcomplete system (with integration replaced by summation);
however, in the general case the vectors $\psi_{m}$ can be non-orthogonal and
linearly dependent. The system of coherent vectors $\left\{  \left\vert
\psi_{m}\right\rangle ,m\in\mathbb{R}^{2}\right\}  $ generates a
\textit{resolution of the identity}, i.e. a normalized POVM\ $M$ on the
measurable space $\left(  \mathbb{R}^{2},\mathfrak{B}\right)  $ (with
$\mathfrak{B}$ being the Borel sigma-algebra) according to
\begin{equation*}
M\left(  B\right)  =\frac{1}{2\pi}\int_{B}\left\vert \psi_{m}\right\rangle
\left\langle \psi_{m}\right\vert dm\text{, }B\in\mathfrak{B}\text{.}
\label{basic-povm}%
\end{equation*}
The POVM\ $M$ then generates a (generalized) observable $X_{M}$ with values in
$\mathbb{R}^{2}$ which under the state $\rho$ has probability distribution
\begin{equation*}
P\left(  X_{M}\in B\right)  =\mathrm{Tr\;}\rho M\left(  B\right)  \text{,
}B\in\mathfrak{B}\text{.} \label{X_M-def}%
\end{equation*}
This is called the \textit{canonical covariant measurement} in Section 3.6 of
\cite{Hol-2nd-ed-11}, covariance referring to the action of the Weyl unitaries
(or the displacement operators). There also optimality properties are proved,
as well as equivalence to simple measurements on an extended system
(corresponding to an \textit{orthogonal} resolution of the identity there). When the covariant measurement is clear from the context we will often write $X_M$ as $X$. It can be shown that when $\rho$ is a shifted pure state, i.e. $\rho=\mathbb{N}_2\left(\mu,I_2/2\right)$ then $X\sim \mathcal{N}_2(\mu,I_2)$ and if $\rho$ is a shifted thermal state, i.e. $\rho=\mathbb{N}_2\left(\mu,\sigma^2_{\beta}I_2\right)$ then $X\sim \mathcal{N}_2\left(\mu,\frac{(2\sigma_{\beta}^2
+1)}{2}I_2\right)$.

An equivalent description can be given  as follows; cf.  Proposition 3.6.1 of \cite{Hol-2nd-ed-11} and also relation (3.18) in \cite{Hol-78}.  Let $\mathcal{H}%
=L^{2}\left(  \mathbb{R}\right)  $ and let $\mathcal{H}_{0}$ be an
identical Hilbert space with canonical observables $Q_{0}$ and $P_{0}$. In the
tensor product $\mathcal{H\otimes H}_{0}$, consider the operators
\begin{equation*}
\tilde{Q}=Q\otimes \mathbf{1}_{0}+\mathbf{1}\otimes Q_{0},\;\tilde{P}=P\otimes \mathbf{1}_{0}-\mathbf{1}\otimes
P_{0}, \label{operators-extended-system}%
\end{equation*}
where $\mathbf{1}_0$ is the identity operator in $\mathcal{H}_{0}$. Let $\rho$ be the
state in $\mathcal{H}$ to be measured and $\rho_{0}$ be an auxiliary
state in $\mathcal{H}_{0}$ to be chosen; then a simple measurement of
$\rho\otimes\rho_{0}$ can be understood as a ``randomized'' measurement of
$\rho$. These randomized measurements correspond to nonorthogonal resolutions
like (\ref{completeness-relation}).

It can be shown that $\tilde{Q}$ and $\tilde{P}$ commute, which means that the observables $\tilde{Q}$, $\tilde{P}$ are jointly
measurable in the system given by $\mathcal{H\otimes H}_{0}$. The operators
$\tilde{Q}$, $\tilde{P}$ are self-adjoint and thus generate jointly
distributed real valued random variables $X_{\tilde{Q}}$,
$X_{\tilde{P}}$. We define $\tilde{X}$ as follows 
\begin{equation}
\tilde{X}=\left(  X_{\tilde{Q}},X_{\tilde{P}}\right)  \label{X-tilde-def}%
\end{equation}
and it can be checked that if $\rho=\left\vert \psi_{\mu}\right\rangle
\left\langle \psi_{\mu}\right\vert $ and the auxiliary state 
$\rho_{0}$ is the vacuum $\rho_{0}=\left\vert \psi_{0}\right\rangle
\left\langle \psi_{0}\right\vert $ then the distribution of $\tilde{X}$ coincides with the distribution of $X$ (obtained with the covariant measurement), i.e. with $\mathcal{N}_2\left(  \mu,I_{2}\right)$. Similarly when  $\rho=\bb{N}_2(\mu,\sigma^2_{\beta}I_2)$ then the distribution of $\tilde{X}$ coincides with $\mathcal{N}_2\left(\mu,\frac{(2\sigma_{\beta}^2
+1)}{2}I_2\right)$.
 
Next, for the state $\bb{N}_2(\mu,\sigma^2I_2)$, we consider the problem of Bayes estimation with quadratic loss of the parameter
$\mu\in\mathbb{R}^{2}$ under a normal prior $\mu\sim \mathcal{N}_2\left(  0,\sigma
_{0}^{2}I_{2}\right)  $.The solution for a quadratic risk is  given in \cite{Hol-01}, p. 55, with
details and proofs in \cite{Hol-78}. Consider the
loss function
\[
L\left(  \hat{\mu},\mu\right)  =\left(  \hat{\mu}_{1}-\mu_{1}\right)
^{2}+\left(  \hat{\mu}_{2}-\mu_{2}\right)  ^{2}.
\]
The observable which is optimal in the Bayes sense (along with the equivalent POVM) is discussed in \cite{Holevo1973}, Section 3. By Proposition 3 there, the optimal POVM is given by (denoting $m=(x,y)$)
\[
M_{c}\left(  B\right)  =\frac{1}{2\pi c^{2}}\int_{B}\left\vert
\psi_{m/c}\right\rangle \left\langle \psi_{m/c}\right\vert
dm\text{, }B\in\mathfrak{B}\text{.}%
\]
From (\ref{completeness-relation}) and a change of variables, it can be easily verified that
\[
\frac{1}{2\pi c^{2}}\int_{\mathbb{R}^2}\left\vert
\psi_{m/c}\right\rangle \left\langle \psi_{m/c}\right\vert
dm=\mathbf{1}
\]
and hence $M_c$ is a resolution of identity (the other properties of a POVM are also easy to check). Here $c=2\sigma_0^2/(2\sigma_0^2+2\sigma^2+1)$ and the Bayes risk is given by (writing $\sigma^2$ as $\sigma^2_{\beta}$)
\begin{equation}
\inf_{M} R(M,\pi)=\frac{2\sigma_0^2(2\sigma^2_{\beta}+1)}{2(\sigma_0^2+\sigma^2_{\beta})+1}.
\label{Bayes-risk_thermal}
\end{equation}
When $\sigma^2=1/2$ (i.e. the state is pure), we have  $c=\frac{\sigma_0^2}{\sigma_0^2+1}$ and the Bayes risk according to (17) in \cite{Holevo1973} thus becomes
\begin{equation}
\inf_{M} R(M,\pi)=\frac{2\sigma_0^2}{\sigma_0^2+1}.
\label{Bayes-risk-pure}
\end{equation}
The randomized measurements are then given by
\begin{equation}
\tilde{Q}_c=c(Q\otimes \mathbf{1}_0+\mathbf{1}\otimes Q_0), \quad \tilde{P}_c=c(P\otimes \mathbf{1}_0-\mathbf{1}\otimes P_0),
\label{generalized-shrinkage}
\end{equation}
cf. equation (12) of \cite{Holevo1973}; it can easily be shown that these randomized measurements and $M_c$ generate the same random variables.

From (\ref{generalized-shrinkage}) it is clear that
$\tilde{X}_c=(X_{\tilde{Q}_c},X_{\tilde{P}_c})$ satisfies $\tilde{X}_c=c\tilde{X}$ where $\tilde{X}$ is given in (\ref{X-tilde-def}). As $c<1$, we note that Bayes estimation in the quantum case exhibits the same shrinkage phenomenon as witnessed in the classical counterpart. 

However we note that the risk in (\ref{Bayes-risk-pure}) is greater than the Bayes risk in the classical model $\mathcal{N}_2(\mu,I_2/2)$ with the same prior, i.e. we obtain $2\sigma_0^2/(\sigma^2_0+1)$ in the quantum case compared to $2\sigma_0^2/(2\sigma^2_0+1)$ in the classical case. A similar statement holds for the thermal state as well. This inflation of risk is essentially due to the quantum nature of the data, in particular to the non-commutativity of the observables $Q$ and $P$ (as a consequence of which the random variables generated by $Q$ and $P$ separately do not admit a joint distribution).

\section{Optimal estimation of the low-rank state}

In this section, we outline a two-stage procedure to construct an estimator which is asymptotically minimax optimal. To construct the estimator, we split the sample into two parts. In Subsection 5.1 we use the first part (which has sample size $\lfloor n^{\delta}\rfloor$) to obtain a preliminary estimator of $\rho$ which lies within a ball of shrinking radius, centered at $\rho$, with high probability. Then we treat the initial estimator $\tilde{\rho}_n$  as the central state and parametrize the original state $\rho$ as  $\rho_{\theta}$, so that it suffices to estimate $\theta$. Now we use Theorem 3.3 to approximate the i.i.d. model by a limiting Gaussian model using the remaining part of the sample of size $n-\lfloor n^{\delta}\rfloor$. The risk in the limiting Gaussian model is computed in Section 5.2. In Subsection 5.3, we use the channel $T_n^r$ and the risk transfer mechanism (\ref{risk-transfer}) to give an upper bound to the  risk in the original i.i.d model. To show that the constructed estimator is asymptotically minimax optimal we choose an arbitrary state $\sigma$ (which lies in ball of shrinking radius around $\rho$) as the central state and parametrize the original state $\rho$ with a local parameter $\theta$. Again, the i.i.d model is approximated by the limiting Gaussian model and using the channel $S_n^r$ and equation (\ref{risk-transfer}), we obtain a lower bound for the minimax risk in the i.i.d model. We show that the upper and lower bounds match, ensuring that the constructed estimator is asymptotically minimax optimal.

For notational convenience, we work with $n$ instead of $\lfloor n^{\delta}\rfloor$ while showing the properties of the preliminary estimator and replace $n$ with $\lfloor n^{\delta}\rfloor$ in the final computations.

\subsection{Preliminary estimator}

Rank penalized and ``physical'' estimators have been constructed in \cite{ButuceaGutaKypraios} in the context of estimation of the joint quantum state of k two-dimensional systems (qubits), i.e. in a special case of our setting where $d=2^k$. The authors obtain a least square estimator constructed using Pauli observables and then use rank penalization and spectral thresholding to obtain rank penalized and physical estimators, respectively. Not only do these estimators lie within a ball of radius $n^{-1/2+\epsilon}$ (for $\epsilon<1/2$) centered at the original state with high probability; they are also rank consistent in the sense that the ranks of estimated states match the rank of the original state with high probability. 

The general $d$ dimensional case has been considered in \cite{MR4093475}, where the authors construct a least squares estimator using a so-called uniform POVM and project it into the space of density matrices to obtain an estimator which lies within a ball of radius $n^{-1/2+\epsilon}$ (for arbitrary small $\epsilon$) centered at the original state with high probability. Since the authors are interested in the error of estimation in the trace norm, the rank consistency is not discussed in the paper. We show the least squares estimator of \cite{MR4093475}, when processed suitably like the physical estimator in  \cite{ButuceaGutaKypraios}, enjoys the rank consistency property. We note that in order to use LAN we need a local parametrization around a preliminary estimate of the state; it is essential that its rank matches that of the original state  and hence rank consistency is crucial. The following theorem (proved in Appendix C) summarizes the above discussion.

\begin{theorem}{\label{conc_prelim_est}}
Let $\rho$ be a rank $r$ state in $\bb{C}^d$ with the minimum eigenvalue $\lambda_r>6\epsilon$. Then there exists an estimator $\tilde{\rho}_{n}$ of $\rho$ such that the following concentration inequality holds:
\[P[||\rho-\tilde{\rho}_{n}||^2_2\geq 25r\epsilon^2]\leq de^{-3n\epsilon^2/16d}.\]
Further, if $\hat{r}$ is the rank of the state $\tilde{\rho}_{n}$, then 
\[P[r=\hat{r}]\geq 1- de^{-3n\epsilon^2/16d}.\]
\end{theorem}

\subsection{Optimal risk in the Gaussian model}
Here we consider minimax estimation in the limiting Gaussian model given in (\ref{Gaussian_model_low}).
Recall that the local parameter $\theta$ has two components: $u$  corresponds to the classical part and $\mathfrak{z}$ corresponds to the quantum part.
First, we define the loss function with respect to which the risk will be computed:
\begin{equation}
    \mathcal{L}(\theta,\hat{\theta})=\sum_{i=1}^{r-1}(u_i-\hat{u}_i)^2+(\sum_{i=1}^{r-1}(u_i-\hat{u}_i))^2+2\sum_{\substack{{1\leq i\leq r}\\{i<j\leq d}}}(\mu_i-\mu_j)|\mathfrak{z}_{ij}-\hat{\mathfrak{z}}_{ij}|^2.\label{loss_theta}
\end{equation}
 Although the particular form of loss may seem arbitrary, we will see that the $L^2$ distance between two density matrices is approximately locally quadratic and then given by this loss. The following theorem gives the asymptotic minimax bound for estimation of $\theta$:
\begin{theorem}{\label{optimal_risk_Gaussian}}
Let $\Phi^{\theta,r}$ be the limiting Gaussian model given in (\ref{Gaussian_model_low}), then 
\begin{align*}\lim_{n\rightarrow\infty}\inf_{\hat{\theta}}\sup_{\theta\in\Theta_{n,r,\beta,\gamma}}E_{\theta}[\mathcal{L}(\theta,\hat{\theta})]&=\lim_{n\rightarrow\infty}\inf_{M}\sup_{\theta\in\Theta_{n,r,\beta,\gamma}}\int\mathcal{L}(\theta,\hat{\theta})\mathrm{Tr}(\Phi^{\theta,r}M(d\hat{\theta}))\\&=\sum_{i=1}^r \mu_i(1-\mu_i)+\sum_{\substack{{1\leq i\leq r}\\{i<j\leq d}}}2\mu_i.\end{align*}
\end{theorem}
Note that the limiting model is a tensor product of classical and quantum shifted Gaussian states. To show the upper bound we use a classical estimator and a POVM for the 
classical and quantum parts respectively. For the classical part, we use $Z\sim\cal{N}_{r-1}(u,V_{\mu})$, as the estimator $\hat{u}$. For the quantum part, we note that each component of the tensored quantum Gaussian state is either a shifted thermal state or a shifted pure state indexed by $(i,j)$. We will use the covariant measurement $M$ (see Section 4) on each component and hence the joint POVM will be given by $\bar{M}=\bigotimes_{\substack{{1\leq i\leq r}\\{i<j\leq d}}}M$.

We use a Bayes risk to give a lower bound for the minimax risk, by setting up Gaussian priors for the parameters in both the classical and quantum parts. The Bayes estimator for the classical part is well known and is given by a shrinkage of the Gaussian random variable. For the quantum part, we directly compute the Bayes risk using equations (\ref{Bayes-risk_thermal}) or (\ref{Bayes-risk-pure}) depending on whether the component is a thermal or pure Gaussian state. We again note that the optimal POVMs for each component are obtained by appropriately shrinking covariant measurements. We defer the proof to the Appendix.
\subsection{Localization and the optimal risk}
Finally, we are ready to compute the exact asymptotics of the minimax risk in the low-rank qudit model. First, we state a lemma (proved in Appendix D) that shows that the $L^2$ distance between two states is approximately a quadratic loss in the local parameters.
\begin{lemma}{\label{local_quadratic_risk}}
Let $\rho_{\theta_1,n}$ and $\rho_{\theta_2,n}$ be two states indexed by the local parameters $\theta_1$ and $\theta_2$, where $\theta_1,\theta_2 \in \Theta$, then
\begin{align*}||\rho_{\theta^{(1)}/\sqrt{n},r}-\rho_{\theta^{(2)}/\sqrt{n},r}||^2_2=&\frac{1}{n}\sum_{i=1}^r(u^{(1)}_i-u^{(2)}_i)^2+\frac{2}{n}\sum_{\substack{{1\leq i\leq r}\\{i<j\leq d}}}(\mu_i-\mu_j)|\mathfrak{z}^{(1)}_{ij}-\mathfrak{z}^{(2)}_{ij}|^2\\&+O\left(\frac{||\theta^{(1)}||^3,||\theta^{(2)}||^3}{n^{3/2}}\right),\end{align*}
where $u^{(i)}_r=1-\sum_{k=1}^{r-1}u^{(i)}_k$ for $i=1,2$.
\end{lemma}
In Subsection 5.1 we have used the first part of the sample to construct a preliminary estimator $\tilde{\rho}_n$ which lies within an  $n^{-1/2+\epsilon}$ ball of the original state $\rho$.

Note that  $\tilde{\rho}_n$ is of standard form $\tilde{\mathscr{V}} \tilde{\rho}_{0,n} \tilde{\mathscr{V}}^*$ with $\tilde{\rho}_{0,n} $ diagonal and $ \tilde{\mathscr{V}}$ unitary. Thus it can be brought into diagonal form by the channel $\tilde{L}: \rho \mapsto \tilde{\mathscr{V}}^* \rho \tilde{\mathscr{V}}$. All states parametrized by $\theta$ around the central state $\tilde{\rho}_{0,n}$ can be reparametrized around the central state $\tilde{\rho}_n$
by using the inverse transformation $\tilde{L}^*: \rho_{\theta} \mapsto \tilde{\mathscr{V}} \rho_{\theta} \tilde{\mathscr{V}}^*$. Since the norm $||.||_2$ is invariant under unitary transformation of states we will have
$$||\tilde{\rho}_{0,n}-\rho_{\theta}||_2=||\tilde{\rho}_n-L^*(\rho_{\theta})||_2.$$
Henceforth we will assume that $\tilde{\rho}_n$ is diagonal i.e. we set it to be  $\rho_{0,0,r}$ (see (\ref{local_unrotated_qudit})). Thus  the transformations $L$ and $L^*$ are understood to be part of the two stage estimation process, but in the notation, following \cite{Guta_Kahn_qubit,Kahn&Guta}, we will suppress them for the sake of brevity. 

Now $\rho$, which lies within a $n^{-1/2+\epsilon}$ ball around $\tilde{\rho}_n$, can be parametrized with the local parameter $\theta$, whereupon it is sufficient to estimate $\theta$ using $\hat{\theta}_n$. Thus we let $\rho=\rho_{\theta}$ and denote its estimator $\hat{\rho}_n$ by $\rho_{\hat{\theta}}$.  Let $\mu_1>\mu_2>\ldots>\mu_r>0$ be the eigenvalues of a density matrix $\sigma$ of rank $r$ and $||\rho-\sigma||_2<n^{-1/2+\varepsilon}$. Since (\ref{forward_channel}) holds one can easily transfer the risk between the i.i.d and Gaussian models invoking (\ref{equivalence-rate}) and (\ref{risk-transfer}), i.e we have
\begin{equation*}\limsup_{n\rightarrow\infty}\sup_{||\rho-\sigma||_2<n^{-1/2+\varepsilon}}E_{\rho}n||\rho-\hat{\rho}_n||_2^2\leq\sum_{i=1}^r \mu_i(1-\mu_i)+\sum_{\substack{{1\leq i\leq r}\\{i<j\leq d}}}2\mu_i \label{qudit_upper risk}.
\end{equation*}
The last part uses the upper bound of Theorem \ref{optimal_risk_Gaussian} and the continuity of eigenvalues, i.e. the eigenvalues of $\sigma$ and the central state $\tilde{\rho}_n$ are asymptotically the same (see Appendix C for details). For the lower bound we choose an arbitrary $\sigma$ and again parametrize the state and its estimate by local parameters $\theta$ and $\hat{\theta}$, i.e we let $\rho=\rho_{\theta}$ and $\hat{\rho}_n=\rho_{\hat{\theta}}$. Since (\ref{reverse_channel}) holds, we can  transfer the risk between the i.i.d and Gaussian models invoking (\ref{equivalence-rate}) and (\ref{risk-transfer}) and then using the lower bound of Theorem \ref{optimal_risk_Gaussian} we have
\begin{equation*}\liminf_{n\rightarrow\infty}\inf_{\hat{\rho}_n}\sup_{||\rho-\sigma||_2<n^{-1/2+\varepsilon}}E_{\rho}n||\rho-\hat{\rho}_n||_2^2\geq\sum_{i=1}^r \mu_i(1-\mu_i)+\sum_{\substack{{1\leq i\leq r}\\{i<j\leq d}}}2\mu_i \label{qudit_lower risk}.
\end{equation*}
We note that since the dimension $d$ is kept constant, the $||.||_1$ and $||.||_2$ norms are equivalent and hence we can easily replace the $L^2$ neighborhood with a trace norm neighborhood. We define the trace norm neighborhood as follows:
\begin{equation}\label{trace-norm-neighbor}
\Sigma_{n,\varepsilon}\left(  \sigma\right)  :=\left\{  \rho \text{ is a rank } r \text{  state in } \cal{H}:\left\Vert
\rho-\sigma\right\Vert _{1}\leq n^{-1/2+\varepsilon}\right\}.
\end{equation}
The upper  and  lower bounds  can now be summarized in the following theorem.
\begin{theorem}{\label{optimal_risk_iid}}
Let $\mu_1>\mu_2>\ldots>\mu_r>0$ be the eigenvalues of a density matrix $\sigma$ of rank $r$. Then there exists an $\varepsilon<1/2$ such that
\begin{equation*}\lim_{n\rightarrow\infty}\inf_{\hat{\rho}_n}\sup_{\rho \in \Sigma_{n,\varepsilon}\left(  \sigma\right)}E_{\rho}n||\rho-\hat{\rho}_n||_2^2=\sum_{i=1}^r \mu_i(1-\mu_i)+\sum_{\substack{{1\leq i\leq r}\\{i<j\leq d}}}2\mu_i.
\end{equation*}
\end{theorem}
Note that in the above display we have used the first of the following equivalent expressions:
$$E_{\rho}n||\rho-\hat{\rho}_n||_2^2=\int n||\rho-\hat{\rho}_n||_2^2\Tr(\rho^{\otimes n}M(d\hat{\rho}_n))$$
and the infimum over $\hat{\rho}_n$ is an abbreviated notation for  an  infimum over all  POVMs $M$.
\section{Optimal estimation of linear functional of a low-rank state}

We first discuss the classical problem of estimation of a linear functional of a density. Consider i.i.d observations $X_{1},\ldots,X_{n}$ having density $f$ on
$[0,1]$, and consider a linear functional
\begin{equation*}
\Psi\left(  f\right)  =\int\varphi\left(  x\right)  f(x)dx,
\label{functional-1}%
\end{equation*}
where $\varphi$ is a bounded function on [0,1]. The estimator
\[
\hat{\Psi}_{n}=n^{-1}\sum_{i=1}^{n}\varphi\left(  X_{i}\right)
\]
satisfies the CLT
\[
n^{1/2}\left(  \hat{\Psi}_{n}-\Psi\left(  f\right)  \right)  \leadsto N\left(
0,V_{f}^{2}\right),
\]
where
\begin{equation*}
V_{f}^{2}=\mathrm{Var}_{f}\left(  \varphi\left(  X\right)  \right)  .
\label{Var-def}%
\end{equation*}
One can show by LAN type results that $\hat{\Psi}_{n}$ is asymptotically
optimal (see \cite{KosLev76}), i.e. $V_{f}^{2}$ is the best possible variance in a local asymptotic minimax sense. In this section, we discuss an analogous result for the quantum case.

Suppose $\rho$ is a state in a d-dimensional Hilbert
space $\mathcal{H}$ of rank $r$ and consider the functional
\begin{equation*}
\Psi\left(  \rho\right)  =\mathrm{Tr}\;(A\rho),\label{func-of-mixed-state}%
\end{equation*}
 where $A$ is a self-adjoint matrix. We will denote the operator norm of $A$ by $||A||$. 
Let $X_{A}$ be the r.v.
generated by $A$; then recall that the probability distribution of $X_A$ is given by (\ref{prob_measurement}) and we denote the expectation by 
$E_{\rho}X_{A}=\left\langle A\right\rangle _{\rho}$.
Also the variance of $X_{A}$ is given by
\begin{align}
\mathrm{Var}_{\rho}\left(  X_{A}\right)  &=\left\langle A^{2}\right\rangle
_{\rho}-\left\langle A\right\rangle _{\rho}^{2}=\left\langle \tilde{A}%
^{2}\right\rangle _{\rho}=:V_{\rho}^{2},\text{ where}\label{V-rho-def} \nonumber\\
\tilde{A}  &  :=A-\left\langle A\right\rangle _{\rho}\mathbf{1}.\nonumber%
\end{align}
Let $\Sigma_{n,\varepsilon}\left(  \rho_0\right)$ be a shrinking neighborhood in trace norm around a certain rank $r$ state $\rho_0$ defined as in (\ref{trace-norm-neighbor}).
 We aim to show that with a sample $X_{1},\ldots,X_{n}$, the estimator
$\bar{X}_{n}$ attains the asymptotic variance $V_{\rho_0}^{2}$ uniformly over
$\rho\in\Sigma_{n,\varepsilon}\left(  \rho_0\right)  $. For the purpose of proving an
asymptotic minimax theorem, 
consider the expression%
\[
\sup_{\rho\in\Sigma_{n,\varepsilon}\left(  \rho_0\right)  }E_{\rho}\left(
n^{1/2}\left(  \bar{X}_{n}-\Psi\left(  \rho\right)  \right) \right)^2  .
\]
We will show that this expression converges to $V_{\rho_0}^2$.

To give a lower bound to the risk, we will have to find, for a $d$-dimensional center state
$\rho_{0}$ and a parametrization of all $d$-dimensional nearby states, a ``least favorable parametric subfamily'' so that the
information bound for estimating $\Psi\left( \rho\right)  $ along this family
coincides with $V_{\rho_0}^{2}$. We consider the following local model indexed by the single parameter $t$:
\begin{equation*}
\rho_{t}=\rho_{0}+n^{-1/2}t H\text{, \ \ }t\in\left(  -n^{\varepsilon}
,n^{\varepsilon}\right),  
\end{equation*}
where the matrix $H$ is suitably chosen. The problem of estimation of the functional $\Psi$ can be shown to be equivalent to the estimation of the parameter $t$ in the limiting model of the form $\mathfrak{N}(t\tau,\mathscr{S})$ (see (\ref{classical_quantum})) for appropriately chosen $\tau$ and $\mathscr{S}$. We then use a Bayesian result from \cite{Hol-78} to establish a minimax bound for the estimation of $t$ in the limiting model.
We have the following theorem, the proof of which is deferred to Appendix C.
\begin{theorem}{\label{semiparametric}}
Let $\Psi$ be the functional given in (\ref{func-of-mixed-state}); then there exists an $\varepsilon<1/2$ such that
\begin{equation*}
\lim_{n\rightarrow \infty}\inf_{\hat{\Psi}}\sup_{\rho\in\Sigma_{n,\varepsilon}\left(  \rho_0\right)
}E_{\rho}\left( n^{1/2}\left(  \hat{\Psi}-\Psi\left(  \rho\right)
\right) \right)^2  = V_{\rho_0}^2.  \label{functional-risk-bound}%
\end{equation*}
\end{theorem}
Again note that in the above display we have used the first of the following equivalent expressions:
$$E_{\rho}\left( n^{1/2}\left(  \hat{\Psi}-\Psi\left(  \rho\right)
\right) \right)^2=\int \left( n^{1/2}\left(  \hat{\Psi}-\Psi\left(  \rho\right)
\right) \right)^2\Tr(\rho^{\otimes n}M(d\hat{\Psi}))$$
and the infimum over $\hat{\Psi}$ stands for an infimum taken over all POVMs $M$.
\section{Discussion}
We have obtained a generalization of the LAN result of \cite{Kahn&Guta} in this paper. Since a local asymptotic equivalence result was proved in \cite{BGN-QAE} for rank 1 states in the infinite dimensional case, a natural question is about the limiting model in the case of an ensemble of rank $r$ states if the Hilbert space is infinite dimensional. We conjecture that the limiting state will consist of three parts:   a classical multivariate normal law, a tensor product of $r(r-1)/2$ shifted thermal states and a tensor product of infinitely many shifted pure states. Although this result is outside the scope of this paper, we observe that a restriction of the conjectured model to finite dimension will correspond to our result while a restriction to the rank 1 case (in infinite dimension) will correspond to the result obtained in \cite{BGN-QAE}.

Another interesting phenomenon occurs in the estimation of a functional where the limiting model is a tensor product of classical and quantum states indexed by a one-dimensional parameter. It is well known that if the loss function $L(\hat{\mu},\mu)$ in Section 4 is replaced by a weighted loss function $L_g(\hat{\mu},\mu)$ given by
$$L_g\left(  \hat{\mu},\mu\right)  =g_1\left(  \hat{\mu}_{1}-\mu_{1}\right)
^{2}+g_2\left(  \hat{\mu}_{2}-\mu_{2}\right)  ^{2},$$
then for $g_1=1$ and $g_2=0$, the POVM induced by $Q$ suffices for optimal estimation, i.e. the Bayes risk is obtained by a simple measurement (the spectral measure of $Q$) and the associated Bayes risk matches the one in the the classical bivariate normal distribution $\mathcal{N}_2(\mu,\sigma^2I_2)$. Analogously the risk in the limiting model for the functional estimation matches the same for a related classical model (see the remark after the proof of Theorem 6.1 in Appendix C). This is in sharp contrast with the risk observed in the measurement of the shift parameter of the coherent or thermal state ($g_1=g_2=1$ in the expression of $L_g(\hat{\mu},\mu)$), where we observe an inflation of the risk compared to the classical case (see the last paragraph of Section 4). However, we emphasize that the classical-quantum Gaussian shift model indexed by a one-dimensional parameter is \textit{not} equivalent to a classical model (in terms of Le Cam equivalence). A discussion of this phenomenon in a similar model can be found in \cite{Guta_Kahn_qubit} where the authors argue that although measuring the position observable results in a classical model whose ``classical Fisher information'' matches that of the quantum Fisher information of the original model, optimal testing procedures in the classical model are suboptimal in the sense that the associated risk is higher than that of the optimal testing procedure of the original quantum model. A more in-depth discussion in terms of sufficient statistics and Le Cam equivalence can be found in \cite{GutaJencova2007}.

\textbf{Acknowledgements}

We would like to thank the associate editor and two anonymous referees whose suggestions helped in improving the quality of the paper.

\bibliographystyle{plain}
\bibliography{bib_low_LAN_v2}

\appendix

\section{Representation theoretic tools}
Our main tool in proving Theorem 3.3 will be the representation of the special unitary group $SU(d)$. We first define the representations of the group $GL(d)$ and $S(n)$ (the group of invertible matrices and permutations respectively) on $(\mathbb{C}^d)^{\otimes n}$. Let $f_{\mathbf{a}}=f_{{\mathbf{a}_1}}\otimes\ldots\otimes f_{{\mathbf{a}_n}}$,where $f_1,\ldots,f_d$ are basis elements of $\bb{C}^d$ and $\mathbf{a}_i\in\{1,\ldots,d\}$ (note that $f_{\mathbf{a}}$'s span $(\bb{C}^d)^{\otimes n}$) and consider the following actions:
\begin{align}
    \pi_n(T): f_{\mathbf{a}_1}\otimes\ldots\otimes f_{\mathbf{a}_n}&\rightarrow Tf_{\mathbf{a}_1}\otimes\ldots\otimes Tf_{\mathbf{a}_n}, \quad T\in GL(d)\label{T_action}\\
    \tilde{\pi}_n(\sigma): f_{\mathbf{a}_1}\otimes\ldots\otimes f_{\mathbf{a}_n}&\rightarrow f_{\sigma^{-1}(\mathbf{a}_1)}\otimes \ldots\otimes f_{\sigma^{-1}(\mathbf{a}_n)}, \quad \sigma\in S(n)\label{sigma_action}.
\end{align}
It can be shown that the representation space can be decomposed into direct sum of subspaces each of which is a tensor product of irreducible representations of $GL(d)$ and $S(n)$. This decomposition is called Schur-Weyl duality in the literature (for a detailed account see \cite{MR1153249,MR2522486}).
\begin{equation}
    (\mathbb{C}^d)^{\otimes n}=\bigoplus_{\lambda} \mathcal{H}_{\lambda}\otimes \mathcal{K}_{\lambda}\label{Schur-Weyl}
\end{equation}
Here the direct sum is finite and each $\lambda$ corresponds to a \textit{Young diagram}. A Young diagram is defined by a ordered tuple of integers $\lambda=(\lambda_1,\ldots,\lambda_d)$ with $\lambda_1\geq\ldots\geq\lambda_d\geq0$ and can be represented graphically by a diagram with $d$ lines each containing $\lambda_i$ boxes. For example a typical Young diagram looks like:
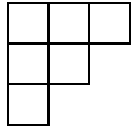
\begin{center}
\begin{ytableau}
 $\text{ }$&& \cr
 & \cr
 \cr
\end{ytableau}
\end{center}
with $\lambda=(3,2,2)$.
Also the operator $\pi_n(T)$ can be decomposed as follows
:
\begin{align*}
    \pi_n(T)=\bigoplus_{\lambda}\pi_{\lambda}(T)\otimes \mathbbm{1}_{\mathcal{K}_{\lambda}}
\end{align*}
i.e the tensor representation of $T\in GL(d)$ can be decomposed into blocks $\pi_{\lambda}(T)$ acting on the irreducible representation $\mathcal{H}_{\lambda}$ and these blocks have multiplicities equal to dimension of $\mathcal{K}_{\lambda}$. Thus we can decompose the tensor product of the states as follows
\begin{align}
    \rho^{\otimes n}_{0,u,n}&=\bigoplus_{\lambda} \tilde{\rho}_{\lambda}^{0,u,n}\otimes \mathbbm{1}_{\mathcal{K}_{\lambda}}\label{prelim_decomp}\\
    \rho^{\otimes n}_{0,u,n}&=\bigoplus_{\lambda} \left(\rho^{0,u,n}_{\lambda}\otimes \frac{\mathbbm{1}_{\mathcal{K}_{\lambda}}}{\dim(\mathcal{K}_{\lambda})}\right)p_n^{\lambda}\label{state_decomp}\\
    U(\mathfrak{z}/\sqrt{n})^{\otimes n}\rho^{\otimes n}_{0,u,n}U^*(\mathfrak{z}/\sqrt{n})^{\otimes n}&=\bigoplus_{\lambda} \left(U_{\lambda}(\mathfrak{z}/\sqrt{n})\rho^{0,u,n}_{\lambda}U_{\lambda}^*(\mathfrak{z}/\sqrt{n})\otimes \frac{\mathbbm{1}_{\mathcal{K}_{\lambda}}}{\dim(\mathcal{K}_{\lambda})}\right)p_n^{\lambda}\label{decomp_unitary}.
\end{align}
In equation (\ref{prelim_decomp}) we have used the usual decomposition of the $\pi_n(T)$ while in equation (\ref{state_decomp}) we have adjusted the factors $p_n^{\lambda}$ and $dim(\mathcal{K}_{\lambda})$ (later we give the exact expressions for these quantities) so that both $\rho^{0,u,r,n}_{\lambda}$ and $\frac{\mathbbm{1}_{\mathcal{K}_{\lambda}}}{\dim(\mathcal{K}_{\lambda})}$ are states (i.e. they are normalized). Equation (\ref{decomp_unitary}) follows from the fact that $SU(d)$ is a subgroup of $GL(d)$ and its tensor representation has a similar decomposition.
Define $$\Delta^{\mathfrak{z},n}(A)=U(\mathfrak{z}/\sqrt{n})AU^*(\mathfrak{z}/\sqrt{n})$$ and $\Delta^{\mathfrak{z},n}_{\lambda}$ in a similar fashion. Then the last decomposition can be written in a compact fashion as
$$\rho^{\otimes n}_{\theta,n}=\bigoplus_{\lambda} \left(\Delta^{\mathfrak{z},n}_{\lambda}(\rho^{0,u,n}_{\lambda})\otimes \frac{\mathbbm{1}_{\mathcal{K}_{\lambda}}}{\dim(\mathcal{K}_{\lambda})}\right)p_n^{\lambda}.$$
Next we describe the subspaces $\mathcal{H}_{\lambda}$ in detail. Let $\cal{A}(S(n))$ denote the group algebra, i.e. the linear space spanned by elements of the form $\sum_{\sigma} c_{\sigma}\sigma$ for complex numbers $c_{\sigma}$.
Also define multiplication  between its elements in the usual way i.e
\begin{align*}
    &a=\sum_{\sigma\in S(n)}a_{\sigma}\sigma,\quad b=\sum_{\tau\in S(n)}b_{\sigma}\sigma\\
    & ab= \sum_{\sigma \in S(n)}\left(\sum_{\tau\in S(n)}a_{\sigma\tau^{-1}}b_{\tau}\right)\sigma.
\end{align*}
Let $\lambda$ be a Young diagram with $|\lambda|=\sum_i \lambda_i =n$. We define two subgroups $\cal{R}_{\lambda}$ and $\cal{C}_{\lambda}$ of $S(n)$ defined as follows
\begin{align*}\cal{R}_{\lambda}&=\{\sigma: \sigma \text{ leaves the rows of } \lambda \text{ unaltered }\}\\
\cal{C}_{\lambda}&=\{\sigma: \sigma \text{ leaves the columns of } \lambda \text{ unaltered }.\}
\end{align*}
Now consider the action of $\tilde{\pi}_n(\sigma)$ (henceforth we will denote this operator as $\sigma$) given in relation (\ref{sigma_action}) and extend it to the group algebra. Next we define the following operators

$$p
_{\lambda}=\sum_{\sigma\in R_{\lambda}}\sigma, \quad q
_{\lambda}=\sum_{\sigma\in C_{\lambda}}sgn(\sigma)\sigma ,\quad y_{\lambda}=q_{\lambda}p_{\lambda}.$$

It can be shown that the $\{y_{\lambda}f_{\mathbf{a}}\}$ spans the subspace $\cal{H}_{\lambda}$. Of course the elements in  $\{y_{\lambda}f_{\mathbf{a}}\}$ are not linearly independent. To describe a basis for $\cal{H}_{\lambda}$ we need the concept of Young tableaux.
 
 A Young tableau is a Young diagram filled with integers. A semistandard Young tableau is a Young diagram filled with numbers from $1$ to $d$ such that the numbers are non-decreasing in a row and strictly increasing in a column, for example:
 \begin{center}
\begin{ytableau}
 $1$&$1$& $1$ &$2$\cr
  $2$&$3$ &$4$\cr
 $4$\cr
\end{ytableau}
\end{center}

is a semistandard Young tableau. Now if we fill a Young tableau with the integers $\mathbf{a}_1,\ldots,\mathbf{a}_n$ along the rows then for every $f_{\mathbf{a}}$ we get a corresponding Young tableau $t_{\mathbf{a}}$. As an example we associate the following vector with the semistandard Young tableau shown above -
$f_{1}\otimes f_{1}\otimes f_{1}\otimes f_{2}\otimes f_{3}\otimes f_{4}\otimes f_{4}$.
Define the set
 $$\mathfrak{A}=\{\mathbf{a}:t_{\mathbf{a}} \text{ a semistandard Young tableau }\}.$$
 It can be shown that $\{y_{\lambda}f_{\mathbf{a}}\}_{\mathbf{a}\in \mathfrak{A}}$ forms a basis of $\cal{H}_{\lambda}$.
 
 An alternative notation (also used in \cite{Kahn&Guta}) is to use $\mathbf{m}=\{m_{i,j}\}_{1\leq i<j\leq d}$ where $m_{i,j}$ is the number of $j$'s in row $i$. Note that if $\mathbf{a}\in\mathfrak{A}$, then $\mathbf{m}$ uniquely identifies $\mathbf{a}$. For the example above $m_{1,2}=m_{2,3}=m_{2,4}=m_{3,4}=1$ and the rest $m_{i,j}$ are $0$. 
 
 Now denote $|\mathbf{m}_{\lambda}\rangle=y_{\lambda}f_{\mathbf{m}}/||y_{\lambda}f_{\mathbf{m}}||$. Mapping these basis elements of $\cal{H}_{\lambda}$ to the basis elements of the multimode Fock space will be a key element in constructing the channel.
 
 From the definition of $p_{\lambda}$ and $q_{\lambda}$ it can be easily shown that
\begin{equation}
\label{projection_formula}
p_{\lambda}^2=(\prod_{i=1}^d\lambda_{i}!)p_{\lambda}, \quad q_{\lambda}^2=(\prod_{i=1}^di^{\lambda_i-\lambda_{i+1}})q_{\lambda}
\end{equation}
and consequently 
\begin{equation}
\label{csq1}\la y_{\lambda}f_{\mathbf{a}}|y_{\lambda}f_{\mathbf{b}}\ra=\la p_{\lambda}f_{\mathbf{a}}|q^2_{\lambda}p_{\lambda}f_{\mathbf{b}}\ra=(\prod_{i=1}^di^{\lambda_i-\lambda_{i+1}})\la p_{\lambda}f_{\mathbf{a}}|y_{\lambda}f_{\mathbf{b}}\ra.
\end{equation}
We also consider the properties of the vector $|{\bf 0}_ \lambda\rangle$ (which is the \emph{highest weight
vector} of the representation $(\pi_{\lambda},\mathcal{H}_{\lambda})$) 
corresponding to the semi-standard Young tableau where all the entries in row $i$
are $i$. An immediate consequence is that
\begin{equation}
\label{vpP}
p_{\lambda}| f_{\bf 0} \rangle  =
(\prod_{i=1}^{d} \lambda_i!) |f_{\bf 0}\rangle. 
\end{equation}
Moreover $\langle f_{\bf 0} | q_{\lambda} f_{\bf 0} \rangle = 1 $ since any column permutation produces a vector orthogonal to $f_{\bf 0}$ and thus the normalised vector is: \begin{equation}\label{eq.norm.0.lambda}
|{\bf 0}_\lambda\rangle =
\frac1{\prod_{i=1}^{d} \lambda_i! \sqrt{
i^{\lambda_i - \lambda_{i+1}}}} y_{\lambda} |f_{\bf 0}\rangle.
\end{equation}

The finite dimensional coherent states are defined as 
$\pi_{\lambda}(U) |{\bf 0}_{\lambda}\rangle$ for $U\in
SU(d)$. Since $[p_{\lambda}, \pi_{\lambda}(U)]=0$ and (\ref{vpP}) holds, we get
$p_{\lambda} \pi_{\lambda}(U)  |f_{\bf 0}\rangle  =
(\prod_{i=1}^{d} \lambda_i!) \pi_{\lambda}(U) |f_{\bf 0}\rangle$. Thus we have
\begin{align}
\label{fcoherent}
\langle y_{\lambda} f_{\bf m} | y_{\lambda}\pi_{\lambda}(U) f_{\bf 0} \rangle &=\langle p_{\lambda} f_{\bf m} | q_{\lambda}^2p_{\lambda}\pi_{\lambda}(U) f_{\bf 0} \rangle\nonumber\\
& = \langle p_{\lambda} f_{\bf m} | q_{\lambda}^2\pi_{\lambda}(U) p_{\lambda}f_{\bf 0} \rangle\nonumber\\
&=\prod_{i=1}^d
i^{\lambda_i - \lambda_{i+1}} \lambda_i!
\langle p_{\lambda} f_{\bf m}| q_{\lambda} \pi_{\lambda}(U) f_{\bf 0} \rangle.
\end{align} 
\section{Bayes estimation of a one dimensional shift parameter}
In this section, we adopt the framework for Gaussian states $\rho$ in \cite{Hol-78} to develop a Bayes result for estimation of $t$ in the model $\left\{
\mathfrak{N}(t\tau,\mathscr{S})  ,t\in\mathbb{R}\right\}$ (cf. (19)) .
In \cite{Hol-78} the characteristic function of $\rho$ is given by%
\begin{equation}
\phi_{\rho}\left(  z\right)  =\exp\left(  i\theta\left(  z\right)  -\frac
{1}{2}\alpha\left(  z,z\right)  \right)  \text{, }z\in\mathbb{R}%
^{m}\label{char-func-1}%
\end{equation}
where $\alpha$ is a positive symmetric form on $\mathbb{R}^{m}$ and $\theta$
is a linear function on $\mathbb{R}^{m}$ (cf. \cite{Hol-78}, III.2.4). Here the $\rho$ is a state on the Weyl algebra generated by the
Weyl unitaries
$W\left(  z\right)  $, $z\in\mathbb{R}^{m}$, which satisfy the canonical
commutation relations
\begin{equation}
W\left(  z\right)  W\left(  w\right)  =W\left(  z+w\right)  \exp\left(
\frac{i}{2}\Delta\left(  z,w\right)  \right)  \label{commut-relations}%
\end{equation}
(cf. \cite{Hol-78}
, III.1.3) where $\Delta$ is a bilinear antisymmetric (symplectic) form on
$\mathbb{R}^{m}$. The forms $\alpha$ and $\Delta$ must satisfy the condition
\begin{equation}
\frac{1}{4}\Delta\left(  z,w\right)  ^{2}\leq\alpha\left(  z,z\right)
\alpha\left(  w,w\right)  \text{, }z,w\in\mathbb{R}^{m}%
.\label{equiv-condition}%
\end{equation}
This is equivalent to the following condition (cf. \cite{Hol-2nd-ed-11}, Theorem
5.5.1): for any $r$ and $z_{1},\ldots,z_{r}$ $\in\mathbb{R}^{m}$, the two
complex Hermitian matrices below are positive semidefinite:
\begin{equation}
\left(  \alpha\left(  z_{j},z_{k}\right)  \pm\frac{i}{2}\Delta\left(
z_{j},z_{k}\right)  \right)  _{j,k=1}^{r}\geq0.\label{pos-semidef}%
\end{equation}
Suppose $\alpha$ and $\Delta$ are given by $m\times m$ real matrices:
\begin{equation}
\alpha\left(  z,w\right)  =z^{T}\Sigma w\text{ where }\Sigma
>0\label{Sigma-def}%
\end{equation}
(strictly positive definite) and
\[
\Delta\left(  z,w\right)  =z^{T}Dw
\]
where $D$ is antisymmetric. The symplectic form $\Delta$ is allowed to be
degenerate (cf. \cite{Hol-78}, beginning of Chap III, \S 2), which means that $D$
is allowed to be singular. Indeed $\Delta\equiv0$, i.e. $D=0$ is included as
the classical case, where $\rho$ is identified with the $m$-variate normal
distribution $\cal{N}_{m}\left(  \theta,\Sigma\right)  $ (here we identified the
linear form $\theta\left(  z\right)  $ with a vector $\theta\in\mathbb{R}^{m}%
$). The framework of an antisymmetric, possibly singular $D$ now allows to
include products of classical and quantum Gaussian states.

It is easy to see that (\ref{pos-semidef}) is equivalent to
\begin{equation}
\Sigma\pm\frac{i}{2}D\geq0 \label{complex-cov-matrix}%
\end{equation}
i.e. the above complex Hermitian matrices are positive
semidefinite.

In \cite{Hol-78}, III, \S 2, (2.1) it is further noted that (under some
conditions) one has $W\left(  z\right)  =\exp\left(  iR\left(  z\right)
\right)  $ where $R\left(  z\right)  $ is a self-adjoint operator (often
called a "field operator"), and from (\ref{commut-relations}) it follows that%
\[
\left[  R\left(  z\right)  ,R\left(  w\right)  \right]  =-i\Delta\left(
z,w\right)  \;\mathbf{1}.
\]
(the equality holding in appropriate domains). The
case $m=2$ corresponds to
\begin{equation}
\Delta\left(  z,w\right)  =z^{T}Dw\text{ where }D=\Omega:=\left(
\begin{array}
[c]{cc}%
0 & 1\\
-1 & 0
\end{array}
\right)  \label{Omega-def}%
\end{equation}
which yields a one-mode Gaussian state. Assume $z=\left(  1,0\right)  $ and $w=\left(
0,1\right)  $. Then $\Delta\left(  z,w\right)  =1$ and
\[
\left[  R\left(  \left(  1,0\right)  \right)  ,R\left(  \left(  0,1\right)
\right)  \right]  =-iz^{T}\Omega w\;\mathbf{1}=-i\;\mathbf{1}%
\]
and identifying $R\left(  \left(  1,0\right)  \right)  $ with $P$ and
$R\left(  \left(  0,1\right)  \right)  $ with $Q$, we obtain the standard
commutation relations
\[
\left[  Q,P\right]  =i\;\mathbf{1.}%
\]
Then, since for field operators the mapping $z\rightarrow R\left(  z\right)  $
is real-linear (cf. \cite{Partha-book}), we can write $R\left(  z\right)  =z_{1}%
P+z_{2}Q$, and the Weyl unitaries take the form
\[
W\left(  z\right)  =\exp\left(  i\left(  z_{1}P+z_{2}Q\right)  \right)  .
\]
It is well know that the characteristic function of the vacuum state in this
setting should be
\[
\phi_{\rho}\left(  z\right)  =\exp\left(  -\frac{\left\Vert z\right\Vert ^{2}%
}{4}\right)  \text{, }z\in\mathbb{R}^{2}.
\]
Comparing with (\ref{char-func-1}), we see that $\alpha\left(  z,z\right)
=\left\Vert z\right\Vert ^{2}/2$ and hence the covariance matrix $\Sigma\ $in
(\ref{complex-cov-matrix}) should be
\[
\Sigma=\frac{1}{2}I_{2}\text{.}%
\]
We can check condition (\ref{complex-cov-matrix}) for this $\Sigma$ and $D=\Omega$:
considering the equivalent condition (\ref{equiv-condition}), we have to
verify%
\[
\left(  \frac{1}{2}z^{T}\Omega w\right)  ^{2}\leq z^{T}\Sigma
z\;w^{T}\Sigma w
\]
or
\[
\left(  z^{T}\Omega w\right)  ^{2}\leq\left\Vert z\right\Vert ^{2}\left\Vert
w\right\Vert ^{2}.
\]
Since $\Omega^{T}\Omega=I_{2}$, we have $\left\Vert \Omega w\right\Vert ^{2}=\left\Vert
w\right\Vert ^{2}$, and setting $u=\Omega w$, the above inequality is equivalent to
\[
\left(  z^{T}u\right)  ^{2}\leq\left\Vert z\right\Vert ^{2}\left\Vert
u\right\Vert ^{2},
\]
i.e. the Cauchy-Schwarz inequality, which becomes an equality for $u=z$. Hence
(\ref{complex-cov-matrix}) is fulfilled for $\Sigma=I_{2}/2$, and for
multiples of the unit matrix, $1/2$ is the minimal factor so that $\Sigma$ can
be a covariance matrix. This agrees with the well-known minimum uncertainty
property of the vacuum state; the thermal states are then those with
covariance matrix $\gamma I_{2}$ where $\gamma>1/2$.

It is now obvious that an $n$-mode zero mean Gaussian state which is a product
of $n$ thermal (or vacuum) states is captured by the framework
(\ref{char-func-1}) as follows: $m=2n$,
\[
\alpha\left(  z,z\right)  =z^{T}\Sigma z\text{ where }\Sigma=%
{\textstyle\bigoplus\nolimits_{j=1}^{n}}
\gamma_{j}I_{2}\text{, }%
\]
where $\gamma_{j}\geq1/2$, $j=1,\ldots,n$ and
\[
\Delta\left(  z,w\right)  =z^{T}Dw\text{ where }D=%
{\textstyle\bigoplus\nolimits_{j=1}^{n}}
\Omega
\]
with $\Omega$ from (\ref{Omega-def}).

Since the form $\Delta$ is allowed to be degenerate, we can also consider the
product of the above with a classical $d$-variate normal
distribution $\mathcal{N}_{d}\left(  \mu,V\right)  $ by setting
$m=2n+d$,
\begin{align}
\alpha\left(  z,z\right)   &  =z^{T}\Sigma z\text{ where }\Sigma=%
V\oplus {\textstyle\bigoplus\nolimits_{j=1}^{n}}
\gamma_{j}I_{2}\text{,}\label{alpha_expression}\\
\text{ }\Delta\left(  z,w\right)   &  =z^{T}Dw\text{ where }D=%
0_{d\times d}\oplus{\textstyle\bigoplus\nolimits_{j=1}^{n}}
\Omega\nonumber.
\end{align}
A non-zero mean $\theta\in\mathbb{R}^{m}$ is simply accounted for by setting
$\theta\left(  z\right)  =\theta^{T}z$ in (\ref{char-func-1}).

In \cite{Hol-78}, III.5 estimation of a one dimensional shift parameter   $t$
is considered, which has a normal prior distribution with zero mean and
variance $b$. The family of Gaussian states is given by (\ref{char-func-1})
where $\theta\left(  z\right)  $ is a linear form on $\mathbb{R}^{m}$ given by
$\theta\left(  z\right)  =t\theta_{0}\left(  z\right)  $, $\theta
_{0}\left(  z\right)  $ being a fixed linear form and $t\in\mathbb{R}$. Assume the linear form $\theta
_{0}\left(  z\right)  $ is given by a vector $\tau$ such that $\theta
_{0}\left(  z\right)  =\tau^{T}z$, then the family of Gaussian states
will have shift parameters  $t\tau$, $t\in\mathbb{R}$ and fixed covariance
function $\alpha\left(  z,w\right)  .$ Recall that according to
(\ref{Sigma-def}) we have $\alpha\left(  z,w\right)  =z^{T}\Sigma w$
where $\Sigma>0$; thus, in this setup the characteristic function  in (\ref{char-func-1}) matches the characteristic function of $\mathfrak{N}(t\tau,\mathscr{S})  ,t\in\mathbb{R}.$ Comparing (\ref{alpha_expression}) and (19), we note that $S=D$ when $d=r-1$ and $n=r(r-1)/2+r(d-r)$.

In \cite{Hol-78} the linear form $\theta_{0}\left(  z\right)  $ is represented
by a vector in a different way. First it is to be noted that in this section
in \cite{Hol-78} the quadratic form $\alpha$ has acquired a different meaning
from (\ref{char-func-1}): the $\alpha$ there is now written $\alpha_{0}$, and
$\alpha$ is now defined as (see (4.1) in III.4)
\[
\alpha\left(  z,w\right)  :=\alpha_{0}\left(  z,w\right)  +\beta\left(
z,w\right)
\]
where $\beta$ is the covariance function of the prior distribution, defined
as
\[
\beta\left(  z,w\right)  =\int\theta\left(  z\right)  \theta\left(  w\right)
\pi\left(  d\theta\right)
\]
where $\pi$ is the prior distribution and $d\theta$ represents integration
over linear forms (i.e. over vectors in $\mathbb{R}^{m}$). In the present
case, the prior distribution is given by $\theta\left(  z\right)  =t\tau^Tz$
where $\tau$ is fixed and $t\sim N\left(  0,b\right)$. Hence
\[
\beta\left(  z,w\right)  =E\left(  tz^{T}\tau\right)
\left(t\tau^{T}w\right)  =Et^{2}\left(  z^{T}\tau\tau^{T
}w\right)  =z^{T}\left(  b\tau\tau^{T}\right)  w.
\]
So we have a matrix representation of the quadratic form $\alpha$:
\[
\alpha\left(  z,w\right)  =z^{T}\Lambda w\text{, }\Lambda=\Sigma+b\tau\tau^{T}%
\]
where $\Lambda>0$ since $\Sigma>0$. 

Now the vector $\theta_{0}\in\mathbb{R}^{m}$ to represent the linear form
$\theta_{0}\left(  z\right)  $ is chosen such that
\[
\theta_{0}\left(  z\right)  =\alpha\left(  \theta_{0},z\right)  .
\]
The relation to $\tau$ is given by
\begin{align*}
\tau^{T}z  & =\theta_{0}\left(  z\right)  =\theta_{0}^{T}\Lambda z\text{,
}z\in\mathbb{R}^{m},\\
\theta_{0}  & =\Lambda^{-1}\tau.
\end{align*}
The loss function for estimation of $t$ is $g\left(  \hat{t}-t\right)^2  $ for
some $g>0$; we can set $g=1$ here. Then in \cite{Hol-78} an optimal
measurement for estimating the parameter $t$ is found, and the Bayes risk is
given by
\begin{equation}
E\left(  \hat{t}-t\right)  ^{2}=\frac{ba_{0}}{a_{0}+b}\label{Bayes-risk}%
\end{equation}
(after relation (5.3) of \cite{Hol-78}), where $a_{0}=a-b$ and $a=\alpha\left(  \theta
_{0},\theta_{0}\right)  ^{-1}$.
Note that we have used the condensed notation while stating the Bayes risk (see the discussion after the statement of Bayes problem in subsection 2.4).
In order to clarify the dependence of the Bayes risk on the parameters $\tau$
and $\Sigma$ of the Gaussian shift model, note that 
\begin{align*}
a^{-1}  & =\alpha\left(  \theta_{0},\theta_{0}\right)  =\tau^{T}%
\Lambda^{-1}\Lambda\Lambda^{-1}\tau=\tau^{T}\Lambda^{-1}\tau\\
& =\tau^{T}\left(  \Sigma+b\tau\tau^{T}\right)  ^{-1}\tau.
\end{align*}
Define $v=\Sigma^{-1/2}\tau$, then
\[
\tau^{T}\left(  \Sigma+b\tau\tau^{T}\right)  ^{-1}\tau=v^{T
}\left(  I_{m}+bvv^{T}\right)  ^{-1}v.
\]
Now defining a unit vector $e=v/\left\Vert v\right\Vert $ we have
\begin{align*}
v^{T}\left(  I_{m}+bvv^{T}\right)  ^{-1}v  & =\left\Vert
v\right\Vert ^{2}e^{T}\left(  I_{m}+\left\Vert v\right\Vert
^{2}bee^{T}\right)  ^{-1}e\\
& =\left\Vert v\right\Vert ^{2}\left(  1+\left\Vert v\right\Vert ^{2}b\right)
^{-1}.
\end{align*}
Consequently
\begin{align*}
a  & =\left\Vert v\right\Vert ^{-2}\left(  1+\left\Vert v\right\Vert
^{2}b\right)  =\left\Vert v\right\Vert ^{-2}+b,\\
a_{0}  & =a-b=\left\Vert v\right\Vert ^{-2}=\left(  \tau^{T}\Sigma
^{-1}\tau\right)  ^{-1}.
\end{align*}
The Bayes risk (\ref{Bayes-risk}) now takes the form
\begin{equation}
E\left(  \hat{t}-t\right)  ^{2}=\frac{\left(  \tau^{T}\Sigma^{-1}%
\tau\right)  ^{-1}b}{\left(  \tau^{T}\Sigma^{-1}\tau\right)  ^{-1}%
+b}.\label{Bayes-risk-2}%
\end{equation}
\section{Proof of the main theorems}
\subsection{Integral representation of Gaussian states}
Recall the displacement operator $D(z)$ (and the coherent state $D(z)|0\rangle$) defined in equation (7) and $D\left(  \mathbf{z}\right)$ (and the multimode coherent state $D\left(  \mathbf{z}\right)  \left\vert \mathbf{0}\right\rangle$) in equation (12). We adopt the notation 
$$|z\rangle=D(z)|0\rangle,\quad |\mathbf{z}\ra =D\left(  \mathbf{z}\right)  \left\vert \mathbf{0}\right\rangle$$ which will be the used in the following integral representation and subsequent proofs.

An alternate representation of the thermal state is given by
\[\phi_{\beta}=\frac{e^{\beta}-1}{\pi}\int\exp{-(e^{\beta}-1)|z|^2})|z\ra\la z|dz\]
i.e. a Gaussian mixture of $|z\ra\la z|$'s with the mixture distribution having mean $0$ and covariance matrix $\frac{1}{2(e^{\beta}-1)}I_2$.
Similarly we have the following integral representation for shifted thermal states:
\[
\phi^{z_0}_{\beta}=D(z_0)\phi_{\beta}D^*(z_0)=\frac{e^{\beta}-1}{\pi}\int\exp{-(e^{\beta}-1)|z|^2})D(z_0)|z\ra\la z|D^*(z_0)dz.
\]
  Generalizing to the multimodal case, the tensor product of shifted thermal states given in (16) can be written as:
 $$\phi^{\mathbf{\mathfrak{z}}}= \prod_{1\leq i<j \leq d}\frac{e^{\beta_{ij}}-1}{\pi}\int \exp\left(-\sum_{1\leq i<j \leq d}\frac{e^{\beta_{ij}}-1}{\pi}|z_{ij}|^2\right)|\mathfrak{z}+\mathbf{z}\ra\la \mathfrak{z}+\mathbf{z}| d\mathbf{z}.$$
Similarly we have the following  integral representation of the tensor product of shifted thermal states and shifted pure states given in (17):
\begin{align}
    \phi^{\mathbf{\mathfrak{z},r}}&= \bigotimes_{1\leq i<j \leq r}\phi_{\beta_{ij}}^{\mathfrak{z}_{ij}}\otimes\bigotimes_{\substack{1\leq i\leq r\\ r+1\leq j \leq d}}\phi_{\infty}^{\mathfrak{z}_{ij}}\quad \in \mathcal{T}_1(\mathcal{F}^r)\nonumber\\
    &= \left(\prod_{1\leq i<j \leq r}\frac{e^{\beta_{ij}}-1}{\pi}\int \exp\left(-\sum_{1\leq i<j \leq r}\frac{e^{\beta_{ij}}-1}{\pi}|z_{ij}|^2\right)|\mathfrak{z}_1+\mathbf{z}\ra\la \mathfrak{z}_1+\mathbf{z}| d\mathbf{z}\right) \otimes |\mathbf{\mathfrak{z}_2}\ra\la \mathbf{\mathfrak{z}_2}|\label{low_integral_form}.
\end{align}

Here $\mathfrak{z}_1=(\mathfrak{z}_{ij})_{1\leq i<j\leq r}$ is a vector in $\bb{C}^{r(r-1)/2}$ and $\mathfrak{z}_2=(\mathfrak{z}_{ij})_{1\leq i\leq r,r+1\leq j \leq d}$ is a vector in $\bb{C}^{r(d-r)}$. Note that we have used the notation $|\mathfrak{z}_1+\mathbf{z}\ra$ to denote the vector as an element of $\mathcal{F}^{r,1}$ and $|\mathfrak{z}_2\ra$ as element of $\mathcal{F}^{r,2}$, where
$$\mathcal{F}^{r,1}:=\bigotimes_{1\leq i<j \leq r}L^2(\mathbb{R}),\quad \mathcal{F}^{r,2}:=\bigotimes_{\substack{1\leq i\leq r\\ r+1\leq j \leq d}}L^2(\mathbb{R}).$$

Henceforth we will use this notation without mentioning the underlying spaces $\mathcal{F}^{r,1}$ and $\mathcal{F}^{r,2}$ when there is no chance of confusion. Following (10) we will define the number basis for $\mathcal{F}^{r},\mathcal{F}^{r,1},\mathcal{F}^{r,2}$ as $|\bm\ra,|\bm_1\ra$ and $|\bm_2\ra$, respectively. In particular, we have
\begin{align*}
    |\mathbf{m}\ra&=\otimes_{\ijr}|m_{ij}\ra,\quad \mathbf{m}=\{m_{ij}\in \mathbb{N}:\ijr\}\\
    |\mathbf{m}_1\ra&=\otimes_{1\leq i<j\leq r}|m_{ij}\ra,\quad \mathbf{m}_1=\{m_{ij}\in \mathbb{N}:1\leq i<j \leq r\}\\
    |\mathbf{m}_2\ra&=\otimes_{1\leq i\leq r, r+1\leq j \leq d}|m_{ij}\ra,\quad \mathbf{m}_2=\{m_{ij}\in \mathbb{N}:1\leq i\leq r, r+1\leq j \leq d\}.
\end{align*}
\begin{proof}[Proof of Theorem 3.3]

\textbf{ Construction of the channel $T^r_n$}

We will follow \cite{Kahn&Guta} to construct the channel. However we first state the redundancy of some representations in the low dimensional case. We will see that it is because of these redundancies that the mixture of pure and thermal states arise in the limit. In particular when stated for the low dimensional case some of the representations (those with Young diagram containing more than $r$ rows) will vanish from the sum (\ref{prelim_decomp}) or (\ref{decomp_unitary}). Further the $m_{i,j}$s with $1\leq i<j\leq r$ will determine the central state and the rotation, while the $m_{i,j}$s with $1\leq i\leq r$ and $r+1\leq j\leq d$ will contribute only to rotation. The former will correspond to the thermal states in the limit while the latter will give rise to pure states.

We start by noting that $f_{\mathbf{a}}=f_{{\mathbf{a}_1}}\otimes\ldots\otimes f_{{\mathbf{a}_n}}$ is an eigenvector of $\rho^{\otimes n}_{0,u,n}$
 with eigenvalues $\prod_{i=1}^d (\mu^{u,n}_{i})^{m_i}$, where $\mu^{u,n}_{i}=\mu_i+n^{-1/2}u_i$ and $m_i$ is the number of times $i$ occur in $(\mathbf{a}_1,\ldots,\mathbf{a}_n)$. 
 Since ${\mu^{u,n}_{i}}=0$ for $i>r$ we have
 $$\rho^{\otimes n}_{0,u,r,n}|f_\mathbf{a}\rangle=\prod_{i=1}^r (\mu^{u,n}_{i})^{m_i}|f_\mathbf{a}\rangle$$
 if $\mathbf{a}_j\in\{1,2,\ldots,r\}$ for all $j=1,2,\ldots,n$ and the RHS of the above equation is $0$ if $\exists \text{ }j$ with $\mathbf{a}_j>r$.
 
 Note that $y_{\lambda}f_\mathbf{a}$ is a linear combination of tensor products each containing the same multiplicities (i.e. the set $\{m_i\}$ is same since $\sigma$ just permutes the basis vectors) and that each tensor product has the same eigenvalue. Thus $y_{\lambda}f_\mathbf{a}$ being a linear combination of eigenvectors (all of them having same eigenvalue), is itself an eignevector with the same eigenvalue. Thus we have
 $$\rho^{\otimes n}_{0,u,r,n}y_{\lambda}|f_\mathbf{a}\rangle=\prod_{i=1}^r (\mu^{u,n}_{i})^{m_i}y_{\lambda}|f_\mathbf{a}\rangle$$
 if $\mathbf{a}_j\in\{1,2,\ldots,r\}$ for all $j=1,2,\ldots,n$ and the RHS of the above equation is $0$ if $\exists \text{ }j$ with $\mathbf{a}_j>r$.
 
 Recall that $\{y_{\lambda}f_{\mathbf{a}}\}_{\mathbf{a}\in \mathfrak{A}}$ spans $\cal{H}_{\lambda}$. Now consider a semistandard tableau with more than $r$ rows or that $\lambda_{r+1}>0$. By the definition, the entries of a column are strictly increasing and hence the $(r+1,1)$ entry of the tableau is greater than $r$. Consequently there exists $j$ such that $\mathbf{a}_j>r$ and hence $y_{\lambda}f_{\mathbf{a}}$ belongs to the null space of $\rho^{\otimes n}_{0,u,r,n}$. 
 
 Let $\Lambda_1$ be the collection of Young diagrams with number of row less than equal to $r$ i.e. $\Lambda_1=\{\lambda: \lambda_k =0 \text{ for } k>r\}$. Consider the alternate notation given in the last paragraph Appendix A where one identifies $\mathbf{a}$ with $\mathbf{m}$ and denotes the normalized basis elements as $|\bm_{\lambda}\ra$. If $\lambda\in \Lambda_1$, $\mathbf{m}=\{m_{i,j}\}_{1\leq i\leq r,i<j\leq d}$ since $i\leq r$ for these $\lambda$s. Let $\mathcal{H}(\{m_i\})=\text{span }\{|\mathbf{l}_{\lambda}\rangle:l_i=m_i\}$ i.e the span of the vectors $l_{\lambda}$ with the same multiplicity $\{m_i\}$. Note that for each $\mathbf{m}$ we have a set $\{m_i\}$ but different $\mathbf{m}$ can give rise to the same $\{m_i\}$. We now have a modified version of (\ref{prelim_decomp}):
 \begin{equation}\rho^{\otimes n}_{0,u,r,n}=\bigoplus_{\lambda\in \Lambda_1} \tilde{\rho}_{\lambda}^{0,u,r,n}\otimes \mathbbm{1}_{\mathcal{K}_{\lambda}}\end{equation}
 with $$\tilde{\rho}_{\lambda}^{0,u,r,n}=\sum_{\{m_i\}}\prod_{i=1}^r (\mu^{u,n}_{i})^{\lambda_i}\prod_{j=i+1}^r \left(\frac{\mu^{u,n}_{j}}{\mu^{u,n}_{i}}\right)^{m_{i,j}}\sum_{m \in \cal{M}^{\lambda}_r\cap \{m_i\}}|\bar{\bm}_{\lambda}\rangle\langle\bar{\bm}_{\lambda}|$$
 where $\cal{M}^{\lambda}_r=\{m: m_{i,j}=0 \text { for } i>r \text{ or } j>r\}$ and $|\bar{\bm}_\lambda\rangle$s are orthogonal vectors which spans the space $\mathcal{H}(\{m_i\})$. We would like to write the above decomposition in terms of  $|\bm_{\lambda}\rangle\langle \bm_{\lambda}|$ but since they are not orthogonal we have to use $|\bar{\bm}_{\lambda}\rangle\langle\bar{\bm}_{\lambda}|$ instead. Later we will prove that we can write the same decomposition in terms of $|\bm_{\lambda}\rangle\langle \bm_{\lambda}|$ and an error term that goes to $0$. Note that although $|\bm_{\lambda}\rangle$ is a basis element for all $m$ coming from a semistandard tableau, $|\bm_{\lambda}\rangle$ still belongs to the null space of $\rho^{\otimes n}_{0,u,n}$ if $m_{i,j}>0$ for some $j>r$.
 
 We normalize $\tilde{\rho}_{\lambda}^{0,u,r,n}$ so that  \begin{equation}\rho^{0,u,r,n}_{\lambda}=\frac{1}{N^{u,r,n}_{\lambda}}\sum_{\{m_i\}}\prod_{i=1}^r (\mu^{u,n}_{i})^{\lambda_i}\prod_{j=i+1}^r \left(\frac{\mu^{u,n}_{j}}{\mu^{u,n}_{i}}\right)^{m_{i,j}}\sum_{m \in \cal{M}^{\lambda}_r\cap \{m_i\}}|\bar{\bm}_{\lambda}\rangle\langle\bar{\bm}_{\lambda}|\label{diagonal_rho_lambda}
 \end{equation}
 where $$N^{u,r,n}_{\lambda}=\sum_{\bm\in \cal{M}^{\lambda}_r}\prod_{i=1}^r (\mu^{u,n}_{i})^{\lambda_i}\prod_{j=i+1}^r \left(\frac{\mu^{u,n}_{j}}{\mu^{u,n}_{i}}\right)^{m_{i,j}}.$$ Finally we define $p_{\lambda}^{u,r,n}=N^{u,r,n}_{\lambda}.\dim(\mathcal{K}_{\lambda})$ (the dependence of $p_{\lambda}^{u,r,n}$ on $\theta$ coming only from $u$) and we obtain the modified version of (\ref{state_decomp}):
  $$\rho^{\otimes n}_{0,u,r,n}=\bigoplus_{\lambda \in \Lambda_1} \left(\rho^{0,u,r,n}_{\lambda}\otimes \frac{\mathbbm{1}_{\mathcal{K}_{\lambda}}}{\dim(\mathcal{K}_{\lambda})}\right)p_{\lambda}^{u,r,n}.$$
Since $\Tr{\rho^{\otimes n}_{0,u,r,n}}=1$, taking trace on both sides of the above equation we obtain $\sum_{\lambda \in \Lambda_1}p_{\lambda}^{u,r,n}=1$. We also observe that because of the normalization both $\rho^{0,u,r,n}_{\lambda}$ and $\frac{\mathbbm{1}_{\mathcal{K}_{\lambda}}}{\dim(\mathcal{K}_{\lambda})}$ are states and so is their tensor product.
We have decomposed the tensor state $\rho^{\otimes n}_{0,u,n}$ into  blocks (each having trace 1 and thus a state) multiplied by a weight. All these weights are positive and add upto $1$ and we can interpret this as decomposition into mixture of smaller dimensional states (original state was $dn$-dimensional). Similarly the tensored unitaries $U(\mathfrak{z}/\sqrt{n})^{\otimes n}$ are decomposed into block matrices.

Defining
$$\Delta_{\lambda}^{\mathfrak{z},r,n}(A)=U_{\lambda}^r(\mathfrak{z}/\sqrt{n})AU_{\lambda}^{r*}(\mathfrak{z}/\sqrt{n}),$$ 
we obtain the following equivalent of (\ref{decomp_unitary}) in this modified setting
$$\rho^{\otimes n}_{\theta,r,n}=\bigoplus_{\lambda \in \Lambda_1} \left(\Delta^{\mathfrak{z},r,n}_{\lambda}(\rho^{0,u,r,n}_{\lambda})\otimes \frac{\mathbbm{1}_{\mathcal{K}_{\lambda}}}{\dim(\mathcal{K}_{\lambda})}\right)p_{\lambda}^{u,r,n}.$$
 The channel that we construct will contain two parts - the classical part maps the weights or probabilities of the block to a normal distribution via a classical channel or a Markov kernel. The quantum part involves mapping the basis of each representation to the Fock basis of multimode state and the action of unitaries under this channel corresponds to action of the displacement operators. We now give a brief sketch of the channel mentioned above, the exact details of the terms involved deferred to later sections.

Define 
$$T^{r}_{n}(\rho^{\otimes n}_{\theta,r,n})=\bigoplus_{\lambda \in \Lambda_1}  (p_{\lambda}^{u,r,n}\circ\tau^{r,n}_{\lambda})\otimes T^r_{\lambda}(\Delta^{\mathfrak{z},r,n}_{\lambda}(\rho^{0,u,r,n}_{\lambda}))$$
where $\tau^{r,n}_{\lambda}$ is a Markov kernel defined below and $T^r_{\lambda}$ is an map of the from $V^r_{\lambda}(.)V^{r*}_{\lambda}$ from $M(\cal{H}_{\lambda})$ to the multimode Fock space. The explicit from of the maps $V^r_{\lambda}$ will be discussed later. We define $\tau^{r,n}_{\lambda}$ as follows:
\begin{equation}\label{taulambda}
\tau^{r,n}_{\lambda}(A)=n^{(r-1)/2}\mu(A\cap A_{\lambda,r,n})
\end{equation}
where both $A$ and $A_{\lambda,r,n}$ are sets in $\mathbb{R}^{r-1}$ and $\mu$ is the Lebesgue measure in $\mathbb{R}^{r-1}$. Here

$$A_{\lambda,r,n}=\{x \in \mathbb{R}^{r-1}:|n^{1/2}x_i+n\mu_i-\lambda_i|\leq 1/2,1\leq i\leq r-1\}.$$

Also define $b_{\lambda}^{u,r,n}=p_{\lambda}^{u,r,n}\circ\tau^{r,n}_{\lambda}$ and $\Delta^{\mathfrak{z},n}_{\lambda}(\rho^{0,u,r,n}_{\lambda})=\rho_{\lambda}^{\theta,r,n}$. In this setup we have to show that
$$\sup_{\Theta_{n,r,\beta,\gamma}}||\bigoplus_{\lambda \in {\Lambda_1}}  b_{\lambda}^{u,r,n} \otimes T^r_{\lambda}(\Delta^{\mathfrak{z},r,n}_{\lambda}(\rho^{0,u,r,n}_{\lambda}))-\cal{N}_{r-1}(u,V_{\mu})\otimes\phi^{\mathfrak{z},r}||_1\rightarrow 0.$$
Define the following set of Young diagrams
\begin{equation}
    \Lambda_{n,\alpha}:=\{\lambda:|\lambda_i-n\mu_i|\leq n^{\alpha},1\leq i\leq r\}, \quad \text{ for } \alpha>1/2.
\end{equation}
For convenience of notation we will denote $\Lambda_{n,\alpha}$ by  $\Lambda_2$ and $\Theta_{n,r,\beta,\gamma}$ by $\Theta$.  For the displacement operator $D(\mathfrak{z})$ we will denote the action $D(\mathfrak{z})(.)D^*(\mathfrak{z})$ by $D^{\mathfrak{z}}$. Finally we note that  $D(\mathbf{\mathfrak{z}})(\phi^{0,r})=\phi^{\mathbf{\mathfrak{z}},r}$. By triangle inequality one obtains
\begin{align*}
&||\bigoplus_{\lambda \in {\Lambda_1}}
 b_{\lambda}^{u,r,n}\otimes T^r_{\lambda}(\rho^{\theta,r,n}_{\lambda})-\cal{N}_{r-1}(u,V_{\mu})\otimes\phi^{\mathbf{\mathfrak{z},r}}||_1\\
\leq &\sum_{\lambda \in \Lambda_1}||T^r_{\lambda}(\rho^{\theta,r,n}_{\lambda})-\phi^{\mathbf{\mathfrak{z},r}}||_1||b_{\lambda}^{u,r,n}||_1+||\sum_{\lambda \in \Lambda_1}b_{\lambda}^{u,r,n}-\cal{N}_{r-1}(u,V_{\mu})||_1\\
 \leq &||\sum_{\lambda \in \Lambda_1}b_{\lambda}^{u,r,n}-\cal{N}_{r-1}(u,V_{\mu})||_1+2\sum_{\lambda \in \Lambda_1\cap \Lambda^c_2}p_{\lambda}^{u,r,n}+\sup_{\lambda\in \Lambda_1\cap \Lambda_2}||T^r_{\lambda}(\rho_{\lambda}^{\theta,r,,n})-\phi^{\mathbf{\mathfrak{z},r}}||_1\\
 \leq &||\sum_{\lambda \in \Lambda_1}b_{\lambda}^{u,r,n}-\cal{N}_{r-1}(u,V_{\mu})||_1+2\sum_{\lambda \in \Lambda_1\cap \Lambda^c_2}p_{\lambda}^{u,r,n}+\sup_{\lambda\in \Lambda_1\cap \Lambda_2}||T^r_{\lambda}\Delta_{\lambda}^{\mathfrak{z},r,n}T^{r*}_{\lambda}(\phi^{0,r})-D^{{\mathfrak{z}}}(\phi^{0,r})||_1\\
 &+\sup_{\lambda\in \Lambda_1\cap \Lambda_2}||T^r_{\lambda}\Delta_{\lambda}^{\mathfrak{z},r,n}T^{r*}_{\lambda}T^r_{\lambda}(\rho^{0,u,r,n}_{\lambda})-T^r_{\lambda}\Delta_{\lambda}^{\mathfrak{z},r,n}T^{r*}_{\lambda}(\phi^{0,r})||_1\\
 \leq &||\sum_{\lambda \in \Lambda_1}b_{\lambda}^{u,r,n}-\cal{N}_{r-1}(u,V_{\mu})||_1+2\sum_{\lambda \in \Lambda_1\cap \Lambda^c_2}p_{\lambda}^{u,r,n}+\sup_{\lambda\in \Lambda_1\cap \Lambda_2}||(T^r_{\lambda}\Delta_{\lambda}^{\mathfrak{z},r,n}T^{r*}_{\lambda}-D^{\mathfrak{z}})(\phi^{0,r})||_1\\
 &+\sup_{\lambda\in \Lambda_1\cap \Lambda_2}||T^r_{\lambda}(\rho^{0,u,r,n}_{\lambda})-\phi^{0,r}||_1\\
\end{align*}
since $T^r_{\lambda}$ and $\Delta_{\lambda}^{\mathfrak{z},r,n}$ are isometries.  Using the integral form of $\phi^{0,r}$ (we set $\mathfrak{z}=0$ in (\ref{low_integral_form})), we obtain
\begin{align*}||(T^r_{\lambda}\Delta_{\lambda}^{\mathfrak{z},r,n}T^{r*}_{\lambda}-D^{{\mathfrak{z}}})(\phi^{0,r})||_1&\leq \int f(\bz)||(T^r_{\lambda}\Delta_{\lambda}^{\mathfrak{z},r,n}T^{r*}_{\lambda}-D^{{\mathfrak{z}}})(|\bz\ra\la \bz|\otimes |\mathbf{0}_2\ra\la \mathbf{0}_2|)||_1d\bz\\
 \leq &\int_{||\bz||>n^{\beta}}2f(\bz)d\bz\\&+\sup_{||\bz||\leq n^{\beta}} ||(T^r_{\lambda}\Delta_{\lambda}^{\mathfrak{z},r,n}T^{r*}_{\lambda}D^{z,0}-D^{(\mathfrak{z}_1+z,\mathfrak{z}_2)})(|\mathbf{0}\ra\la \mathbf{0}|)||_1
\end{align*}
where 
$$f(\bz)=\prod_{1\leq i<j \leq r}\frac{e^{\beta_{ij}}-1}{\pi} \exp\left(-\sum_{1\leq i<j \leq r}\frac{e^{\beta_{ij}}-1}{\pi}|z_{ij}|^2\right)$$
and
$$\int_{||\bz||>n^{\beta}}f(\bz)d\bz=c_1\exp(-c_2n^{2\beta})$$
by usual concentration of a Gaussian. The second term can be written as:
\begin{align*}&||(T^r_{\lambda}\Delta_{\lambda}^{\mathfrak{z},r,n}T^{r*}_{\lambda}D^{\bz,0}-D^{\mathfrak{z}_1+\bz,\mathfrak{z}_2})(|\mathbf{0}\ra\la \mathbf{0}|)||_1\\\\
\leq & ||(T^r_{\lambda}\Delta_{\lambda}^{(\mathfrak{z}_1+\bz,\mathfrak{z}_2),r,n}T^{r*}_{\lambda}-D^{(\mathfrak{z}_1+\bz,\mathfrak{z}_2)})(|\mathbf{0}\ra\la \mathbf{0}|)||_1\\\\
&+||(T^r_{\lambda}\Delta_{\lambda}^{(\mathfrak{z}_1+\bz,\mathfrak{z}_2),r,n}T^{r*}_{\lambda}-T^r_{\lambda}\Delta_{\lambda}^{\mathfrak{z},r,n}\Delta_{\lambda}^{\bz,r,n}T^{r*}_{\lambda})(|\mathbf{0}\ra\la \mathbf{0}|)||_1\\\\
&+||(T^r_{\lambda}\Delta_{\lambda}^{\mathfrak{z},r,n}T^{r*}_{\lambda}D^{\bz,0}-T^r_{\lambda}\Delta_{\lambda}^{\mathfrak{z},r,n}\Delta_{\lambda}^{\bz,r,n}T^{r*}_{\lambda})(|\mathbf{0}\ra\la \mathbf{0}|)||_1\\\\
\leq & ||(T^r_{\lambda}\Delta_{\lambda}^{(\mathfrak{z}_1+\bz,\mathfrak{z}_2),r,n}T^{r*}_{\lambda}-D^{(\mathfrak{z}_1+\bz,\mathfrak{z}_2)})(|\mathbf{0}\ra\la \mathbf{0}|)||_1+||(\Delta_{\lambda}^{(\mathfrak{z}_1+\bz,\mathfrak{z}_2),r,n}-\Delta_{\lambda}^{\mathfrak{z},r,n}\Delta_{\lambda}^{\bz,r,n})T^{r*}_{\lambda}(|\mathbf{0}\ra\la \mathbf{0}|)||_1\\\\
&+||(T^{r*}_{\lambda}D^{\bz,0}-\Delta_{\lambda}^{\bz,r,n}T^{r*}_{\lambda})(|\mathbf{0}\ra\la \mathbf{0}|)||_1\\\\
\leq&||(T^r_{\lambda}\Delta_{\lambda}^{(\mathfrak{z}_1+\bz,\mathfrak{z}_2),r,n}T^{r*}_{\lambda}-D^{(\mathfrak{z}_1+\bz,\mathfrak{z}_2)})(|\mathbf{0}\ra\la \mathbf{0}|)||_1+||(\Delta_{\lambda}^{(\mathfrak{z}_1+\bz,\mathfrak{z}_2),r,n}-\Delta_{\lambda}^{\mathfrak{z},n}\Delta_{\lambda}^{\bz,r,n})T^{r*}_{\lambda}(|\mathbf{0}\ra\la \mathbf{0}|)||_1\\\\
&+||(D^{\bz,0}-T^r_{\lambda}\Delta_{\lambda}^{\bz,r,n}T^{r*}_{\lambda})(|\mathbf{0}\ra\la \mathbf{0}|)||_1
\end{align*}

In the above equations $\Delta_{\lambda}^{\bz,r,n}$ is understood to be $\Delta_{\lambda}^{\mathfrak{z},r,n}$, where $\mathfrak{z}=(\bz,0)$. The entire chain of inequalities can be summarized as follows:
\begin{align*}
    &\sup_{\theta \in \Theta}||T_n(\rho^{\otimes n}_{\theta,n})-\cal{N}_{r-1}(u,V_{\mu})\otimes\phi^{\mathfrak{z},r}||_1 \\\\
    \leq &\sup_{\theta \in \Theta}||\sum_{\lambda \in \Lambda_1}b_{\lambda}^{u,r,n}-\cal{N}_{r-1}(u,V_{\mu})||_1\text{ (Approximation error of classical channel) }\\\\
    &+\sup_{\theta \in \Theta}2\sum_{\Lambda_1\cap \Lambda^c_2}p_{\lambda}^{u,r,n}\text{ (probability of extreme representations) }\\\\
    &+\sup_{\theta \in \Theta}\sup_{\lambda\in \Lambda_1\cap \Lambda_2}||T^r_{\lambda}(\rho^{0,u,r,n}_{\lambda})-\phi^{0,r}||_1\text{ (comparison of unshifted states) }\\\\
    &+\sup_{\theta \in \Theta,||\bz||\leq n^{\beta}}\sup_{\lambda\in \Lambda_1\cap \Lambda_2}||(T^r_{\lambda}\Delta_{\lambda}^{(\mathfrak{z}+\bz),r,n}T^{r*}_{\lambda}-D^{\mathfrak{z}+\bz})(|\mathbf{0}\ra\la \mathbf{0}|)||_1 \text{ (comparing shift operators) }\\\\
    &+\sup_{\theta \in \Theta,||\bz||\leq n^{\beta}}\sup_{\lambda\in \Lambda_1\cap \Lambda_2}||(D^{\bz,0}-T^r_{\lambda}\Delta_{\lambda}^{\bz,r,n}T^{r*}_{\lambda})(|\mathbf{0}\ra\la \mathbf{0}|)||_1\text{ (comparing shift operators) }\\\\
    &+\sup_{\theta \in \Theta,||\bz||\leq n^{\beta}}\sup_{\lambda\in \Lambda_1\cap \Lambda_2}||(\Delta_{\lambda}^{(\mathfrak{z}+\bz),r,n}-\Delta_{\lambda}^{\mathfrak{z},r,n}\Delta_{\lambda}^{\bz,r,n})T^{r*}_{\lambda}(|\mathbf{0}\ra\la \mathbf{0}|)||_1 \text{ (noncomm. of fin. dim. shifts) }\\\\
    &+ \int_{||\bz||>n^{\beta}}f(\bz)d\bz \text{ (concentration term) }.
\end{align*}
The following lemmas deals with each of the above terms separately:
\begin{lemma}{\label{approx_classical}}. For any given $\epsilon>0$,
$$\sup_{\theta \in \Theta}||\sum_{\lambda \in \Lambda_1}b_{\lambda}^{u,r,n}-\cal{N}_{r-1}(u,V_{\mu})||_1=O(n^{-1/4+\epsilon},n^{-1/2+\gamma}).$$
\end{lemma}
\begin{lemma}{\label{extreme_rep}}Recall that $\Lambda_2=\Lambda_{n,\alpha}$. Then for $\alpha-\gamma-1/2>0$ we have
$$\sum_{\lambda \in \Lambda_1\cap \Lambda^c_2}p_{\lambda}^{u,r,n}=O(n^{r^2}\exp(-n^{2\alpha-1}/2)).$$
\end{lemma}
\begin{lemma}{\label{comparison_unshifted}}
For $0<\eta<2/9$ we have
$$\sup_{\theta \in \Theta}\sup_{\lambda\in \Lambda_1\cap \Lambda_2}||T^r_{\lambda}(\rho^{0,u,r,n}_{\lambda})-\phi^{0,r}||_1=O(n^{-1/2+\gamma+\eta},n^{9\eta-2/24})$$
\end{lemma}
\begin{lemma}{\label{comparison_shift_operator}}
Let $\epsilon>0$ be such that $2\beta+\epsilon\leq \eta<2/9$, then we have
$$\sup_{\theta \in \Theta,||\bz||\leq n^{\beta}}\sup_{\lambda\in \Lambda_1\cap \Lambda_2}||(T^r_{\lambda}\Delta_{\lambda}^{\mathfrak{z}+\bz,r,n}T^{r*}_{\lambda}-D^{\mathfrak{z}+\bz})(|\mathbf{0}\ra\la \mathbf{0}|)||_1=R(n)$$
with
$$R^2(n) =  O\left(n^{(9\eta -2) / 12},n^{-1 + 2\beta + \eta}, n^{-1/2 + 3\beta+2\epsilon} , n^{-1 + \alpha +2 \beta}, n^{-1 + \alpha + \eta}, n^{-1 + 3\eta}, n^{-\epsilon}\right)$$
\end{lemma}
\begin{lemma}{\label{group_structure}}
Assume the condition given in Lemma \ref{comparison_shift_operator}, then the following result holds:
$$\sup_{\theta \in \Theta,||\bz||\leq n^{\beta}}\sup_{\lambda\in \Lambda_1\cap \Lambda_2}||(\Delta_{\lambda}^{\mathfrak{z}+\bz,r,n}-\Delta_{\lambda}^{\mathfrak{z},r,n}\Delta_{\lambda}^{\bz,r,n})T^{r*}_{\lambda}(|\mathbf{0}\ra\la \mathbf{0}|)||_1=R(n)$$
\end{lemma}
Combining these lemmas we obtain 
\begin{lemma}{\label{summarize_channel}}
For $0<\gamma<1/4$ and $0<\beta<1/9$ there exists $\kappa>0$ such that 
$$\sup_{\theta\in \Theta_{n,r,\beta,\gamma}}||\Phi^{\theta,r}-T^r_n(\rho^{\theta,r,n})||_1=O(n^{-\kappa})$$
\end{lemma}
\begin{proof}
For $\epsilon>0$, $1/2+\gamma<\alpha<1/2$, $\eta<2/9$ and $0<\beta<(\eta-\epsilon)/2$
we have 
\begin{multline}
\label{kappa_bound}
\sup_{\theta\in \Theta_{n,r,\beta,\gamma}}||\Phi^{\theta,r}-T^r_n(\rho^{\theta,r,n})||_1=O\left(n^{-1/4+\epsilon}+n^{-1/2+\gamma}+n^{-1/2+\gamma+\eta}+n^{9\eta-2/24}\right.\\
\left.+ n^{-1/4 + 3\beta/2+\epsilon} + n^{-1/2 + \alpha/2 + \eta/2}+ n^{-1/2 + 3\eta/2}+ n^{-\epsilon/2}\right)
\end{multline}
using Lemmas \ref{approx_classical} through \ref{group_structure}. Note that we have omitted the error term in Lemma \ref{extreme_rep} since it decays expontially upto a polynomial term and hence is smaller than the rest of the errors. Also we have used the relations $n^{-1/2+\beta+\eta/2}=O(n^{-1/2+3\eta/2})$ and $n^{-1/2+\alpha/2+\beta}=O(n^{-1/2+\alpha/2+\eta/2})$ to gather the error terms. Now we can easily choose $\epsilon,\alpha$ and $\eta$ suitably so that the RHS of (\ref{kappa_bound}) is $O(n^{-\kappa})$ for some $\kappa>0$.
\end{proof}
\textbf{Construction of the reverse channel $S^{r}_n$}

 We follow the proof given in \cite{Kahn&Guta} except that in our case we have a restricted class of Young diagrams.
 Let $\sigma^{r,n}$ be
defined by
\begin{equation}
\label{sigmalambda}
\sigma^{r,n}: x \in \mathbb{R}^{r-1} \mapsto \delta_{\lambda_x}  
\end{equation}
where $\lambda_x$ is the Young diagram such that $\sum_{i=1}^r \lambda_i = n$, 
and $ |n^{1/2}x_i+n\mu_i - \lambda_i| < 1/2 $, for 
$1\leq i\leq r-1$. If no such diagram exists, we set $\lambda_{x}$ to any admissible value, for example $(n,0,\dots,0)$. 
One can easily show that
with (\ref{taulambda}), $\sigma^{r,n} \circ\tau^{r,n}\circ \sigma^{r,n} = \sigma^{r,n}$ and for any probability distribution 
 on the $\lambda$ which is in the
image of $\sigma^{r,n}$ we have $\sigma^{r,n} \circ\tau^{r,n} (p^{u,r,n})= p^{u,r,n}$. 
One then shows (see lemma 6.8 of \cite{Kahn&Guta}) that for any $\epsilon$, 
we have 
\begin{equation}
\sup_{ \| u\| \leq n^{\gamma} }\left\|\sigma^n \mathcal{N}_{r-1}(u,V_{\mu}) -
p^{u,r,n}\right\|_1  =
 O\left( n^{-1/2+\epsilon} ,  n^{-1/4 + \gamma}\right).
\label{reverse_classical}
\end{equation}
The channel $S^r_n$ (which is essentially the pseudo-inverse of $T^r_{n}$) is given as follows. Given a sample from the probability distribution $\cal{N}_{r-1}(u,V_{\mu})$, we use the Markov kernel $\sigma^{r,n}$ to produce a Young diagram $\lambda$. Conditional on $\lambda$ we send the quantum part through the  channel 
\[
S^r_{\lambda}: \phi 
\mapsto 
\tilde{S}^r_{\lambda}(\phi)\otimes \frac{{\bf 1}_{\mathcal{K}_{\lambda}} }{\dim(\cal{K}_{\lambda})}
\]
with 
\[
\tilde{S}^r_{\lambda}:  \phi 
\mapsto 
T^{r*}_{\lambda} \phi  +
(1-\Tr(T^{r*}_{\lambda}(\phi))) |{\bf 0}_{\lambda}\rangle\langle{\bf 0}_ {\lambda}|.
\]
For any density operator $\rho_{\lambda}$ on the block 
$\lambda$, the operator $\tilde{S}^r_{\lambda}$ reverts the action of 
$T^r_{\lambda}$:
\begin{align*}
\tilde{S}^r_{\lambda} T^r_{\lambda}
(\rho_{\lambda})  & =
T^{r*}_{\lambda}T^r_{\lambda}(\rho_{\lambda}) + (1 -
\Tr(T^{r*}_{\lambda}T^r_{\lambda}(\rho_{\lambda})))|{\bf
0}_{\lambda}\rangle\langle{\bf 0}_{\lambda}| \\
& = \rho_{\lambda} +  (1 -
\Tr(\rho_{\lambda}))|{\bf
0}_{\lambda}\rangle\langle{\bf 0}_{\lambda}| \\
& = \rho_{\lambda}.
\end{align*} 
Now
$$
S^r_n( \mathcal{N}_{r-1}(u,V_{\mu}) \otimes\phi^{\mathfrak{z},r}) = 
\bigoplus_{\lambda \in \Lambda_1}
[\sigma^{r,n} \mathcal{N}_{r-1}(u,V_{\mu})](\lambda) \tilde{S}^r_{\lambda}(\phi^{\mathfrak{z},r})\otimes
\frac{{\bf 1}_{\mathcal{K}_{\lambda} }}{\dim(\cal{K}_{\lambda})}.
$$
and with the notation $\sigma^{r,n} \mathcal{N}^{u}_{\lambda} : =
[\sigma^{r,n}\mathcal{N}_{r-1}(u,V_{\mu}))](\lambda)$ and $q^{u,r,n}_{\lambda} :=
\min(\sigma^{r,n} \mathcal{N}^{u}_{\lambda}, p^{u,r,n}_{\lambda})$
we have 
\begin{align*}
&S_n(\phi^{\mathfrak{z},r}\otimes \mathcal{N}_{r-1}(u,V_{\mu})) - \rho^{\theta,r,n} 
\\
=&
\bigoplus_{\lambda\in \Lambda_1}
\left\{q^{u,r,n}_{\lambda}
(\tilde{S}^r_{\lambda}(\phi^{\mathfrak{z},r})-\rho^{\theta,n}_{\lambda}) + (\sigma^{r,n}
\mathcal{N}^{u}_{\lambda} -
q^{u,r,n}_{\lambda})\tilde{S}^r_{\lambda}(\phi^{\mathfrak{z},r})\right.\\
&-
\left.(p^{u,r,n}_{\lambda} - q^{u,r,n}_{\lambda})\rho^{\theta,r,n}_{\lambda}\right\}\otimes
\frac{{\bf 1}_{
\mathcal{K}_{\lambda}}}{\dim(\cal{K}_{\lambda})}.
\end{align*} 
Taking $L^1$ norms, and using that all $\phi$'s and $\rho$'s have trace $1$
and that channels (such as $\tilde{S}^r_{\lambda}$) are trace preserving, we
get the bound:
\begin{align*}
&\left\| S^r_n(\phi^{\mathfrak{z},r}\otimes \mathcal{N}_{r-1}(u,V_{\mu})) - \rho^{\theta,r,n}
\right\|_1\\
 \leq &\sum_{\lambda\in \Lambda_1} \left\|q^{u,r,n}_{\lambda}
(\tilde{S}^r_{\lambda}(\phi^{\mathfrak{z},r})-\rho^{\theta,r,n}_{\lambda})   \right\|_1 +
\sum_{\lambda\in \Lambda_1} \left| \sigma^{r,n}\mathcal{N}^{u}_{\lambda} - p^{u,r,n}_{\lambda}
\right|
\\
\leq & 2\sum_{\lambda\in \Lambda_1\cap\Lambda^c_2} q^{u,r,n}_{\lambda} +
\sup_{\lambda\in \Lambda_1\cap\Lambda_2}  \left\|
\tilde{S}^r_{\lambda}(\phi^{\mathfrak{z},r})-\rho^{\theta,n}_{\lambda}   \right\|_1 +
\left\|\sigma^{r,n} \mathcal{N}_{r-1}(u,V_{\mu}) -
p^{u,r,n}\right\|_1
\\
\leq & 2\sum_{\lambda\in \Lambda_1\cap\Lambda^c_2} p^{u,r,n}_{\lambda} +
\sup_{\lambda\in\Lambda_1\cap\Lambda_2}  \left\|
\phi^{\mathfrak{z},r} - T^r_{\lambda}(\rho^{\theta,n}_{\lambda})  \right\|_1 +
\left\|\sigma^{n} \mathcal{N}_{r-1}(u,V_{\mu}) -
p^{u,r,n}\right\|_1
. 
\end{align*}
Now the first term is $O(n^{r^2}\exp(-n^{2\alpha-1}/2))$ by Lemma \ref{extreme_rep} and the third term is treated in the relation (\ref{reverse_classical}). For the second term one uses Lemma \ref{comparison_unshifted} through \ref{group_structure} as in the proof of Lemma \ref{summarize_channel}. Thus we have the same bound for the reverse channel i.e.
\begin{equation}
   \sup_{\theta\in \Theta_{n,r,\beta,\gamma}}||S^r_n(\Phi^{\theta,r})-\rho^{\theta,r,n}||_1=O(n^{-\kappa})
   \end{equation}
\end{proof}
\begin{proof}[Proof of Theorem 5.1]
In the first step we use a $2$-design to construct a POVM which helps in constructing the least squares estimator. A $2$-design is a finite set of rank $1$ projectors $\{|v_i\ra\la v_i|\}_{i=1}^m$ which satisfies 
\[\frac{1}{m}\sum_{i=1}^m(|v_i\ra\la v_i|)^{\otimes 2}=\binom{d+1}{2}^{-1}P_{sym(2)}\]
where $P_{sym(2)}$ is the projection operator onto the symmetric subspace of $(\bb{C}^d)^{\otimes 2}$. It can be easily shown by taking partial trace that $\{\frac{d}{m}|v_i\ra\la v_i|\}_{i=1}^m$ forms a POVM. When the state $\rho$ is measured using this POVM, we obtain outcomes from the set $\{1,\ldots,m\}$ with the probability of observing $i$ given by $\frac{d}{m}\la v_i|\rho|v_i \ra$. Now repeat this measurement on $n$ states to get a set of outcomes $\{i_1,\ldots,i_n\}$. Denote $n_i$ to be the number of times $i$ is observed and $f_i=\frac{n_i}{n}$ - the relative frequency of observing $i$. The least squares estimator is defined as follows
\[\hat{L}_n=\argmin_{\rho} \sum_{i=1}^{m}\left(f_i-\frac{d}{m}\la v_i|\rho| v_i\ra\right)^2.\]
Let $\cal{M}$ be a linear map from $\cal{H}(\bb{C}^d)$ (the space of Hermitian operators from $\bb{C}^d$ to $\bb{C}^d$) defined as follows
\begin{align*}
    \cal{M}:&\cal{H}(\bb{C}^d) \rightarrow \bb{R}^m\\
    \cal{M}(X)&=\frac{d}{m}\{\la v_i|X| v_i\ra\}_{i=1}^m, \quad X\in \cal{H}(\bb{C}^d)
\end{align*}
The adjoint of $\cal{M}$ can be easily defined as follows 
\begin{align*}
    \cal{M}^*:& \bb{R}^m \rightarrow \cal{H}(\bb{C}^d)\\
    \cal{M}^*(u)&=\frac{d}{m}\sum_{i=1}^m u_i |v_i\ra\la v_i|, \quad u\in \bb{R}^m
\end{align*}
Then the least squares problem can restated as 
\[\hat{L}_n=\argmin_{\rho} ||f-\cal{M}(\rho)||^2\]
where $f=\{f_i\}_{i=1}^m$. The closed form of the solution is well known and is given by $$\hat{L}_n=(\cal{M}^*\cal{M})^{-1}\cal{M}^*(f).$$ Explicit computation yields (see \cite{MR4093475} for details)
\[\hat{L}_n=(d+1)\sum_{i=1}^mf_i|v_i\ra\la v_i|-I.\]
One uses matrix concentration inequalities to show that $\hat{L}_n$ concentrates around $\rho$ with high probability. Equation (6) in \cite{MR4093475} gives the following rate of concentration
\[P[||\hat{L}_n-\rho||\geq \epsilon]\leq de^{-\frac{3n\epsilon^2}{16d}}\]
where $||.||$ denotes the operator norm. However, $\hat{L}_n$ may not be a density matrix although it has trace 1. One projects $\hat{L}_n$ into the space of density matrices to obtain what is called a projected least squares estimator. However, as mentioned in the beginning we will take the route of \cite{ButuceaGutaKypraios} to project into the space of spectrally thresholded density matrices which would estimate the rank correctly. We define our estimator (called physical estimator in \cite{ButuceaGutaKypraios}) as follows 
\[\tilde{\rho}_n=\argmin_{\sigma\in S(\epsilon)}||\sigma-\hat{L}_n||_2^2\]
where $S(\epsilon)$ denotes the density matrices spectrally thresholded at $2\epsilon$ i.e.
\[S(\epsilon)=\{\sigma: \sigma \text{ is a density matrix with eigenvalues } \lambda_j\in \{0\} \cup (2\epsilon,1], j=1,2,\ldots,d\}  \]
We show that $\tilde{\rho}_n$ concentrates around $\rho$ and also satisfies the rank consistency condition. 

The proof follows in the same way as \cite{ButuceaGutaKypraios} but we include it for the sake of completion.

Let $\hat{\lambda}_1\geq\hat{\lambda_2}\geq\ldots\geq\hat{\lambda}_d$ denote the eigenvalues of $\hat{L}_n$. Let $k=0$ and define $\tilde{\lambda}_j^{(0)}=\hat{\lambda}_j$ for $j=1,2\ldots,d$
\begin{algorithmic}
\For {$k=1,2,\ldots,d$}
\{\If {$\tilde{\lambda}_{d-k+1}^{(k-1)}> 2\epsilon$}
    STOP
\Else
    \State $\tilde{\lambda}_{d-k+1}^{(k)}=0$
    \State $\tilde{\lambda}_{j}^{(k)}=\tilde{\lambda}_{j}^{(k-1)}+\frac{\tilde{\lambda}_{d-k+1}^{(k-1)}}{d-k}$ for $j=1,2,\ldots,d-k$
\EndIf
\EndFor
\end{algorithmic}
Note that if the algorithm stops at $k=\hat{k}$, then $\hat{r}=d-\hat{k}+1$ is the rank of the estimator $\tilde{\rho}_n$ and the eigenvalues are given by $\tilde{\lambda}_j=0$ for $j>\hat{r}$ and for $j\leq \hat{r}$
\[\tilde{\lambda}_j=\hat{\lambda}_j+\frac{1}{\hat{r}}\sum_{j>\hat{r}}\hat{\lambda}_j.\]
In particular if  $\hat{L}_n=\hat{U}_n\hat{D}_n\hat{U}^{*}_n$, then $\tilde{\rho}_n=\hat{U}_n\tilde{D}_n\hat{U}^{*}_n$ where $$\tilde{D}_n=diag(\tilde{\lambda}_1,\ldots,\tilde{\lambda}_{\hat{r}},0,\ldots,0).$$  
Henceforth, we condition on the event  $||\hat{L}_n-\rho||\leq \epsilon$ which occurs with probability greater than $1-de^{-\frac{3n\epsilon^2}{16d}}$. By Weyl perturbation inequality (see \cite{MR1477662}), $|\lambda_j-\hat{\lambda}_j|\leq \epsilon$ where $\lambda_1\geq \lambda_2 \geq \ldots \geq \lambda_{r}>6\epsilon$ (by assumption) and $\lambda_{r+1}=\ldots \lambda_{d}=0$. Now if $\hat{r}>r$, then for $j\leq \hat{r}$,
\begin{align*}
    |\tilde{\lambda}_j-\lambda_j|&=|\hat{\lambda}_j+\frac{1}{\hat{r}}\sum_{j>\hat{r}}\hat{\lambda}_j-\lambda_j|\\
    &\leq|\hat{\lambda}_j-\lambda_j|+\frac{1}{\hat{r}}|1-\sum_{j\leq\hat{r}}\hat{\lambda}_j|\\
    &\leq|\hat{\lambda}_j-\lambda_j|+\frac{1}{\hat{r}}\sum_{j\leq\hat{r}}|\lambda_j-\hat{\lambda}_j|\\
    &\leq 2\epsilon
\end{align*}
Since $\lambda_{\hat{r}}=0$ when $\hat{r}>r$, the above inequality implies $\tilde{\lambda}_{\hat{r}}\leq\lambda_{\hat{r}}+2\epsilon=2\epsilon$, which is a contradiction since $\hat{r}$ was chosen so that $\tilde{\lambda}_{\hat{r}}>2\epsilon$.
Now consider the case $\hat{r}<r$, so that $\tilde{\lambda}_r=0$
\begin{equation}0=\tilde{\lambda}_r\geq \lambda_r-|\lambda_r-\tilde{\lambda}_r|\geq \lambda_r-|\lambda_r-\hat{\lambda}_r|-|\hat{\lambda}_r-\tilde{\lambda}_r|=\lambda_r-|\lambda_r-\hat{\lambda}_r|-|\hat{\lambda}_r|\label{eigen-ineq}
\end{equation}
 Now consider the step $k$ of the algorithm with $k=d-r+1$. Since $\tilde{\lambda}_r=\tilde{\lambda}^{(d-r+1)}_r=0$, $\tilde{\lambda}^{(d-r)}_r<2\epsilon$. A careful analysis of the algorithm shows that at after step $d-r$, the sum of the last $d-r$ eigenvalues is distributed equally among the first $r$ eigenvalues. Thus 
 \[\tilde{\lambda}^{(d-r)}_r=\hat{\lambda}_r+\frac{1}{r}(1-\sum_{i=1}^r\hat{\lambda}_r)=\hat{\lambda}_r+\frac{1}{r}\sum_{i=1}^r(\lambda_r-\hat{\lambda}_r)<2\epsilon.\]
 This implies that $\hat{\lambda}_r<2\epsilon+\frac{1}{r}\sum_{i=1}^r|\lambda_r-\hat{\lambda}_r|<3\epsilon$
 .Since by the condition of the theorem $\lambda_r>6\epsilon$, (\ref{eigen-ineq}) leads to
 \[0=\tilde{\lambda}_r>6\epsilon-\epsilon-3\epsilon=2\epsilon\]
a contradiction. Thus with probability at least $1-de^{-\frac{3n\epsilon^2}{16d}}$, $r=\hat{r}$.
 
 To describe the concentration $\tilde{\rho}_n$, consider the following diagonalizations of the states 
 \[\rho=UDU^*,\quad \hat{L}_n=\hat{U}_n\hat{D}_n\hat{U}_n^*,\quad \tilde{\rho}_n=\hat{U}_n\tilde{D}_n\hat{U}_n^*\]
 and estimate the error of the projected estimator as follows:
\begin{align*}
    ||\tilde{\rho}_n-\rho||_2 & = ||\hat{U}_n\tilde{D}_n\hat{U}_n^*-UDU^*||_2\\
    & \leq ||\hat{U}_n\tilde{D}_n\hat{U}_n^*-\hat{U}_nD\hat{U}_n^*||_2+||\hat{U}_nD\hat{U}_n^*-UDU^*||_2\\
    & = ||\tilde{D}_n-D||_2+||\hat{U}_nD\hat{U}_n^*-UDU^*||_2.
\end{align*}
For the first norm one has the following bound
\[||\tilde{D}_n-D||_2=\sqrt{\sum_{k\leq\hat{r}}(\tilde{\lambda}_j-\lambda_j)^2}\leq 2\epsilon\sqrt{r}.\]
For the second norm first we give an upper bound in terms of the operator norm
\begin{align*}||\hat{U}_nD\hat{U}_n^*-UDU^*||&\leq ||\hat{U}_n\hat{D}_n\hat{U}_n^*-\hat{U}_nD\hat{U}_n^*||+||\hat{U}_n\hat{D}_n\hat{U}_n^*-UDU||\\
& = ||\hat{D}_n-D||+||\hat{L}_n-\rho||\leq 2\epsilon.
\end{align*}
Both $\hat{U}_nD\hat{U}_n^*$ and $UDU^*$ are rank $r$ matrices, so rank of $\hat{U}_nD\hat{U}_n^*-UDU^*$ is at most $2r$ and hence $||\hat{U}_nD\hat{U}_n^*-UDU^*||_2\leq 2\epsilon\sqrt{2r}$. And thus we obtain
\[||\tilde{\rho}_n-\rho||_2^2\leq (2\sqrt{2}+2)^2r\epsilon^2\leq 25r\epsilon^2\]
with probability at least $1-de^{-\frac{3n\epsilon^2}{16d}}$.
\end{proof}
\begin{proof}[Proof of Theorem 5.2]

\textbf{Upper Bound}

 With $Z\sim\cal{N}_{r-1}(u,V_{\mu})$, as the estimator $\hat{u}$
 the classical part of the risk is 
\[R^C=E[\sum_{i=1}^{r-1}(u_i-Z_i)^2+(\sum_{i=1}^{r-1}(u_i-Z_i))^2]=\sum_{i=1}^r\mu_i(1-\mu_i)\]
where we have used the relation $\mu_r=1-\sum_{i=1}^{r-1}\mu_i$.

 We start to compute the risk for the quantum part by first verifying that we obtain normal random variables when we use the measurement $M$ on the shifted thermal states. Denote $R_t=(-\infty,t_1]\times (-\infty,t_2]$. Let $X$ denote the two dimensional random variable generated when using measurement $M$ to estimate the state $\phi^{z_0}_{\beta}$. Then using the relation $P(X\in R_t) =\Tr(\phi^{z_0}_{\beta}M(R_{t}))$ we obtain,
\[
P(X\in R_t) 
 = \Tr\left( \frac{e^{\beta}-1}{\pi}\int\exp{-(e^{\beta}-1)|z|^2})|z+z_0\ra\la z+z_0|dz\frac{1}{2\pi}\int_{R_t}|\psi_m\ra\la\psi_m|dm\right).
\]
 Changing from complex to two dimensional reals in the second integrals by using $\mu=(\sqrt{2}Re(z),\\\sqrt{2}Im(z))$ and similarly obtaining $\mu_0$ from $z_0$ we have
 \begin{align*}
   P(X\in R_t)&=  \Tr\left( \frac{e^{\beta}-1}{2\pi}\int\exp(-(e^{\beta}-1)||\mu||^2/2)|\psi_{\mu+\mu_0}\ra\la \psi_{\mu+\mu_0}|d\mu \frac{1}{2\pi}\int_{R_t}|\psi_m\ra\la\psi_m|dm\right)\\
   & = \frac{e^{\beta}-1}{4\pi^2}\int_{R_t}\int\exp(-(e^{\beta}-1)||\mu||^2/2)|\la\psi_{\mu+\mu_0}|\psi_m\ra|^2d\mu dm\\
   & = \frac{e^{\beta}-1}{4\pi^2}\int_{R_t}\int\exp(-(e^{\beta}-1)||\mu||^2/2)\exp(-\frac{1}{2}||\mu+\mu_0-m||^2)d\mu dm\\
   &=  \frac{e^{\beta}-1}{4\pi^2}\int_{R_t}\int\exp(-\frac{1}{2}(e^{\beta}||\mu||^2-2\mu^T(m-\mu_0)+||m-\mu_0||^2))d\mu dm\\
   & = \frac{1-e^{-\beta}}{2\pi}\int_{R_t}\exp(-\frac{1}{2}(1-e^{-\beta})||m-\mu_0||^2)dm.
 \end{align*}
where the third equality is obtained by using the expansion (7). Thus the generated random variable $X$ is a bivariate random variable with mean $\mu_0$ and variance $\frac{1}{1-e^{-\beta}}I_2=\frac{2\sigma^2_{\beta}+1}{2}I_2$ (since $\sigma^2_{\beta}=\frac{\coth(\beta/2)}{2}$). A similar argument would yield that $M_c$ (see Section 4) generates $cX$. As a special case when $\beta=\infty$,  $\phi_{\infty}=|0\ra\la0|$ and thus the shifted thermal state at infinite temperature is just a shifted pure state. In this case the covariance matrix  of $X$ is $I_2$.

Now consider the quantum part i.e.
$$\phi^{\mathbf{\mathfrak{z}},r}=\bigotimes_{1\leq i<j \leq r}\phi_{\beta_{ij}}^{\mathfrak{z}_{ij}}\otimes\bigotimes_{\substack{1\leq i\leq r\\ r+1\leq j \leq d}}\phi_{\infty}^{\mathfrak{z}_{ij}}$$
 and estimation of the parameter $\{\mathfrak{z}_{ij}\}$ under the constraint that $|\mathfrak{z}_{ij}|\leq n^{\beta}$. Changing from complex to two dimensional reals (in this case write $\xi_{ij}=(\sqrt{2}Re(\mathfrak{z}_{ij}),\sqrt{2}Im(\mathfrak{z}_{ij}))$) we have the constraint set:
\[\mathcal{Z}=\{\{\xi_{ij}\}:|\xi_{ij}|\leq \sqrt{2}n^{\beta}\}.\]

Let $\kappa_{ij}=\mu_i-\mu_j$. Note that for the case $j\geq r+1$, $\sigma^2_{\beta_{ij}}=1/2$ (since the state is pure or alternatively since $\beta_{ij}=\ln(\mu_i/\mu_j)$ and $\mu_j=0$ implies that $\beta_{ij}=\infty$). Also reparametrize the Gaussian state $\phi^{\mathbf{\mathfrak{z}},r}$ as $\phi^{\mathbf{\xi},r}$. The risk for estimation with the POVM $\bar{M}$ where $\bar{M}=\bigotimes_{\substack{{1\leq i\leq r}\\{i<j\leq d}}}M$ is given by: 
\begin{align*}R^Q_{\bar{M}}&=\int \sum_{\substack{{1\leq i\leq r}\\{i<j\leq d}}}\kappa_{ij}||\xi_{ij}-\hat{\xi}_{ij}||^2 Tr(\phi^{\mathbf{\xi},r}\bar{M}(d \hat{\xi}))\\
 &= \sum_{\substack{{1\leq i\leq r}\\{i<j\leq d}}}\kappa_{ij}(2\sigma_{\beta_{ij}}^2+1).
 \end{align*}
Thus when $\theta\in \Theta=\{|u_i|\leq n^\gamma \text{ }\forall i,|\mathfrak{z}_{kj}|\leq n^{\beta} \text{ }\forall 1\leq i\leq r, i<j\leq d\}$, the risk for the classical estimator $Z$ and the POVM $\bar{M}$ is given by
\begin{equation}
    R_{Z,\bar{M}}(\theta)= \sum_{i=1}^r \mu_i(1-\mu_i)+ \sum_{\substack{{1\leq i\leq r}\\{i<j\leq d}}}\kappa_{ij}(2\sigma_{\beta_{ij}}^2+1).\label{const_risk}
\end{equation}
Note that $2\sigma^2_{\beta_{ij}}+1=2(1-e^{-\beta_{ij}})^{-1}=2\mu_i/(\mu_i-\mu_j)$ for $j\leq r$ and $2\sigma^2_{\beta_{ij}}+1=2$ for $j>r$. Plugging back the expression for $\kappa_{ij}$ we obtain
\begin{equation}\label{upper_bound_final}
    \limsup_{n\rightarrow \infty}\sup_{\theta\in\Theta}R_{Z,\bar{M}}(\theta)= \sum_{i=1}^r \mu_i(1-\mu_i)+\sum_{\substack{{1\leq i\leq r}\\{i<j\leq d}}}2\mu_i.
\end{equation}
since the risk in (\ref{const_risk}) is independent of $\theta$ and $n$.

\textbf{Lower Bound}

Consider the expression of the Bayes risk for estimating $\theta$ under the following prior $\pi$
\begin{align*}u_i\sim &\cal{N}(0,n^{2\gamma_{0}}) \text{ }\forall 1\leq i\leq r-1\\
\xi_{ij}\sim& \cal{N}_2(0,n^{2\beta_{0}}I_2 )\text{ }\forall 1\leq i\leq r, i<j\leq d\
\end{align*}

where the random variables are mutually independent and $\gamma_0<\gamma$ and $\beta_0<\beta$. Again we change the notation from complex to the real i.e. we set $\xi_{ij}=(\sqrt{2}Re(\mathfrak{z}_{ij}),\sqrt{2}Im(\mathfrak{z}_{ij}))$. Since both the loss and the parameter space are separable into quantum and classical parts we introduce the following notations
\begin{align*}
 \mathcal{L}(u,\hat{u})&= \sum_{i=1}^{r-1}(u_i-\hat{u}_i)^2+(\sum_{i=1}^{r-1}(u_i-\hat{u}_i))^2\\
 \mathcal{L}(\xi,\hat{\xi})&=\sum_{\substack{{1\leq i\leq r}\\{i<j\leq d}}}(\mu_i-\mu_j)||\xi_{ij}-\hat{\xi}_{ij}||^2\\
 \cal{U}&=\{u:|u_i|\leq n^\gamma \text{ }\forall 1\leq i\leq r-1\}\\
 \cal{Z}&=\{\xi:|\xi_{ij}|\leq n^\beta \text{ }\forall 1\leq i\leq r, i<j\leq d\}\\
\end{align*}
We denote $\pi_1$ as the prior on $u$ and $\pi_2$ the prior on $\xi$ and by the structure of the prior we have $\pi(d\theta)=\pi_1(du)\pi_2(d\xi)$.We give a lower bound to the maximum risk by the average risk as follows:
\begin{align}
    \sup_{\theta\in\Theta} R_{\hat{u},M}(\theta)&\geq \int_{\Theta}(R_{\hat{u}}^C(\theta)+R^Q_{M}(\theta))d\pi(\theta)\nonumber\\
    &\geq\pi_2(\cal{Z})\inf_{\hat{u}\in \cal{U}}\int_\cal{U}E_u[\cal{L}(u,\hat{u})]d\pi_1(u)+\pi_1(\cal{U})\inf_{M:M(\cal{Z}^c=0)}\int_\cal{Z}E_{\xi}[\cal{L}(\xi,\hat{\xi})]d\pi_2(\xi)\nonumber\\
    &=\pi_2(\cal{Z})\inf_{\hat{u}\in \cal{U}}(\int E_u[\cal{L}(u,\hat{u})]d\pi_1(u)-\int_{\cal{U}^c
    }E_u[\cal{L}(u,\hat{u})]d\pi_1(u))\nonumber\\&+\pi_1(\cal{U})\inf_{M:M(\cal{Z}^c=0)}(\int E_{\xi}[\cal{L}(\xi,\hat{\xi})]d\pi_2(\xi)-\int_{\cal{Z}^c}E_{\xi}[\cal{L}(\xi,\hat{\xi})]d\pi_2(\xi)).\label{lower_bound_intermed}
\end{align}
Since the parameters lie in $\cal{U}$ and $\cal{Z}$ respectively it is enough to consider estimators supported on the parameter space themselves and hence the second inequality follows. 

Let $Y\sim \cal{N}_d(u,\Sigma_1)$ and $u\sim \cal{N}_d(0,\Sigma_2)$, then $E[u|Y]$ is known to be the Bayes estimator of $u$ under the loss $||u-\hat{u}||^2$. It can be easily shown that $E[u|Y]=(I+\Sigma_1\Sigma^{-1}_2)^{-1}Y$ and  the Bayes risk is given by
$$E[||u-E[u|Y]||^2]=\mathrm{Tr}((\Sigma^{-1}_1+\Sigma^{-1}_2)^{-1}).$$
Using the above result we observe that
\begin{align*}
    \inf_{\hat{u}\in\cal{U}}\int E_u[\cal{L}(u,\hat{u})]d\pi_1(u)\geq & \inf_{\hat{u}} \left(\int E_u[\sum_{i=1}^{r-1}(u_i-\hat{u}_i)^2]d\pi_1(u)+\int E_u[(\sum_{i=1}^{r-1}(u_i-\hat{u}_i))^2]d\pi_1(u)\right)\\
    & \geq \inf_{\hat{u}} \int E_u[\sum_{i=1}^{r-1}(u_i-\hat{u}_i)^2]d\pi_1(u)+\inf_{w}\int E_u[(\sum_{i=1}^{r-1}u_i-w)^2]d\pi_1(u)\\
    &= \mathrm{Tr}((\Sigma_1^{-1}+\Sigma_2^{-1})^{-1})+\frac{\sigma_r^2\tau_r^2}{\sigma_r^2+\tau_r^2}.
\end{align*}
Here $\Sigma_1=V(\mu)$, $\Sigma_2=n^{2\gamma_0}I_{r-1}$, $\tau_r^2=(r-1)n^{2\gamma_0}$ and
\[\sigma_r^2=\sum_{i=1}^{r-1}\mu_i(1-\mu_i)-\sum_{i\neq  j}\mu_i\mu_j=\mu_r(1-\mu_r).\]
Taking limit one obtains
\begin{equation}\liminf_{n\rightarrow\infty}\inf_{\hat{u}\in \cal{U}}\int E_u[\cal{L}(u,\hat{u})]d\pi_1(u)\geq \mathrm{Tr}\;V(\mu)+\sigma_r^2=\sum_{i=1}^{r}\mu_i(1-\mu_i).\label{lower_bound_class}
\end{equation}
For the quantum part the Bayes risk for a single component $\xi_{ij}$ with the prior $\xi_{ij}\sim \sigma^2 I_2$ is obtained by the covariant measurement followed by an appropriate shrinkage (see section 4 for details) and the Bayes risk is given by
\[\inf_M\int\int||\xi_{ij}-\hat{\xi_{ij}}||^2\mathrm{Tr}\left(\phi^{\xi_{ij}}_{\beta_{ij}}M(d\hat{\xi_{ij}})\right)d\pi(\xi_{ij})=\frac{2\sigma^2(2\sigma^2_{\beta_{ij}}+1)}{2(\sigma^2+\sigma^2_{\beta_{ij}})+1}\]
which is essentially equation (28). Thus we have 
\begin{align*} \inf_{M:M(\cal{Z}^c=0)}\int E_{\xi}[\cal{L}(\xi,\hat{\xi})]d\pi_2(\xi)\geq&\inf_M\int(\sum_{\substack{{1\leq i\leq r}\\{i<j\leq d}}}\kappa_{ij}||\xi_{ij}-\hat{\xi_{ij}}||^2\mathrm{Tr}[\phi^{\xi_{ij}}_{\beta_{ij}}M(d\hat{\xi_{ij}})])d\pi_2(\xi)\\
\geq &\sum_{\substack{{1\leq i\leq r}\\{i<j\leq d}}}\kappa_{ij}\frac{2n^{2\beta_0}(2\sigma^2_{\beta_{ij}}+1)}{2(n^{2\beta_0}+\sigma^2_{\beta_{ij}})+1}.
\end{align*}
Taking limit we obtain
\begin{equation}\liminf_{n\rightarrow \infty}\inf_{M:M(\cal{Z}^c=0)}\int E_{\xi}[\cal{L}(\xi,\hat{\xi})]d\pi_2(\xi)\geq \sum_{\substack{{1\leq i\leq r}\\{i<j\leq d}}}\kappa_{ij}(2\sigma^2_{\beta_{ij}}+1)=\sum_{\substack{{1\leq i\leq r}\\{i<j\leq d}}}2\mu_i\label{lower_bound_quantum}
\end{equation}
the last equality following from the calculation preceding (\ref{upper_bound_final}). We also have the following lemma (the proof of which is deferred to Appendix D) which shows that the second and fourth term of (\ref{lower_bound_intermed}) goes to zero.
\begin{lemma}{\label{lower_bound_error}}
Let $\hat{u}$ be an estimator of $u$ and suppose $\hat{u}\in\cal{U}$. Similarly suppose we use the POVM $M$ to estimate $\xi$ where $M(\cal{Z}^c)=0$. Then we have
\[\int_{\cal{U}^c
    }E_u[\cal{L}(u,\hat{u})]d\pi_1(u),\int_{\cal{Z}^c}E_{\xi}[\cal{L}(\xi,\hat{\xi})]d\pi_2(\xi)\rightarrow 0.\]
\end{lemma} Finally we observe that $\pi_1(\mathcal{U}),\pi_2(\mathcal{Z})\rightarrow 1$, which together with (\ref{lower_bound_class}),(\ref{lower_bound_quantum}) and Lemma \ref{lower_bound_error} yields
\[\liminf_{n\rightarrow \infty}\inf_{\hat{u},M}\sup_{\theta\in\Theta} R_{\hat{u},M}(\theta)\geq\sum_{i=1}^r \mu_i(1-\mu_i)+\sum_{\substack{{1\leq i\leq r}\\{i<j\leq d}}}2\mu_i\]
when the terms are plugged in (\ref{lower_bound_intermed}).
\end{proof}

\begin{proof}[Proof of Theorem 5.4]

\textbf{Upper Bound}

Split the sample into two parts $n_1=\lfloor n^{\delta_0}\rfloor+1$ for some $0<\delta_0<1$ (we will describe the choice of $\delta_0$ in the next step) and $n_2=n-n_1$. Let $\tilde{\rho}_n$ be the estimator obtained in the first stage with the the sample $n_1$. Set $25r\epsilon^2=n^{-1+2\epsilon_0}$ (where $\epsilon_0<\gamma$). By Theorem 5.1
\begin{equation}P[||\rho-\tilde{\rho}_{n}||^2_2\geq n^{-1+2\epsilon_0}]\leq de^{-3n^{-1+2\epsilon_0+\delta_0}/400rd}.
\label{state_conc}\end{equation}
since we have obtained $\tilde{\rho}_{n}$ using $n_1>n^{\delta_0}$. For any $\epsilon_0$ we can choose $\delta_0$ such that $2\epsilon_0+\delta_0>1$ and the concentration still holds. 
Further we have
\[P[r=\hat{r}]\geq 1- de^{-3n^{-1+2\epsilon_0+\delta_0}/400rd}.\]
Now for the event $A_n=\{||\rho-\tilde{\rho}_{n}||^2_2< n^{-1+2\epsilon_0} \text{ and  } r=\hat{r}\}$ we have, 
\begin{equation}
  P[A_n]=1-2de^{-3n^{-1+2\epsilon_0+\delta_0}/400rd}.\label{prob_A_n}  
\end{equation}
Recall the definition $\mathcal{L}(\theta,\hat{\theta})$ given in (31).  The risk can now be decomposed as follows
\begin{align*}
E_{\rho}||\rho-\hat{\rho}_n||_2^2\leq &4P[A^c_n]+E_{\rho}[||\rho-\hat{\rho}_n||^2\mathds{1}_{A_n}]\\
    = & 4P[A^c_n]+E_{\rho_{\theta}}\left[\left(\frac{1}{n}\mathcal{L}(\theta,\hat{\theta}_n)+O\left(\frac{||\theta||^3,||\hat{\theta}_n||^3}{n^{3/2}}\right)\right)\mathds{1}_{A_n}\right].
\end{align*}
We have used the notation involving local parameters in the last line with $\rho_0=\tilde{\rho}_n$ as the central state. Also the local parameter estimate $\hat{\theta}_n$(and the corresponding state $\hat{\rho}_n$) is formed using the second part of the sample $n_2$. Conditioned on $A_n$, $||\theta||<n^{\epsilon_0}$ and we can restrict $\hat{\theta}_n$ to lie within a radius of $2n^{\epsilon_0}$. This ensures that  $||\theta||<n^{\epsilon_0}<n^{\gamma}$ (choose $\beta<\gamma<1/9$) we can use LAN. Secondly both $||\theta||^3$ and $||\hat{\theta}_n||^3$ are less than $n^{3\gamma}<n^{1/2-\epsilon_1}$ for a sufficiently small $\epsilon_1$ so that $O(\frac{||\theta||^3,||\hat{\theta}_n||^3}{n^{3/2}})=o(n^{-1})$. Thus we obtain 
$$E_{\rho}n||\rho-\hat{\rho}_n||_2^2\leq 4nP[A^c_n]+E_{\rho_{\theta}}\left[\left(\mathcal{L}(\theta,\hat{\theta})\right)\mathds{1}_{A_n}\right]+o(1).$$
Since $\mathcal{L}(\theta,\hat{\theta})=O(n^{2\epsilon_0})$ multiplying with the asymptotic equivalence rate $n^{-\kappa}$ we have $n^{2\epsilon_0-\kappa}=o(1)$ (we also choose $\epsilon_0<\kappa/2$). Finally we use relations (21) and (22) to obtain 
$$E_{\rho_{\theta}}\left[E_{\rho_{\theta}}\left[\mathcal{L}(\theta,\hat{\theta}_n)|\mathds{1}_{A_n}\right]\right]\leq E_{\rho_{\theta}}\left[E_{\Phi^{\theta,r}}\left[\mathcal{L}(\theta,\hat{\theta}_n)|\mathds{1}_{A_n}\right]\right]+o(1).$$
Note that in this case $\Phi^{\theta,r}$ depends on the center $\tilde{\rho}_{n}$ (obtained in the first step). Using relation (\ref{prob_A_n}) we observe that  $nP[A^c_n]=o(1)$ and hence
\begin{equation}E_{\rho}n||\rho-\hat{\rho}_n||_2^2\leq E_{\rho_{\theta}}\left[E_{\Phi^{\theta,r}}\left[\mathcal{L}(\theta,\hat{\theta}_n)|\mathds{1}_{A_n}\right]\right]+o(1).
\label{cond_exp}
\end{equation}
Since we have used $\tilde{\rho}_n$ as the centralizing state, if $\tilde{\mu}_1>\ldots>\tilde{\mu}_r$
are the eigenvalues of $\tilde{\rho}_n$, then with $\theta\in \Theta$ we have
$$E_{\Phi^{\theta,r}}\left[\mathcal{L}(\theta,\hat{\theta}_n)|\mathds{1}_{A_n}\right]=\sum_{i=1}^r \tilde{\mu}_i(1-\tilde{\mu}_i)+\sum_{\substack{{1\leq i\leq r}\\{i<j\leq d}}}2\tilde{\mu}_i$$
by equation (\ref{const_risk}). Although we have used the second part of the sample $n_2$ to form the estimate $\hat{\theta}_n$, since $n_2=n(1+o(1))$ the same asymptotics hold.

Finally let $\sigma$ be a rank-$r$ qudit with eigenvalues $\mu_1>\ldots>\mu_r$, then if $||\rho-\sigma||_2<n^{-1/2+\varepsilon}$ for some $\varepsilon<\epsilon_0$ then $||\tilde{\rho}_n-\sigma||_2\leq n^{-1/2+\epsilon_0}$ with high probability (see RHS of (\ref{state_conc})). By the equivalence of norms in the finite dimension the same claim can be made with the operator norm i.e. 
\[P[||\sigma-\tilde{\rho}_{n}||\geq cn^{-1/2+\varepsilon}]\leq de^{-3n^{-1+2\epsilon_0+\delta_0}/400rd}\]
for some constant c. Using the above relation and Weyl perturbation inequality one can easily show that $\tilde{\mu}_i=\mu_i+o_p(1)$ for all $i=1,\ldots,r-1$ and since $\tilde{\mu}_i$ are bounded we have
\begin{equation}\limsup_n\sup_{||\rho-\sigma||_2<n^{-1/2+\varepsilon}}E_{\rho_{\theta}}\left[E_{\Phi^{\theta,r}}\left[\mathcal{L}(\theta,\hat{\theta}_n)|\mathds{1}_{A_n}\right]\right]=\sum_{i=1}^r \mu_i(1-\mu_i)+\sum_{\substack{{1\leq i\leq r}\\{i<j\leq d}}}2\mu_i.
\label{cont_eigen}
\end{equation}
Finally using relations (\ref{cond_exp}) and (\ref{cont_eigen}) we obtain
$$\limsup_n\sup_{||\rho-\sigma||_2<n^{-1/2+\varepsilon}}E_{\rho}n||\rho-\hat{\rho}_n||_2^2\leq \sum_{i=1}^r \mu_i(1-\mu_i)+\sum_{\substack{{1\leq i\leq r}\\{i<j\leq d}}}2\mu_i.$$

\textbf{Lower Bound}

Again we observe that for $||\rho-\sigma||<n^{-1/2+\varepsilon}$ we can use the local parametrization and write $\rho$ and $\hat{\rho}_n$ as $\rho_\theta$ and $\rho_{\theta_n}$ respectively with the central state now being $\sigma$. We have the following set of inequalities 
\begin{align*}
    \inf_{\hat{\rho}_n}\sup_{||\rho-\sigma||<n^{-1/2+\varepsilon}}E_{\rho}n||\rho-\hat{\rho}_n||_2^2\geq& \inf_{\hat{\theta}_n}\sup_{||\theta||<n^{\varepsilon}}\mathcal{L}(\theta,\hat{\theta})+O\left(\frac{||\theta||^3,||\hat{\theta}_n||^3}{n^{1/2}}\right)\\
    \geq & \inf_{||\hat{\theta}||<n^{\varepsilon}}\sup_{||\theta||<n^{\varepsilon}}E_{\rho_{\theta}}\left[\mathcal{L}(\theta,\hat{\theta})+O\left(\frac{||\theta||^3,||\hat{\theta}_n||^3}{n^{1/2}}\right)\right]\\
    \geq & \inf_{||\hat{\theta}||<n^{\varepsilon}}\sup_{||\theta||<n^{\varepsilon}}E_{\rho_{\theta}}\left[\mathcal{L}(\theta,\hat{\theta})\right]-o(1)\\
    \geq & \inf_{||\hat{\theta}||<n^{\varepsilon}}\sup_{||\theta||<n^{\varepsilon}}E_{\Phi^{\theta,r}}\left[\mathcal{L}(\theta,\hat{\theta})\right]-o(1)\\
    \geq & \inf_{\hat{\theta}}\sup_{\theta\in \Theta_{n,\beta_1,\gamma_1}}E_{\Phi^{\theta,r}}\left[\mathcal{L}(\theta,\hat{\theta})\right]-o(1)\\
    \geq & \sum_{i=1}^r \mu_i(1-\mu_i)+\sum_{\substack{1\leq i\leq r\\i<j\leq d}}2\mu_i-o(1).
\end{align*}
The first inequality follows since $||\theta||<n^{\varepsilon}$ implies $||\rho-\sigma||_2<n^{-1/2+\varepsilon}$. Since $||\theta||<n^{\varepsilon}$ it is enough to consider $\hat{\theta}_n$ with $||\hat{\theta}||_n<n^{\varepsilon}$ hence the second inequality follows. For the third inequality we observe that one can choose $\varepsilon$ small enough so that $3\varepsilon<1/2-\epsilon_1$ for some $\epsilon_1$ and hence for the given restriction of $\theta$ and $\hat{\theta}_n$ the remainder term is $o(1)$. Also for the same restriction of $\theta$ and $\hat{\theta}_n$ within the ball of radius $n^{\varepsilon}$, the supremum of the loss $\mathcal{L}(\theta,\hat{\theta}_n)$ is $O(n^{2\varepsilon})$. Let $\Theta_{n,\varepsilon}=\{\theta:||\theta||<n^{\varepsilon}\}$. Now fix $0<\beta<\gamma<1/9$ which gives a rate of equivalence $n^{-\kappa}$ for some $\kappa>0$ (see Theorem 3.3). Choose $\varepsilon$ small enough so that $n^{2\varepsilon-\kappa}=o(1)$  and $\Theta_{n,\varepsilon}\subset\Theta_{n,r,\beta,\gamma}$ so that we can use LAN and using (21) and (23) we obtain the fourth inequality. Now choose $\beta_1<\gamma_1<1/9$ so that $\Theta_{n,r,\beta_1,\gamma_1}\subset \Theta_{n,\varepsilon}$ and the fifth inequality follows. We have also taken infimum over a larger class of estimators without increasing the maximum risk. We use lower bound from Theorem 5.2 to obtain the final inequality.
Finally taking limits we obtain 
$$\liminf\inf_{\hat{\rho}_n}\sup_{||\rho-\sigma||_2<n^{-1/2+\varepsilon}}E_{\rho}n||\rho-\hat{\rho}_n||_2^2\geq \sum_{i=1}^r \mu_i(1-\mu_i)+\sum_{\substack{{1\leq i\leq r}\\{i<j\leq d}}}2\mu_i.$$
\end{proof}

\begin{proof}[Proof of Theorem 6.1]
\textbf{Upper Bound}

Our aim is to show
\begin{equation}
\sup_{\rho\in\Sigma_{n,\varepsilon}\left(  \rho_0\right)  }E_{\rho}|
n^{1/2}\left(  \bar{X}_{n}-\Psi\left(  \rho\right)  \right)|^2
\rightarrow V^2_{\rho_0}. \label{risk-convergence}%
\end{equation}
Let $Y_{1,n},\ldots,Y_{n,n}$ be the array of independent r.v's
all having the distribution of $n^{-1/2}\left(  X_{A}-\Psi\left(
\rho  \right)  \right)  $ under $\rho $;
we have to check%
\begin{align*}
\sup_{\rho\in\Sigma_{n,\varepsilon}\left(  \rho_0\right)  }\sum_{i=1}^{n}EY_{n,i}^{2}  &  \rightarrow V_{\rho_0}^{2}.
\end{align*}
Since $Y_{1,n},\ldots,Y_{n,n}$ have identical distributions, this is
equivalent, for a r.v. $Y_{n,0}=X_{A}-\Psi\left(  \rho
\right)  $ (distribution under $\rho $)
\begin{align}
\sup_{\rho\in\Sigma_{n,\varepsilon}\left(  \rho_0\right)  }EY_{n,0}^{2}  &  \rightarrow V_{\rho_0}^{2} \label{remains-2}.%
\end{align}
Note that
\[
EY_{n,0}^{2}=Var_{\rho }X_{A}=\mathrm{Tr}\;A^{2}%
\rho  -\left(  \mathrm{Tr}\;A\rho
\right)  ^{2}.
\]
Recall that $||A||$ is the operator norm of $A$. We have the following inequalities
\begin{align*}
\left\vert \mathrm{Tr}\;A^{2}\rho  -\mathrm{Tr}\;A^{2}%
\rho_0\right\vert  &  =\left\Vert A^{2}\right\Vert \left\Vert \rho
-\rho_0\right\Vert _{1}\leq\left\Vert A^{2}\right\Vert \gamma
_{n}=o\left(  1\right)  
\end{align*}
where $\gamma_n=n^{-1/2+\varepsilon}$. Similarly we obtain
\[
\left\vert \mathrm{Tr}\;A\rho  -\mathrm{Tr}\;A\rho_0
\right\vert \leq\left\Vert A\right\Vert \gamma_{n}=o\left(
1\right).
\]
The last two displays imply (\ref{remains-2}).

\bigskip

\textbf{Lower Bound}

We now have to establish the complement to
(\ref{risk-convergence}), namely
\begin{equation}
\label{lower-risk-bound}
\liminf_{n}\inf_{\hat{\Psi}}\sup_{\rho\in\Sigma_{n,\varepsilon}\left(  \rho_0\right)
}E_{\rho}|  n^{1/2}\left(  \hat{\Psi}-\Psi\left(  \rho\right)
\right)  |^2  \geq V^2_{\rho_0}. 
\end{equation}
where $\inf_{\hat{\Psi}}$ denotes the infimum over all measurements and all
estimators $\hat{\Psi}$. 
We will construct, for a $d$-dimensional center state  
$\rho_{0}$ (of rank $r$) and a parametrization of all $d$-dimensional nearby states, a "least favorable parametric subfamily" so that the
minimax risk of estimating $\Psi\left(  \rho\right)  $ along this family will
be lower bounded by the minimax risk of estimating the parameter in the limiting model. Our perturbation
setup will be
\begin{equation}
\rho_{t}=\rho_{0}+n^{-1/2}t H\text{, \ \ }t\in\left(  -n^{\varepsilon}
,n^{\varepsilon}\right)  \label{least-fav-family}%
\end{equation}
where $H$ is self-adjoint with $\mathrm{Tr}\;H=0$, and satisfies
\begin{equation}
\mathrm{Tr}\;AH=1.\label{restric-H-2}%
\end{equation}
Then we have
\begin{align*}
\Psi\left(  \rho_{t}\right)   &  =\mathrm{Tr}\;A\rho_{t}%
=\mathrm{Tr}\;A\rho_{0}+n^{-1/2}t\\
&  =\Psi\left(  \rho_{0}\right)  +n^{-1/2}t
\end{align*}
so that $t$ is the appropriate parameter, estimation of which is
equivalent to estimating the functional $\Psi$. Recall that $\left\langle A\right\rangle _{\rho_{0}}=\mathrm{Tr}(\rho_0 A)$ for any operator $A$. For the correct choice of $H$,
set
\begin{align}
\tilde{A} &  :=A-\left\langle A\right\rangle _{\rho_{0}}\mathbf{1,}\nonumber\\
\tilde{H} &  :=\frac{1}{2}\left(  \tilde{A}\rho_{0}+\rho_{0}\tilde{A}\right)
,\nonumber\\
\hat{H} &  :=\frac{\tilde{H}}{\left\langle \tilde{A}^{2}\right\rangle _{\rho_{0}}%
}.\label{H-hat-def}%
\end{align}
With a choice $H=$ $\hat{H}$, the conditions on $H$ are fulfilled:
\begin{align*}
\mathrm{Tr}\;\hat{H} &  =\left\langle \tilde{A}^{2}\right\rangle _{\rho_{0}}%
^{-1}\mathrm{Tr}\;\tilde{H}=\left\langle A^{2}\right\rangle _{\rho_{0}}%
^{-1}\mathrm{Tr}\;\tilde{A}\rho_{0}\\
&  =\left\langle \tilde{A}^{2}\right\rangle _{\rho_{0}}^{-1}\mathrm{Tr}\left(
A-\left\langle A\right\rangle _{\rho_{0}}\mathbf{1}\right)  \rho_{0}\\
&  =0
\end{align*}
and for (\ref{restric-H-2}) we obtain
\begin{align*}
\mathrm{Tr}\;A\hat{H} &  =\mathrm{Tr}\;\left(  A-\left\langle A\right\rangle
_{\rho_{0}}\mathbf{1}\right)  \hat{H}=\mathrm{Tr}\;\tilde{A}\hat{H}\\
&  =\left\langle \tilde{A}^{2}\right\rangle _{\rho_{0}}^{-1}\mathrm{Tr}\;\tilde
{A}\tilde{H}=\left\langle \tilde{A}^{2}\right\rangle _{\rho_{0}}^{-1}\mathrm{Tr}%
\;\tilde{A}^{2}\rho_{0}\\
&  =\left\langle \tilde{A}^{2}\right\rangle _{\rho_{0}}^{-1}\left\langle \tilde{A}%
^{2}\right\rangle _{\rho_{0}}=1.
\end{align*}
We also note that
\begin{align*}
V_{\rho_{0}}^{2} &  =\mathrm{Var}_{\rho_{0}}\left(  X_{A}\right)  \\
&  =E_{\rho_{0}}\left(  X_{A}-E_{\rho_{0}}X_{A}\right)  ^{2}\\
&  =\mathrm{Tr}\;\tilde{A}^{2}\rho_{0}=\left\langle \tilde{A}^{2}\right\rangle
_{\rho_{0}}.
\end{align*}

We now proceed to find the limiting classical-quantum Gaussian model for this choice of perturbation. It can be easily verified that up to first order terms $\rho_{\vartheta,r}$ is equal to $\tilde{\rho}_{\vartheta,r}$ where $\tilde{\rho}_{\vartheta,r}$ given as follows
\begin{equation}\label{rho.theta.tilde}
\tilde{\rho}_{\vartheta/\sqrt{n},r} 
:=
\begin{bmatrix}

\mu_1 + \tilde{u}_1 & \tilde{\zeta}_{1,2}^* & \dots & \tilde{\zeta}_{1,r}^* & \vline & \tilde{\zeta}^*_{1,r+1} & \tilde{\zeta}_{1,r+2} & \dots & \tilde{\zeta}^*_{1,d}
\\
\tilde{\zeta}_{1,2} & \mu_2  + \tilde{u}_2 & \ddots& \vdots &\vline&\tilde{\zeta}_{2,r+1} & \tilde{\zeta}^*_{2,r+2} & \ddots& \vdots 

\\
\vdots & \ddots & \ddots & \tilde{\zeta}_{r-1,r}^* &\vline& \vdots & \ddots & \ddots & \tilde{\zeta}^*_{r-1,d} \\
\tilde{\zeta}_{1,r} & \dots & \tilde{\zeta}_{r-1,r} &\mu_r - \sum_{i=1}^{r-1} \tilde{u}_i&\vline&\tilde{\zeta}^*_{r,r+1} & \dots & \tilde{\zeta}^*_{r,d-1} &\tilde{\zeta}^*_{r,d}\\
 \hline

 \tilde{\zeta}_{1,r+1} & \tilde{\zeta}_{2,r+1} & \dots & \tilde{\zeta}_{r,r+1}&\vline&0 & 0 & \dots & 0
\\
\tilde{\zeta}_{1,r+2} & \tilde{\zeta}_{2,r+2} & \ddots& \vdots &\vline&0 & 0 & \ddots& \vdots  
\\
\vdots & \ddots & \ddots & \tilde{\zeta}_{r,d-1} &\vline&\vdots & \ddots & \ddots & 0 \\
\tilde{\zeta}_{1,d} & \dots & \tilde{\zeta}_{r-1,d} &\tilde{\zeta}_{r,d}&\vline&0 & \dots & 0 &0\\
\end{bmatrix}
\end{equation}
where $\tilde{u}=n^{-1/2}u,~\tilde{\zeta}=n^{-1/2}\zeta$ and $u_i\in \mathbb{R}, ~ \zeta_{j,k}\in \mathbb{C}$. From (\ref{least-fav-family}) we obtain
\begin{equation}\label{theta_H}
n^{-1/2}t \hat{H} 
=n^{-1/2}
\begin{bmatrix}

 u_1 & \zeta_{1,2}^* & \dots & \zeta_{1,r}^* & \vline & \zeta^*_{1,r+1} & \zeta_{1,r+2} & \dots & \zeta^*_{1,d}
\\
\zeta_{1,2} & u_2 & \ddots& \vdots &\vline&\zeta^*_{2,r+1} & \zeta^*_{2,r+2} & \ddots& \vdots 

\\
\vdots & \ddots & \ddots & \zeta_{r-1,r}^* &\vline& \vdots & \ddots & \ddots & \zeta^*_{r-1,d} \\
\zeta_{1,r} & \dots & \zeta_{r-1,r} & - \sum_{i=1}^{r-1} u_i&\vline&\zeta^*_{r,r+1} & \dots & \zeta^*_{r,d-1} &\zeta^*_{r,d}\\
 \hline

 \zeta_{1,r+1} & \zeta_{2,r+1} & \dots & \zeta_{r,r+1}&\vline&0 & 0 & \dots & 0
\\
\zeta_{1,r+2} & \zeta_{2,r+2} & \ddots& \vdots &\vline&0 & 0 & \ddots& \vdots  
\\
\vdots & \ddots & \ddots & \zeta_{r,d-1} &\vline&\vdots & \ddots & \ddots & 0 \\
\zeta_{1,d} & \dots & \zeta_{r-1,d} &\zeta_{r,d}&\vline&0 & \dots & 0 &0\\
\end{bmatrix}.
\end{equation}
Let $x=\la A\ra_{\rho_0}=\sum_{i=1}^r\mu_iA_{ii}$ and $y=\la \tilde{A}^2\ra_{\rho_0}$.
It can be easily verified that
\begin{equation}\label{variance_comp}
 y=\Tr{\rho_0 \tilde{A}^2}=\sum_{i=1}^r \mu_i(A_{ii}-x)^2+\sum_{\substack{1\leq i\leq r\\i<j\leq d}}|A_{ji}|^2(\mu_i+\mu_j).   
\end{equation}
Also from (\ref{H-hat-def}) we obtain 

\begin{equation}\label{hat_H_expansion}
\hat{H}
:=\frac{1}{y}
\begin{bmatrix}

 B_{1,1} & B_{2,1}^* & \dots & B_{r,1}^* & \vline & C^*_{r+1,1} & C^*_{r+2,1} & \dots & C^*_{d,1}
\\
B_{2,1} & B_{2,2} & \ddots& \vdots &\vline&C^*_{r+1,2} & C^*_{r+2,2} & \ddots& \vdots 

\\
\vdots & \ddots & \ddots & B_{r,r-1}^* &\vline& \vdots & \ddots & \ddots & C^*_{d,r-1} \\
B_{r,1} & \dots & B_{r,r-1} & B_{r,r}&\vline&C^*_{r+1,r} & \dots & C^*_{d-1,r} &C^*_{d,r}\\
 \hline

 C_{r+1,1} & C_{r+1,2} & \dots & C_{r+1,r}&\vline&0 & 0 & \dots & 0
\\
C_{r+2,1} & C_{r+2,2} & \ddots& \vdots &\vline&0 & 0 & \ddots& \vdots  
\\
\vdots & \ddots & \ddots & C_{d-1,r} &\vline&\vdots & \ddots & \ddots & 0 \\
C_{d,1} & \dots & C_{d,r-1} &C_{d,r}&\vline&0 & \dots & 0 &0\\
\end{bmatrix}
\end{equation}
with the following relations
\begin{align*}
B_{i,i}& =(A_{ii}-x)\mu_i \text{ for } 1\leq i\leq r\\
B_{j,i}&=\frac{A_{ji}}{2}(\mu_i+\mu_j) \text{ for   }i<j\leq r\\
C_{j,i}&=\frac{A_{ji}}{2}\mu_i \text{ for   }1\leq i\leq r, r+1 \leq j\leq d.
\end{align*}
Comparing equations (\ref{theta_H}) and (\ref{hat_H_expansion}) we get
\begin{equation}\label{compare_perturb}
    u_i=t\frac{(A_{ii}-x)\mu_i}{y} \text{ for } 1\leq i\leq r,\quad
    \zeta_{i,j}=t\frac{A_{ji}}{2y}(\mu_i+\mu_j) \text{ for   }1\leq i\leq r,i<j\leq d
\end{equation}
where $\mu_j=0$ for $j>r$. Also let $\breve{u}_i=\frac{(A_{ii}-x)\mu_i}{y}$ and $\breve{\zeta}_{i,j}=\frac{A_{ji}}{2y}(\mu_i+\mu_j)$.

Recall that after defining $\mathfrak{z}_{ij}=\zeta_{ij}/\sqrt{\mu_i-\mu_j}$ and replacing $\vartheta$ by $\theta$, limiting model for $\rho^{\otimes n}_{\theta/\sqrt{n}}$ is given by equation (17) and is given by
\begin{align*}
    \Phi^{\theta,r}=&\cal{N}_{r-1}(\tilde{u},V_{\mu})\otimes \bigotimes_{1\leq i<j \leq r}\phi_{\beta_{ij}}^{\mathfrak{z}_{ij}}\otimes\bigotimes_{\substack{1\leq j\leq r\\ r+1\leq j \leq d}}\phi_{\infty}^{\mathfrak{z}_{ij}}\\
    =&\cal{N}_{r-1}(t\breve{u},V_{\mu})\otimes \bigotimes_{1\leq i<j \leq r}\mathbb{N}_2\left(t\sqrt{\frac{2}{\mu_i-\mu_j}}(Re \breve{\zeta}_{ij},Im \breve{\zeta}_{ij}),\frac{\mu_i+\mu_j}{2(\mu_i-\mu_j)}I_2\right)\\&\otimes\bigotimes_{\substack{1\leq j\leq r\\ r+1\leq j \leq d}}\mathbb{N}_2\left(t\sqrt{\frac{2}{\mu_i}}(Re \breve{\zeta}_{ij},Im \breve{\zeta}_{ij}),\frac{1}{2}I_2\right).
\end{align*}

Note that an alternate expression for the limiting state is given by $\mathfrak{N}(t\tau,\mathscr{S})$ where $\Sigma$ is given as in (19) with $\sigma^2_{\beta_{ij}}=\frac{\mu_i+\mu_j}{2(\mu_i-\mu_j)}$ for $\ijr$ and $V=V_{\mu}$. Also the shift is given by $$\tau=\breve{u}\oplus{\bigoplus_{\ijr}}\sqrt{\frac{2}{\mu_i-\mu_j}}\left(Re \breve{\zeta}_{ij},Im \breve{\zeta}_{ij}\right).$$Denoting $\hat{\Psi}=\Psi(\rho_0)+n^{-1/2}\hat{t}$, we have the following chain of inequalities
\begin{align*}
    \inf_{\hat{\Psi}}\sup_{\rho\in \Sigma_{n,\varepsilon}(\rho_0)}E_{\rho}|n^{1/2}(\hat{\Psi}-\Psi(\rho))|^2&\geq  \inf_{\hat{t}}\sup_{|t|\leq n^{\varepsilon}}E_{\rho_t}|\hat{t}-t|^2\\
    & \geq  \inf_{|\hat{t}|\leq n^{\varepsilon}}\sup_{|t|\leq n^{\varepsilon}}E_{\mathfrak{N}_t}|\hat{t}-t|^2+o(1)
\end{align*}
where $\mathfrak{N}_t:=\mathfrak{N}(t\tau,\mathscr{S})$. We observe that it is the estimator $\hat{t}$ can also be restricted to a ball of radius $n^{\varepsilon}$ without increasing the risk. Also for the same restriction of $t$ and $\hat{t}_n$  the supremum of the the loss $E_{\rho_{t}}|\hat{t}-t|^2$ is $O(n^{2\varepsilon})$ and multiplying by the equivalence rate of the reverse channel we have $n^{2\varepsilon-\kappa}=o(1)$ (for sufficiently small $\varepsilon$) so that using (21) and (23) we obtain the second inequality.

Now construct the following prior $\pi$ on $t$: $t\sim \mathcal{N}(0,n^{\varepsilon_0})$ with $\varepsilon_0<\varepsilon$. The minimax risk can be lower bounded as follows:
$$\inf_{|\hat{t}|\leq n^{\varepsilon}}\sup_{|t|\leq n^{\varepsilon}}E_{\mathfrak{N}_{t}}|\hat{t}-t|^2\geq \inf_{|\hat{t}|\leq n^{\varepsilon}}\int_{|t|\leq n^{\varepsilon}}E_{\mathfrak{N}_{t}}|\hat{t}-t|^2d\pi(t).$$
Using an argument similar to the proof of Lemma \ref{lower_bound_error} we can show that the RHS is further bounded by $\inf_{|\hat{t}|\leq n^{\varepsilon}}\int E_{\mathfrak{N}_{t}}|\hat{t}-t|^2d\pi(t)+o(1)$ and since taking an infimum over a larger class of estimators (i.e. if $\hat{t}$ is allowed to be unrestricted) decreases the risk, we obtain

$$\inf_{\hat{\Psi}}\sup_{\rho\in \Sigma_{n,\varepsilon}(\rho_0)}E_{\rho}|n^{1/2}(\hat{\Psi}-\Psi(\rho))|^2\geq \inf_{\hat{t}}\int E_{\mathfrak{N}_{t}}|\hat{t}-t|^2d\pi(t)+o(1).$$

By (\ref{Bayes-risk-2}) the first term in RHS is given by $\frac{\left(  \tau^{T}\Sigma^{-1}%
\tau\right)  ^{-1}n^{\epsilon_0}}{\left(  \tau^{T}\Sigma^{-1}\tau\right)  ^{-1}%
+n^{\epsilon_0}}$ and hence we obtain

$$\liminf_n\inf_{\hat{\Psi}}\sup_{\rho\in \Sigma_{n,\varepsilon}(\rho_0)}E_{\rho}|n^{1/2}(\hat{\Psi}-\Psi(\rho))|^2\geq (\tau^{T}\Sigma^{-1}%
\tau)^{-1}.$$
Plugging in the value of $\Sigma$ and $\tau$ we obtain
\begin{align*}
   \tau^{T}\Sigma^{-1}%
\tau&=\breve{u}^TV^{-1}_{\mu}\breve{u}+ \sum_{1\leq i<j \leq r}\frac{4|\breve{\zeta}_{ij}|^2}{\mu_i+\mu_j}+\sum_{\substack{1\leq j\leq r\\ r+1\leq j \leq d}}\frac{4|\breve{\zeta}_{ij}|^2}{\mu_i}\\
    & =\breve{u}^TV^{-1}_{\mu}\breve{u}+\sum_{\substack{1\leq i\leq r\\i<j\leq d}}\frac{|A_{ji}|^2}{y^2}(\mu_i+\mu_j).  
\end{align*}
where the last line is obtained by plugging in the values of $\breve{\zeta}_{ij}$ and letting $\mu_j=0$ for $j>r$.

By example 3.1 of \cite{inverse_multi} $\left[V^{-1}_{\mu}\right]_{ij}=\mu^{-1}_i\delta_{ij}+(1-\sum_{i=1}^{r-1}\mu_k)^{-1}$ and hence noting that $1-\sum_{i=1}^{r-1}\mu_k=\mu_r$ we have
\begin{align*}
    \breve{u}^TV^{-1}_{\mu}\breve{u}=&\frac{1}{y^2}\sum_{i,j=1}^{r-1}\mu_i\mu_j(A_{ii}-x)(A_{jj}-x)(\mu^{-1}_i\delta_{ij}+\mu_r^{-1})\\
    =&\frac{1}{y^2}\left[\sum_{i=1}^{r-1}\mu_i(A_{ii}-x)^2+\mu_r^{-1}\sum_{i,j=1}^{r-1}A_{ii}A_{jj}\mu_i\mu_j-\mu_r^{-1}\sum_{i,j=1}^{r-1}xA_{jj}\mu_i\mu_j\right.\\
    & \left.-\mu_r^{-1}\sum_{i,j=1}^{r-1}A_{ii}x\mu_i\mu_j+\mu_r^{-1}\sum_{i,j=1}^{r-1}x^2\mu_i\mu_j\right]\\
    =&\frac{1}{y^2}\left[\sum_{i=1}^{r-1}\mu_i(A_{ii}-x)^2+\mu_r^{-1}(x-A_{rr}\mu_r)^2-2x\mu_r^{-1}(1-\mu_r)(x-A_{rr}\mu_r)\right.\\
    & \left.+\mu_r^{-1}x^2(1-\mu_r)^{2}\right]\\
    =& \frac{1}{y^2}\left[\sum_{i=1}^{r-1}\mu_i(A_{ii}-x)^2+\mu_r(A_{rr}-x)^2\right]
\end{align*}
Thus we obtain
$$\tau^{T}\Sigma^{-1}%
\tau=\frac{1}{y^2}\sum_{i=1}^{r}\mu_i(A_{ii}-x)^2+\sum_{\substack{1\leq i\leq r\\i<j\leq d}}\frac{|A_{ji}|^2}{y^2}(\mu_i+\mu_j)$$
Plugging back the expression of $y$ from equation (\ref{variance_comp}) we observe that the RHS of the above equation is $y^{-1}$ and hence we have
$$\liminf_n\inf_{\hat{\Psi}}\sup_{\rho\in \Sigma_{n,\varepsilon}(\rho_0)}E_{\rho}|n^{1/2}(\hat{\Psi}-\Psi(\rho))|^2\geq y=V^{2}_{\rho_0}.$$
\end{proof}
\textbf{Remark}: Note that if we had a multivariate Gaussian $\mathcal{N}_k(t\tau,\Sigma)$ with $k=r-1+r(r-1)/2+r(d-r)$, then it can be easily shown that the minimax risk for estimation of $t$ is  $(\tau^{T}\Sigma^{-1}%
\tau)^{-1}$. So the minimax risk for the classical model and the classical-quantum model matches for this particular case where the shift is indexed by $1$-dimensional parameter.
\section{Proof of technical lemmas}

\subsection{Auxiliary lemma for channel construction}
First we set some notations and prove a technical lemma (the equivalent of Lemma 7.1 of \cite{Kahn&Guta}). We note that the some statements of the lemma differs because of the low-rank structure. Although one can prove the statements by suitably modifying Lemma 7.1 of \cite{Kahn&Guta} we include the proof of the entire lemma (and some other lemmas in the following subsection) \textit{mutatis mutandis} for the convenience of the reader.
However the explanations will be brief and the reader can refer to the discussion preceding Lemma 7.1 of \cite{Kahn&Guta} for further clarification.

1) Let $l(c)$ be the length of the column $c$ in the Young diagram $\lambda$. Note that there are $\lambda _i - \lambda _{i+1} $ columns such that $l(c) = i$. Also, in our case $l(c)\leq r$ since the Young diagrams can only have $r$-rows in our setup.

2) Recall that $f_{\mathbf{a}}=f_{\mathbf{a}(1)} \otimes \dots \otimes f_{\mathbf{a}(n)}$ are the basis vectors and we associate a Young tableau $t_{\bf a}$ with each $f_{\mathbf{a}}$ where the indices $\bf{a}(i)$ fill the boxes of a diagram $\lambda$ row-wise.
Let $t^c_a$ be the content of the column $c$ of $t_{\bf a}$, i.e. the function 
$$t_{\bf a}^{c}: \{1,\dots,l(c)\} \rightarrow\{1,\dots, d\}$$ which gives the column entries of the column $c$. For example, if 
$t_{\bf a}$ is given by the following semi-standard Young tableau\begin{center}
\begin{ytableau}
 $1$&$2$&$3$ \cr
 $2$& $3$\cr
\end{ytableau}
\end{center}
we get the values:
\begin{align*}
t_{\bf a}^{1}(1) & = 1, & t_{\bf a}^{1}(2) & = 2, & t_{\bf a}^{2}(1) & = 2, & t_{\bf a}^{2}(2) & = 3, & t_{\bf a}^{3}(1) & = 3. 
\end{align*}

3) We also use the alternate notation $f_{\bf m}$ for $f_{\bf a}$ (with  ${\bf m}= \{m_{i,j} :  1\leq i\leq r,  i<j\leq d\} $, where $m_{i,j}$ is the number of $j$'s in the row $i$ of $t_{\bf a}$) and that $y_{\lambda}f_{\bf m}=q_{\lambda}p_{\lambda}f_{\bf m}$ spans $\mathcal{H}_{\lambda}$ where $t_{\bf{a}}$ is a semi-standard tableau. Let $ \mathcal{O}_{\lambda}({\bf m})$ be the orbit of $f_{\bf m}$ under the 
subgroup $\mathcal{R}_{\lambda}$ of row permutations i.e. the set of vectors 
$f_{\mathbf{b}}$ which have exactly $m_{i,j}$ boxes with $j$ in row $i$, and the rest are $i$. Note that a semi-standard tableau ensures that if we permute among rows only, row $i$ has no entries smaller than $i$. As a consequence of the orbit-stabilizer theorem we have
\begin{equation}\label{plambda.orbit.decomposition}
p_{\lambda} f_{\bf m} = \sum_{\sigma \in
\mathcal{R}_{\lambda}}  f_{{\bf a}\circ \sigma}=
\sum_{\Vb\in \mathcal{O}_{\lambda}({\bf m})}
\frac{\# \mathcal{R}_{\lambda}}{\# \mathcal{O}_{\lambda}({\bf m})} \Vb.
\end{equation}

4) If a column contains two identical entries it does not contribute to $y_{\lambda}f_{\bf a}$ since we are antisymmetrizing with $q_{\lambda}$. Thus only tableaux
$t_{\bf a}$ (not necessarily semistandard after the action of $p_{\lambda}$) which do not have two equal entries in the same column are important and we call them admissible ones and denote the collection of associated vectors $f_{\bf a}$ by $\mathcal{V}$.

5) For any $f_{\bf a}\in\mathcal{O}_{\lambda}({\bf m})$ we define 
$$
\Gamma(f_{\mathbf{a}})  := |{\bf m}| - \#\{ 1\leq c \leq\lambda_{1} : \tca \neq {\rm Id}^{c} \},
$$
and denote by $\mathcal{V}^{\Gamma}({\bf m})$ the set of vectors 
$f_{\bf a}\in \mathcal{O}_{\lambda}({\bf m})\bigcap \mathcal{V}$ with 
$\Gamma(f_{a})=\Gamma$. Then we have

$$
\mathcal{O}_{\lambda}({\bf m})\bigcap \mathcal{V} =
\bigcup_{\Gamma\in \mathbb{N}} \mathcal{V}^{\Gamma}({\bf m}).$$

It can be easily seen that $\Gamma(f_{\bf a})\geq 0$  and is zero if and only if each column 
$\tca$ is either ${\rm Id}^{c}$ or of the form $\tca(r) = j\delta_{r=i} + r\delta_{r\neq i}$ for some $i \leq l(c) < j$. A tableau $\tca$ of the latter form will be called an \emph{$(i,j)$-substitution}.

One then considers the following `algorithm' to build all the possible $\Va\in\mathcal{V}^{\Gamma}({\bf m})$.
\smallskip

\emph{Algorithm}
\nopagebreak

For a given $({\bf m}, \lambda)$ we generate a particular admissible $f_{\bf a} \in \mathcal{O}_{\lambda}({\bf m})$ by selecting the $m_{i,j}$  boxes on row $i$ which are filled with $j$, for all $1\leq i\leq r,i<j\leq d$. To ensure admissibility we have to ensure that no column should have two boxes filled with the same number. Start with $f_{\bf 0}:= f_{{\bf m}={\bf 0}}$ (row $i$ is filled with 
$i$'s only), and with a set of $|{\bf m}|$ ``bricks" containing 
$m_{i,j}$ identical bricks labelled $(i,j)$, for $1\leq i\leq r,i<j\leq d$. These bricks are placed sequentially i.e. for a particular $(i,j)$ brick we replace the content of a box from row $i$ with $j$ (initially the content was $i$). Once the bricks have been placed the content of each column c (i.e $t_{\bf a}^{c}$) is uniquely defined by the set of bricks $\kappa$  which shall be called a \emph{column-modifier}. 

For example if $\kappa =\{ (i ,j), (f,l)\}$ then the column has entries 
$$
t_{\bf a}^{c} (k)= 
\left\{ 
\begin{array}{lll}
j &  \textrm{if $k = i$} ;\\
l &  \textrm{if $k = f$} ;\\
k & \textrm{otherwise}. 
\end{array}
\right.
$$
 A column modifier of the form  $\kappa(i,j)$ is called elementary and can only be used in a column with $i\leq l(c)<j$ (otherwise the entry $j$ would appear twice making the column inadmissible). Now, since the length of a column is at most $r$ and all entries must be different, there are less than $d!$ (although the upper bound is $d!/(d-r)!$ we use the $d!$ since a sharper bound does not change the asymptotics) different types of column-modifiers. Note that a column-modifier always increases the value of the modified cells, so that in this case  $t^c_a(\{1, \dots, l(c)\}) \neq \{1,\dots, l(c)\}$.

 To count the total number of ways the tableau can be formed we can count the total number of ways the columns are modified. A given collection of column-modifiers is uniquely determined by $\{ m_\kappa,\kappa\}$ where $m_{\kappa}$ is the multiplicity of $\kappa$. We consider the following procedure:

\begin{enumerate}
\item[I.]
\label{1} Choose $\Gamma$ bricks among our $|{\bf m}|$. As we have at most $rd$
different types of bricks there are at most $(rd)^{\Gamma} $ possibilities. 
\item[II.]
\label{2} Recall that $|\textbf{m}|-\Gamma$ is the total number of modified columns. As this is the exact number of bricks left after the first stage we can consider them as a set of elementary column-modifiers. Indeed if all the column modifiers were elementary, then $\Gamma=0$. If $\Gamma\neq 0$, we form non-elementary column modifiers from elementary ones by sequentially adding each of the $\Gamma$ bricks selected in the first stage, to some of the elementary column-modifiers.
At each step there are at most $d!$ different choices (of column modifiers) to which we can attach the new brick. Hence, we have at most  $(d!)^{\Gamma}$ possibilities. At the end of stage II at least $\max \{ 0, |{\bf m}| -2\Gamma\}$ of the column-modifiers are elementary, and we note that $m_{\kappa(i,j)}\leq m_{i,j}$.  \item[III.]
\label{3} We apply the column-modifiers to the columns of $f_{\mathbf{0}}$, with the constraints that no two modifiers are applied to the same column and the resulting 
$f_{\bf{a}}$ is admissible. 
\end{enumerate}

While inserting the column modifiers we note that for a non-elementary column modifier $\kappa$ there are less than n possibilities. In the case of an elementary one of type $\kappa(i,j)$ we have the constraints that it can only be inserted in a column with at least $i$ rows, and since the column cannot contain another $j$ (to ensure admissibility) its length has to be smaller than $j$. Noting that number of such columns is $\lambda_i-\lambda_j$, we give an upper bound to the number of possibilities at stage three of the algorithm as follows:
\begin{align}
\label{up3}
\prod_{\kappa \neq \kappa({i,j})}  \frac{n^{ m_{\kappa}}}{m_{\kappa}!} \cdot
\prod_{\ijr} \frac{(\lambda_i - \lambda_j)^{m_{\kappa(i,j)}}} {m_{\kappa(i,j)}!}
\end{align}

 the combinatorial factor $\prod_{\kappa} m_{\kappa}!$ ($m_{\kappa}$ being the number of column modifiers of type $\kappa$) appearing to account for the identical columns.  If $\Gamma = 0$, then all column modifiers are elementary and for each $\kappa(i,j)$, 
the number of available columns is at least 
$(\lambda_i -\lambda_j - \lvert{\bf m}\vert)_{+}:=  0 \wedge\lambda_i -\lambda_j - \lvert{\bf m}\vert $.
Thus we have the following lower bound
\begin{align}
\label{do3}
  \prod_{\ijr} \frac{ (\lambda_i - \lambda_j - |\mathbf{m}|)_{+}^{m_{i,j}}}{m_{i,j}!}.
\end{align}

6) We collect all multiplicities in  
$
E := \{ m_{\kappa}^{\bf a} :\kappa\},
$
where $m_{\kappa}^{\bf a}$ is the number of columns with 
column-modifer $\kappa$, noting that $\Gamma$ is a function of $E$ and the following relation holds
$$
\Gamma(f_{\bf a}) = |{\bf m}| - \sum_{\kappa} m_{\kappa}^{\bf a}.
$$
We note that the condition $\Gamma(f_{\bf a})=0$ identifies the multiplicity 
set $E^{0}$. In particular, we have $m_{\kappa(i,j)} = m_{i,j}$ for all $\ijr$ and the other 
$m_\kappa = 0$. 
In the general case, we decompose the set $\mathcal{V}^{\Gamma}({\bf m})$ in terms of $\mathcal{V}^{E}({\bf m})$ (the set of tableaux in 
$\mathcal{O}_{\lambda}({\bf m})\bigcap \mathcal{V}$ with 
$E(f_{\bf a})= E$) as follows:
\begin{equation}
\mathcal{V}^{\Gamma} ({\bf m})= \bigcup_{ E } \mathcal{V}^{E}({\bf m}).\label{gamma_decomp}
\end{equation}
7) Finally to each column $c$ of $t_{\bf a}$ we associate the set of modified entries: each element of this set is of the form $(i,+)$ or $(i,-)$ depending on whether $i$ is added or deleted. It is easy to verify that if $t_{\bf a}$ is admissible, this modified set determines column-modifier $\kappa$ associated to $c$. and hence shall be denoted by $S(\kappa)$. For example $S(\kappa(i,j))  =
\{ (i,-), (j,+)\} $ and for  $\kappa= \{ (i,j), (j,k)\}$ we have
$S(\kappa)  = \{ (i,-), (k,+)\}$. 
We define the multiplicities $m_S^{\bf a}= \sum_{\kappa : S(\kappa) =S} m_{\kappa}^{\bf a}$ and consider the set $F(f_{\bf a}):= \{m_S^{\bf a} :S \}$ .

We are now ready to state the auxiliary lemma which will be crucial in proving Lemma \ref{comparison_unshifted} and Lemma \ref{comparison_shift_operator}.
\begin{lemma}
\label{lemtools}
\noindent
\begin{enumerate}
\item{For any unitary operator $U\in M(\mathbb{C}^{d})$, 
for any basis vectors $\Va$ and $\Vb$, we have
\begin{align}
\label{formdet}
\langle\Va\vert q_{\lambda } U^{\otimes n} \Vb \rangle & = \prod_{1\leq c \leq \lambda _1} \det(U^{\tca,\tcb}),
\end{align}
where $U^{\tca,\tcb}$ is the $\lc \times \lc$ minor of $U$ given by $[U^{\tca,\tcb}]_{i,j} = U_{\tca(i),\tcb(j)}$.
}
\end{enumerate}

Under the assumptions
\begin{align}
\label{hyp}
 |\mathbf{m}| & \leq n^{\eta}, \\
\lambda & \in \Lambda_{2}, \notag \\
\inf_{i} |\mu_i - \mu_{i+1}| & \geq \delta ,\notag \\
\mu_r & \geq \delta ,\notag \\
\|\mathfrak{z}\| & \leq C n^{\beta}, \qquad \beta\leq 1/2. \notag 
\end{align}
we have the following estimates  with remainder terms uniform in the eigenvalues $\{\mu_i\}$:
\begin{enumerate}[resume]
\item  { The number of admissible $\Va\in \mathcal{O}_{\lambda}({\bf m})$ 
with $\Gamma(\Va) = 0$ is 
\begin{equation}
\label{cg0}
\# \mathcal{V}^{0} ({\bf m})= 
\prod_{\ijr}\frac{(\lambda_i-\lambda_j)^{m_{i,j}}}{m_{i,j}!} (1+ O(n^{-1 + 2\eta} / \delta)).
\end{equation}
}
\item {
Let $E:= \{m_{\kappa}:\kappa\} $ with $\Gamma(E)=\Gamma$. 
The number 
of admissible $\Va\in \mathcal{O}_{\lambda}({\bf m})$ with $E(\Va) = E$ is bounded by:
\begin{equation}
\label{cggam}
\# \mathcal{V}^{E} ({\bf m}) \leq  n^{-\Gamma  + \sum_{\ijr} (m_{i,j} - m_{\kappa({i,j})})} \prod_{\ijr}\frac{(\lambda_i-\lambda_j)^{m_{\kappa({i,j})}}}{m_{\kappa({i,j})}!}.
\end{equation} 
}
\item{
The number of admissible $\Va\in \mathcal{O}_{\lambda}({\bf m})$ with 
$\Gamma(\Va) = \Gamma$ is bounded by:
\begin{equation}
\label{cgG}
\# \mathcal{V}^{\Gamma}({\bf m}) \leq C^{\Gamma} n^{-\Gamma} \delta^{-2 \Gamma} |\mathbf{m}|^{2\Gamma} \prod_{\ijr}\frac{(\lambda_i-\lambda_j)^{m_{i,j}}}{m_{i,j}!},
\end{equation}
for a constant $C= C(d)$.
}
\item { 
Let $\Va \in \mathcal{V}^{\Gamma^{a}}(\mathbf{l})$, and consider $ \mathcal{V} ^{\Gamma ^b} ({\bf m}) \subset\mathcal{O}_{\lambda}(\mathbf{m})$ for some fixed 
$\Gamma^b$. Then:
\begin{align}
\label{prod_id}
\left| \left\langle \Va \Bigg| q_{\lambda} \sum_{\Vb \in \mathcal{V}^{\Gamma^b} ({\bf m})} \Vb \right\rangle \right| \leq \left\{
\begin{array}{cc}
0 & \mbox{if } \Gamma^b \neq |\mathbf{m}| - |\mathbf{l}| + \Gamma^a \\
(C |\mathbf{m}|)^{\Gamma^b} & \mbox{otherwise} \\
\end{array}
\right.,
\end{align}
with $C= C(d)$.
}
\item { If $\Va \in \mathcal{V}^{0}(\mathbf{m})$, then
\begin{align}
\label{prod_id0}
\left\langle \Va \Bigg| q_{\lambda} \sum_{\Vb \in \mathcal{O}_{\lambda}(\mathbf{m})} \Vb \right\rangle  = 1.
\end{align}
}
\item\label{zem0} {
If $f_{\bf a} \in \mathcal{V}^{0}({\bf m})$ so that its set of elementary column-modifiers is  
$E^{0}= \{ m_{\kappa(i,j)}=m_{i,j} \}$, then 
\begin{align}
\langle \Va | q_{\lambda} U^r(\mathfrak{z}/\sqrt{n})^{\otimes n} \Vzero\rangle
\label{interg0}
 & = \exp\left( -\frac{\|\mathfrak{z}\|^2_2}{2}\right) \prod_{\ijr} \left(\frac{\mathfrak{z}_{i,j}}{\sqrt{n}\sqrt{\mu_i - \mu_j}}\right)^{m_{i,j}} r(n),
\end{align}
where
$$
\|\mathfrak{z}\|^2_2 =  \sum_{\ijr}|\mathfrak{z}|^2_{ij}$$

and the error factor is given by
$$r(n) = 1 + O\left(n^{-1 + 2\beta + \eta}\delta^{-1}, n^{-1/2 +2 \beta} \delta^{-1}, 
n^{-1 + 2\beta+ \alpha }\delta^{-1} \right).
$$
}
\item { If $\Va \in \mathcal{V}^{E}({\bf m})$, so that its  set of column-modifiers is 
$E= \{m_{\kappa}: \kappa\}$ and $\Gamma(E)= \Gamma$, then
\begin{equation}
\label{interg}
\left| \langle \Va | q_{\lambda} U^r(\mathfrak{z}/\sqrt{n})^{\otimes n} \Vzero\rangle\right|
\leq \exp\left(-\frac{\|\mathfrak{z}\|^2_2}{2}\right) \left(\frac{ C\|\mathfrak{z}\|}{\sqrt{n\delta}}\right)^{\Upsilon}\prod_{\ijr} \left(\frac{ \mathfrak{z}_{i,j}}{\sqrt{n}\sqrt{\mu_i - \mu_j}}\right)^{m_{\kappa({i,j})}} r(n)
\end{equation}

where $$\Upsilon=- \Gamma+\sum_{\ijr} (m_{i,j} - m_{\kappa({i,j})}),$$ $C=C(d)$ a constant and $r(n)$ as in point \ref{zem0} above.
}
\item {  Under the further hypotheses  that $\|\bz\| \leq n^{\beta}$, $m_{i,j} \leq 2 |\mathfrak{z}_{i,j}+ z_{i,j}|n^{\beta + \epsilon}$ for some $\epsilon > 0$,
 we have:
\begin{align}
\label{numer}
& \left\langle \sum_{\Va \in \mathcal{O}_{\lambda}(\mathbf{m})}\Va \Bigg| q_{\lambda} U^r((\mathfrak{z}+\bz)/\sqrt{n}) \Vzero \right\rangle \notag \\
 & =  \exp\left( -\frac{\|\mathfrak{z} + \bz\|^2_2}{2}\right) \prod_{\ijr}\frac{\big((\mathfrak{z}_{i,j} + z_{i,j})(\sqrt{n}\sqrt{\mu_i - \mu_j})\big)^{m_{i,j}}}{m_{i,j}!} r(n),
\end{align} 
with 
\begin{equation*}
r(n) =  1 + O\left(n^{-1 + 2\beta + \eta}\delta^{-1}, 
n^{-1 + 2 \beta+\alpha}\delta^{-1}, n^{-1 + 2\eta}\delta^{-1}, n^{-1 + \alpha + \eta}\delta^{-1}, \delta^{-3/2} n^{-1/2 + 3 \beta + 2 \epsilon} \right).
\end{equation*}
}
\item {Under the further hypotheses that  $|\mathbf{l}| \leq |\mathbf{m}|$ and $n^{1 - 3 \eta} > 2 C / \delta^{2}$, where $C=C(d)$,
\begin{align}
\label{denom}
\left\vert\left \langle \sum_{\Va \in \mathcal{O}_{\lambda}(\mathbf{l} )}\Va \Bigg| q_{\lambda}  \sum_{\Vb \in \mathcal{O}_{\lambda}(\mathbf{m})}\Vb \right\rangle\right \vert
& \leq   (C |\mathbf{m}|)^{|\mathbf{m}| - |\mathbf{l}|}\prod_{\ijr} \frac{(\lambda_i - \lambda_j)^{l_{i,j}}}{l_{i,j}!} \left(\frac{C|\mathbf{l}|^2|\mathbf{m}|}{n\delta^2}\right)^{\Gamma^a_{\min}(\mathbf{l}, \mathbf{m})}
\end{align} 
with 
\begin{align}
\Gamma^a_{\min}(\mathbf{l}, \mathbf{m}) \geq \frac{\big(|\mathbf{l} - \mathbf{m}|  + 3 |\mathbf{l}| - 3 |\mathbf{m}| \big)_{+}}{6} .
\end{align}
}
\item{
We have
\begin{align}
\label{denom54}
\left\langle \sum_{\Va \in \mathcal{O}_{\lambda}(\mathbf{m})}\Va \Bigg| q_{\lambda}  \sum_{\Vb \in \mathcal{O}_{\lambda}(\mathbf{m})}\Vb \right\rangle 
& =
\prod_{\ijr} \frac{(\lambda_i - \lambda_j)^{m_{i,j}}}{m_{i,j}!} \big(1 + O(n^{3\eta - 1}/\delta)\big).
\end{align} 
}
\end{enumerate}

\end{lemma}

\bigskip

\begin{proof}
~

\emph{Proof of \eqref{formdet}.} Writing $f_{\mathbf{a}}=f_{\mathbf{a}(1)}\otimes\ldots\otimes f_{\mathbf{a}(n)}$ (and similarly for $\Vb$) we obtain:
\begin{align*}
\langle \Va\vert U^{\otimes n} \Vb\rangle & = \prod_{1\leq c \leq \lambda _1} \prod_{1 \leq r \leq \lc} \langle f_{\tca(r)} \vert U f_{\tcb(r)}\rangle \notag \\
& =\prod_{1\leq c \leq \lambda _1} \prod_{1 \leq r \leq \lc} U_{\tca(r), \tcb(r)}.
\end{align*}
Since the subgroup of column permutations $\mathcal{C} _{\lambda }$ is the product of the permutation groups of each column, each $\sigma\in\mathcal{C} _{\lambda }$ can be decomposed as 
$
\sigma = s _1\dots s_{\lambda_{1}}
$
with $\ssc $ a permutation of column $c$ which transforms $\tcb(r)$ into $\tcb(\ssc(r))$.  Then we have
\begin{align*}
\langle\Va\vert q_{\lambda }  U^{\otimes n}  \Vb \rangle =\langle\Va\vert U^{\otimes n}  q_{\lambda }  \Vb \rangle & = \sum_{\sigma \in \mathcal{C} _{\lambda }}  \epsilon(\sigma) \prod_{1\leq c \leq \lambda _1}  \prod_{1\leq r\leq  \lc}  
U_{\tca(r), \tcb(\ssc(r))} \\
& = \prod_{1\leq c \leq \lambda _1} \sum_{\ssc \in S_c} \epsilon (\ssc) \prod_{1\leq r \leq  \lc}  U_{\tca(r), \tcb(\ssc(r))} \\
& = \prod_{1\leq c \leq \lambda _1} \det(U^{\tca,\tcb}).
\end{align*}

\emph{Proof of \eqref{cg0}.}
To find the number of admissible $\Va$ we multiply the possibilities of the three stages of the algorithm. For  $\Gamma(\Va)=0$ the first two stages yield $1$ as the combinatorial factor. Hence $\#\mathcal{V}^0$ is the number of possibilities at the third stage of the algorithm. Thus from \eqref{up3} we obtain the following upper bound for $\# \mathcal{V}^{0} ({\bf m})$: $\prod_{\ijr}(\lambda_i-\lambda_j)^{m_{i,j}} / m_{i,j}! $. Also, we  note that from the conditions \eqref{hyp} $\lambda_i - \lambda_j \geq \delta n / 2$ and 
$ |{\bf m}| \leq n^{\eta}$, so that using \eqref{do3} as a lower bound we obtain \eqref{cg0}.

\bigskip

\emph{Proof of \eqref{cggam}.}
Since $E$ is fixed, the number of $\Va$ in $\mathcal{V}^{E}$ is given by the third stage of the algorithm. We then apply \eqref{up3} with  $\sum_{\kappa} m_{\kappa} = |\mathbf{m}| - \Gamma$. Finally we  neglect the $m_{\kappa }!$ factors to obtain \eqref{cggam}.

\bigskip

\emph{Proof of \eqref{cgG}.}
Recalling (\ref{gamma_decomp}), we note that it is sufficient to give an upper bound to $\# \mathcal{V}^{E}$ since the first two stages of the algorithm yield a combinatorial factor of at most $C^{\Gamma}$, with $C= C(d)$.
Since $\sum m_{\kappa({i,j})} \geq |\mathbf{m}| - 2 \Gamma$, we have
$$\prod_{\kappa} m_{\kappa (i,j)}! \geq \prod_{\ijr} m_{i,j}! M^{-2\Gamma}$$ where $M=\sup_{\ijr}m_{i,j}\leq |\mathbf{m}|$. Note that  $\lambda_i -\lambda_j \geq \delta n / 2$  and using equation (\eqref{cggam}) in conjunction with the above result we obtain

\[
\# \mathcal{V}^{E} \leq n^{-\Gamma} \delta^{-2 \Gamma} |\mathbf{m}|^{2\Gamma} \prod_{\ijr}\frac{(\lambda_i-\lambda_j)^{m_{i,j}}}{m_{i,j}!} , \qquad \forall E ~{\rm with}~
\Gamma(E)= \Gamma .
\]
Finally we multiply by the number of possible $E$ (i.e. $C^{\Gamma}$) to obtain the desired result.

\bigskip

\emph{Proof of \eqref{prod_id}.}
Denote by $\tca([1,l(c)])$ 
the set of entries $\{ \tca(1), \dots, \tca(l(c))\}$ and consider equation \eqref{formdet} with $U = \mathbf{1}$. 
Since both $\Va$ and $\Vb$ are tensor product of basis vectors, the scalar product $\langle \Va \mid q_{\lambda} \Vb \rangle$ is equal to $-1$ or $1$ if $\tca([1,l(c)]) = \tcb([1,l(c)])$ for all columns, and $0$ otherwise. Since a modified  column cannot satisfy $\tca([1,l(c)]) = [1,l(c)]$ 
(and the same for ${\bf b}$) the condition given in the last line implies $\langle \Va \mid q_{\lambda} \Vb \rangle=0$ unless $t_{\bf{a}}$ and $t_{\bf{b}}$ have the same number of modified columns. Since the total number of unmodified columns for $\Va$ and $\Vb$ are $|\mathbf{l}| - \Gamma^a$  and $|\mathbf{m}| - \Gamma^b$ respectively, we obtain the first line of \eqref{prod_id}.

Now consider the case when $\Gamma^b = |\mathbf{m}| - |\mathbf{l}| + \Gamma^a$. Since  
 $|\langle \Va \mid q_{\lambda} \Vb \rangle|\leq 1$, we can bound the LHS of \eqref{prod_id} by the number of non-zero inner products $\langle \Va \mid q_{\lambda} \Vb \rangle$. In other words we want to count the number of $t_{\bf{b}}$ such that $t_{\bf a}^{c}([1,l(c)]) =t_{\bf b}^{c}([1,l(c)])$, or equivalently $S(\kappa_{\bf a}^c) = S(\kappa_{\bf b}^c)$. For building the relevant $\Vb$, we can follow the algorithm with the further condition that, at stage three the last condition (i.e. $S(\kappa_{\bf a}^c) = S(\kappa_{\bf b}^c)$) is met. 

The first two stages of the algorithm  yield a $C^{\Gamma^b}$ factor. In the third stage, 
for each $S$, we identify the column modifiers 
$\kappa_{1},\dots, \kappa_{r(S)}$ such that $S(\kappa_{i})=S$ for all 
$1\leq i\leq r(S)$. The total number of such objects is $m_{S}:= \sum_{i\leq r(S)} m_{\kappa_{i}}$ and the total number of ways in which they can be inserted to produce distinct tableaux is 
$$\left(\begin{array}{c}
m_S \\ m_{\kappa_1} \dots m_{\kappa_{r(S)}} \end{array}\right).
$$ 
Recall that the total number of elementary column-modifiers $\sum_{\ijr} m_{\kappa (i,j)}$ can be lower bounded by $|{\bf m}|- 2\Gamma^{b}$ and that each elementary column-modifier 
$\kappa (i,j)$ corresponds to a different $S(\kappa (i,j))= \{ (i,-), (j,+)\}$. Thus
$$
|{\bf m}|- 2\Gamma^{b} \leq \sum_{\ijr} m_{\kappa(i,j)} \leq 
\sum_{S}  \max_{\kappa : S(\kappa) = S} m_{\kappa} .
$$
Since
$$
\sum_{S} m_{S} = \sum_{\kappa} m_{\kappa} = |{\bf m}| - \Gamma^{b},
$$
we obtain
$$
\sum_{S} \left( m_S - \max_{\kappa : S(\kappa) = S}  m_{\kappa}\right) \leq \Gamma^b.$$ 
This implies
\[
\prod_S \left(
 \begin{array}{c}
 m_S \\ 
 m_{\kappa_1} \dots m_{\kappa_{r(S)}} 
 \end{array}
\right) \leq |\mathbf{m}|^{\Gamma^b}.
\]
Again we multiply by the $C^{\Gamma^{b}}$ of the first two stages and get \eqref{prod_id}.

\bigskip

\emph{Proof of \eqref{prod_id0}.}
From \eqref{prod_id} the only non-zero contributions come from 
$f_{\bf b}\in \mathcal{V}^{0}\subset \mathcal{O}_{\lambda}({\bf m})$.
Since $\Gamma^b = 0$, the combinatorial factor from the two first stages of the algorithm is $1$. Also since all the column modifiers are elementary we have $m_S = m_{i,j} = m_{\kappa({i,j})}$ for all $S$ (i.e. $0$ for the non-elementary ones) and the combinatorial factor is again one. In other words, the only $\Vb$ such that  $\langle \Va \mid q_{\lambda} \Vb \rangle \neq 0$ is $\Va$. Finally it is easy to check that $\langle \Va \mid q_{\lambda} \Va \rangle = 1$.

\bigskip

\emph{Proof of \eqref{interg0}.}
From \eqref{formdet} we have
$$
\langle \Va | q_{\lambda} U^r(\frak{z}/ \sqrt{n})^{\otimes n} \Vzero\rangle
= \prod_{1\leq c\leq \lambda_{1}} {\rm det} (U^{t_{\bf a}^{c}, {\rm Id}^{c}}) ,
\qquad  U= U^r(\frak{z}/\sqrt{n}).
$$
Now consider the following Taylor expansions. Entry-wise, for all $1\leq i\leq r$ on the first line, and all $\ijr$ on the second and third lines:
\begin{eqnarray*}
U^r_{i,i}(\frak{z}/\sqrt{n})  & =& 1 - \frac1{2n}\sum_{i<j\leq d}
\frac{|\frak{z}_{i,j}|^2}{\mu_i
-\mu_j}- \frac1{2n}\sum_{1\leq j<i}
\frac{|\frak{z}_{j,i}|^2}{\mu_i
-\mu_j}
+  O(\|\frak{z}\|^3 n^{-3/2} \delta^{-3/2}) ;
\\
&&\\
U^r_{i,j}(\frak{z}/\sqrt{n})  & =& - \frac1{\sqrt{n}} \frac{\frak{z}_{i,j}^*}{\sqrt{\mu_i -\mu_j}} + O(\|\frak{z}\|^2
n^{-1} \delta^{-1}) ;
\\
U^r_{j,i}(\frak{z}/\sqrt{n})  & = & \frac1{\sqrt{n}} \frac{\frak{z}_{i,j}}{\sqrt{\mu_i -\mu_j}} + O(\|\frak{z}\|^2
n^{-1} \delta^{-1}).
\end{eqnarray*}
If $\frak{z} =O(n^{\beta})$ and 
$\beta < 1/2$, 
the remainder terms are $O(n^{-3/2+3\beta}\delta^{-3/2})$ for the first line. Similarly the remainder terms for the last two lines are
$O(n^{-1+2\beta}\delta^{-1})$.
We use the Taylor expansions given above to find the dominating terms in the expansion $
\det A
=\sum_{\sigma} \prod_{i} \epsilon(\sigma) A_{i,\sigma(i)}.
$

Since $\Va \in \mathcal{V}^0$, all $\tca$ are either ${\rm Id}^{c}$, or an 
\emph{$(i,j)$-substitution}. If $\tca = {\rm Id}^{c}$,
the summands with more than two non-diagonal terms are of the same order as the remainder term, so we need to take into account only the identity and the transpositions in the sum $\sum_{\sigma } \prod_{i} A_{i,\sigma (i)}$. Let $l=l(c)$, then
\[
\upsilon(l):= \det(U^{{\rm Id}^{c}, {\rm Id}^{c}}(\frak{z}/ \sqrt{n})) =  
1  - \frac1{2n}\sumtwo{1\leq i \leq l}{l+1\leq j\leq d}
\frac{|\frak{z}_{i,j}|^2}{\mu_i -\mu_j} +
O(n^{-3/2 + 3 \beta}\delta^{-3/2} ).
\] 
Consider now the case $\tca\neq {\rm Id}^{c}$. Since $\tca(k)\geq k$ for all $k$,  there 
exists a whole column of $U^{\tca,{\rm Id}^{c}}$ whose entries are smaller in
modulus than $O(\|\frak{z}\|/\sqrt{n\delta}) = O(n^{-1/2+\beta}\delta^{-1})$. 
In particular if $\tca$ is an
\emph{$(i,j)$-substitution}, then the only summand that is of this
order comes from the identity permutation. So that 
\begin{equation}
\label{uij}
\upsilon(i,j): =\det(U^{\tca, {\rm Id}^{c}}(\frak{z}/\sqrt{n})) = 
\frac{\frak{z}_{i,j} }{\sqrt{n(\mu_i -\mu_j)}} +
O(n^{-1 + 2\beta}\delta^{-1}).
\end{equation}
 Note that each $\ijr$ there are  $m_{i,j}$ columns of the type $(i,j)$-substitution. Out of the $\lambda_l-\lambda_{l+1}$ columns of length $l=l(c)$, there are 
$\lambda_l-\lambda_{l+1} - R_l$ of the type ${\rm Id}^{c}$, with $0\leq R_l\leq
\lvert{\bf m}\rvert$. We are now ready to expand \eqref{formdet} in terms of our estimates.
\begin{equation}
\label{interg0_et}
\langle \Va | q_{\lambda} U^r(\frak{z}/\sqrt{n})^{\otimes n} \Vzero\rangle 
 = \prod_{l=1}^r\left( \upsilon(l)) \right)^{\lambda_l -\lambda_{l+1}}\prod_{\ijr}
\left(\upsilon(i,j) \right)^{m_{i,j}} \prod_{l=1}^r (\upsilon(l))^{-R_l}.
\end{equation}
Now $\upsilon(l) = 1 + O(n^{-1 + 2\beta}\delta^{-1})$ 
and $R_l\leq | {\bf m}| \leq n^{\eta}$, so the last product is 
$1+O(n^{-1 + 2 \beta + \eta}\delta^{-1} )$.
Similarly, since $\lambda\in\Lambda_{2}$ we have $\lambda_l-
\lambda_{l+1} = n (\mu_l -\mu_{l+1}) + O(n^{\alpha})$. We need the following lemma (see \cite{Kahn&Guta}) for an exponential approximation:
\begin{lemma}
\label{Taylor}
If $x_{n}= O(n^{1/2-\epsilon})$, then
\[
\left(1+\frac{x_{n}}{n}\right)^n  = \exp(x_{n})(1+ O(n^{-\epsilon})).
\]
\end{lemma}
Using 
Lemma \ref{Taylor} we estimate the first product   in (\ref{interg0_et}) as follows
\begin{align}
\prod_{l=1}^{r} \upsilon(l)^{\lambda_l -\lambda_{l+1}} 
& =  \prod_{l=1}^{r}
\exp\left( - \frac1{2}\sumtwo{1\leq i\leq l }{l+1\leq j\leq d}
|\mathfrak{z}_{i,j}|^2 \frac{\mu_l-\mu_{l+1}}{(\mu_i -\mu_j)}
\right) r(n)\nonumber\\
& = 
\exp\left( - \frac1{2}\sumtwo{1\leq i\leq r }{i< j\leq d}\sum_{l=i}^{r\wedge(j-1)}
|\mathfrak{z}_{i,j}|^2 \frac{\mu_l-\mu_{l+1}}{(\mu_i -\mu_j)}
\right) r(n)\nonumber\\
&= \exp\left( - \frac{\|\mathfrak{z} \|^2_2}{2}\right) r(n)\label{first_prod},
\end{align}
with 
\begin{eqnarray*}
r(n)  &=& 1+O(n^{-1 + \alpha + 2\beta} \delta^{-1}, n^{-1/2 + 2\beta}\delta^{-1}), \\
\|\mathfrak{z} \|^2_2  &=& \sumtwo{1\leq i\leq r }{i< j\leq d}
|\mathfrak{z}_{i,j}|^2.
\end{eqnarray*}
Here we have used the relation that 
$$\sum_{l=i}^{r\wedge(j-1)}
 \frac{\mu_l-\mu_{l+1}}{(\mu_i -\mu_j)}=1$$
 which is trivial for $j\leq r$ and for $j>r$ we have used the fact that both $\mu_j$ and $\mu_{r+1}$ are $0$.
We now estimate the middle product in the RHS of \eqref{interg0_et} as
\begin{equation}
\upsilon(i,j)^{m_{i,j}}   =\left(\frac{\mathfrak{z}_{i,j}}{\sqrt{n}\sqrt{\mu_i -\mu_j}} \right)^{m_{i,j}}\left(1 +
O\left(n^{-1 + 2 \beta+\eta}\delta^{-1}\right)\right),\label{mid_prod}
\end{equation}
where we have used that $|\mathbf{m}|\leq n^{\eta}$. From \eqref{interg0_et}, (\ref{first_prod}) and (\ref{mid_prod}), we obtain \eqref{interg0}.\\ 

\bigskip
\emph{Proof of \eqref{interg}.}
We have the following analog of (\ref{interg0_et}),
\[
\langle \Va | q_{\lambda} U^r(\frak{z}/\sqrt{n})^{\otimes n} \Vzero\rangle 
 = \prod_{l=1}^r\left( \upsilon(l)) \right)^{\lambda_l -\lambda_{l+
1}}\prod_{\kappa}
\left(\upsilon(\kappa) \right)^{m_{\kappa}} \prod_{l=1}^r (\upsilon(l))^{-R_l}
\]
where $ 0\leq R_l\leq |{\bf m}| -\Gamma$ and $\upsilon(\kappa)$ is the determinant of the minor of $U$ corresponding to the column where $\kappa$ is applied. We further split the column-modifiers into elementary ($\kappa=\kappa(i,j)$) and non-elementary (say $\kappa=\kappa^{\prime})$ ones. Then $\langle \Va | q_{\lambda} U^r(\frak{z}/\sqrt{n})^{\otimes n} \Vzero\rangle$ can be written as
\[
\prod_{l=1}^r\left( \upsilon(l)) \right)^{\lambda_l -\lambda_{l+1}}  
\prod_{\ijr}
\left(\upsilon(i,j) \right)^{m_{\kappa(i,j)}} 
\prod_{l=1}^r (\upsilon(l))^{-R_l}
\prod_{\kappa^{\prime}} 
\left(\upsilon(\kappa^{\prime}) \right)^{m_{\kappa^{\prime}}} .
\]

The first three products on the RHS can be approximated like the terms in \eqref{interg0_et}. 
 To give an upper bound to the fourth product, we note that if the entries in the column have been modified in an admissible way, then 
$\tca (i) =j>l(c)$ for some $i$, so that $|\upsilon(\kappa)|\leq C\|\mathfrak{z}\|/\sqrt{n\delta}$ for any $\kappa$, with some constant $C=C(d)$. 

Thus we have
\begin{equation}
\left|\langle \Va | q_{\lambda} U^r(\frak{z}/\sqrt{n})^{\otimes n} \Vzero\rangle \right| \leq 
\exp\left(-\frac{\| \mathfrak{z} \|_{2}^{2} }{2}\right) 
\left(\frac{ C\|\mathfrak{z}\|}{\sqrt{n\delta}}\right)^
{\sum_{\kappa^{\prime}} m_{\kappa^{\prime}}}
\prod_{\substack{1\leq i\leq r\\i<j \leq d}}\left(\frac{|\mathfrak{z}_{i,j}|}{\sqrt{n} \sqrt{\mu_{i}-\mu_{j}}}\right)^{m_{\kappa(i,j)}} r(n).
\label{gg} 
\end{equation}

We obtain \eqref{interg} by noting that the number of non-elementary modifiers is 
$$
\Upsilon=\sum_{\kappa^{\prime}} m_{\kappa^{\prime}} = - \Gamma +
\sum_{\ijr} (m_{i,j} - m_{\kappa(i,j)} ).
$$

\bigskip

\emph{Proof of \eqref{numer}.}
 
Recall that we identify $\bz$ with the vector $(\bz,0)$. We shall split the sum over $\mathcal{O}_{\lambda}({\bf m})$ into  sub-sums using the decomposition 
$\mathcal{O}_{\lambda}({\bf m}) \bigcap \mathcal{V}= \bigcup_{E} \mathcal{V}^{E}({\bf m})$, and compare each sub-sum against $\mathcal{V}^{0}=\mathcal{V}^{E^{0}}$. Then we aggregate all the $E$ corresponding to a particular $\Gamma$ (which results in a multiplication by a combinatorial factor) and finally sum over different $\Gamma$. From the given bounds on $\frak{z}$ and $\bz$ we note that
$\|\frak{z} + \bz\| =O(n^{\beta})$, so we can apply the previous results with 
$\frak{z} + \bz$ instead of $\frak{z}$.

Using \eqref{cg0} and \eqref{interg0} and recalling that $\lambda \in \Lambda_2$, we get:
\begin{eqnarray*}
&&\left\langle \sum_{\Va \in \mathcal{V}^0} \Va \Bigg| q_{\lambda} U^r(\frak{z} + \bz
/\sqrt{n})^{\otimes n} \Vzero \right\rangle\\ 
=&& 
 \exp\left(-\frac{\|\mathfrak{z} + \bz\|^2_2}{2}\right) \prod_{\ijr}
 \frac{\big((\mathfrak{z}_{i,j} + z_{i,j})\sqrt{n}\sqrt{\mu_i - \mu_j}\big)^{m_{i,j}}}{m_{i,j}!} r(n)
\end{eqnarray*}
with error factor
\[
r(n)  = 1 + O\left(n^{-1 + 2\beta + \eta}\delta^{-1}, n^{-1/2 + 2\beta} \delta^{-1}, 
n^{-1 +2\beta +\alpha}\delta^{-1}, n^{-1 + 2\eta}\delta^{-1}, n^{-1 + \alpha + \eta}\delta^{-1} \right).
\]

For $E \neq E^{0}$ we combine \eqref{interg} and \eqref{cggam}  to obtain
\begin{eqnarray*}
&&
 \left|\left\langle 
\sum_{\Va \in \mathcal{V}^{E}} \Va \Bigg| q_{\lambda} U((\frak{z} + \bz)/\sqrt{n}) \Vzero \right\rangle \right| \cdot
\left|
\left\langle 
\sum_{\Va \in \mathcal{V}^0} \Va \Bigg| q_{\lambda} U((\frak{z} + \bz)/\sqrt{n}) \Vzero 
\right\rangle
 \right|^{-1}  \\ 
\leq && 
n^{-\Gamma} \prod_{\ijr} \left(\frac{\lambda_i - \lambda_j}{n}\right)^{m_{\kappa({i,j})} - m_{i,j}}\frac{m_{i,j}!}{m_{\kappa({i,j})}!} \left(\frac{\|\mathfrak{z} + \bz\|}{\sqrt{\delta n}}\right)^{- \Gamma}\\
&&\prod_{\ijr}\left(\frac{\sqrt{\delta n} |\mathfrak{z}_{i,j} + z_{i,j}|}{\|\mathfrak{z} + \bz\|\sqrt{n}\sqrt{\mu_i - \mu_j}}\right)^{m_{\kappa({i,j})} - m_{i,j}} r(n) \\
\leq &&  O(n^{-\Gamma (1/2 -\beta)})\delta^{-\Gamma/2} 
\prod_{\ijr:m_{i,j}\neq 0}\left(
\frac{
|\mathfrak{z}_{i,j} + z_{i,j} |\sqrt{\mu_i - \mu_j}}{m_{i,j} }\right)^{m_{\kappa(i,j)} - m_{i,j}  }  \\
\leq &&  O\big((2\delta^{-3/2} n^{-1/2 + 3 \beta + 2 \epsilon})^\Gamma\big),
\end{eqnarray*}
with $O(\cdot)$ uniform in $\Gamma$. In the second inequality we have used  
$$
m_{i,j}!/m_{\kappa(i,j)}! \leq m_{i,j}^{m_{i,j}-m_{\kappa(i,j)}}, \qquad 
\sum_{\ijr} (m_{\kappa({i,j})} - m_{i,j}) \geq - 2 \Gamma, \qquad 
\lambda\in \Lambda_2
 $$ 
and in the third inequality we have used
$$
m_{i,j}\leq 2 |\mathfrak{z}_{i,j}+ z_{i,j}| n^{\beta + \epsilon}
.$$

Again note that for a given $\Gamma$, there are at most $C^{\Gamma}$ different $E$ such that $\Gamma(E) = \Gamma$ (where $C=C(d)$) corresponding to the possible choices in the first two stages of the algorithm, . Hence, if $n$ is large enough such that 
$ 2C\delta^{-3/2} n^{-1/2 + 3 \beta + 2 \epsilon} <1 $, we have:

\begin{eqnarray*}
&&
\left\langle \sum_{\Va \in \mathcal{O}_{\lambda}(\mathbf{m})} \Va \Bigg| q_{\lambda} U((\frak{z} + \bz)/\sqrt{n}) \Vzero \right\rangle 
 = \sum_{\Gamma}\left\langle \sum_{\Va \in \mathcal{V}^{\Gamma}} \Va \Bigg| q_{\lambda} U((\frak{z} + \bz)/\sqrt{n}) \Vzero \right\rangle \\
&&
 = \Big(1 + O(\delta^{-3/2} n^{-1/2 + 3 \beta + 2 \epsilon})\Big)
\exp\left(-\frac{\|\mathfrak{z} + \bz\|^2_2}{2}\right) \prod_{\ijr}\frac{\big((\mathfrak{z} + z)_{i,j}(\sqrt{n}\sqrt{\mu_i - \mu_j})\big)^{m_{i,j}}}{m_{i,j}!} r(n) \\
&& =
\exp\left(-\frac{\|\mathfrak{z} + \bz\|^2_2}{2}\right) \prod_{\ijr}\frac{\big((\mathfrak{z} + z)_{i,j}(\sqrt{n}\sqrt{\mu_i - \mu_j})\big)^{m_{i,j}}}{m_{i,j}!} r_2(n)
\end{eqnarray*}
where the sum over $\Gamma$ was bounded using a geometric series and   
\[
r_2(n) = 1 + O\left(n^{-1 + 2\beta + \eta}\delta^{-1}, 
n^{-1 + \alpha + \beta}\delta^{-1}, n^{-1 + 2\eta}\delta^{-1}, n^{-1 + \alpha + \eta}\delta^{-1}, \delta^{-3/2} n^{-1/2 + 3 \beta + 2 \epsilon} \right).
\]
Thus we obtain \eqref{numer}.

\bigskip

\emph{Proof of \eqref{denom}.}
We choose $\Gamma^{a}$ and $\Gamma^{b}$ satisfying the condition 
$\Gamma^b - \Gamma^a = |{\bf m}| - |{\bf l}|$ to ensure that the inner products in 
\eqref{prod_id} are non-zero. By multiplying \eqref{cgG} and \eqref{prod_id}, we see that:
\begin{align}  
\label{refgam}
\left\vert\left\langle \sum_{\Va \in \mathcal{V}^{\Gamma^{a}}(\mathbf{l})}\Va \Bigg| q_{\lambda}  \sum_{\Vb \in \mathcal{V}^{\Gamma^{b}}({\bf m})
}\Vb \right\rangle \right\vert
& \leq (C |\mathbf{m}|)^{\Gamma^b} \!\prod_{\ijr} \frac{(\lambda_i - \lambda_j)^{l_{i,j}}}{l_{i,j}!} \left(\frac{C|\mathbf{l}|^2}{n\delta^2}\right)^{\Gamma^a} \\
& = (C |\mathbf{m}|)^{|\mathbf{m}| - |\mathbf{l}|}\prod_{\ijr} \frac{(\lambda_i - \lambda_j)^{l_{i,j}}}{l_{i,j}!} \left(\frac{C|\mathbf{l}|^2|\mathbf{m}|}{n\delta^2}\right)^{\Gamma^a} .\notag
\end{align}

Since $\Gamma^b$ is determined by $\Gamma^a$ it is sufficient to sum over all possible $\Gamma^a$. If $n^{1 - 3 \eta} >  2 C / \delta^{2}$, the dominating term in the sum of bounds is that corresponding to the smallest possible $\Gamma^a$. We aim to give a give a lower bound to $\Gamma^a$ under the constraint that the inner product is non-zero.

As noted before the necessary condition for the inner product to be non-zero is $m_S^{\bf a} = m_S^{\bf b}$. We consider a lower bound for $\Gamma(\Va) + \Gamma(\Vb)$ which will in turn give a lower bound for $\Gamma(\Va)$. We start with two tableaux $t_{\bf a^{\prime}}\in \mathcal{O}_{\lambda}({\bf m})$ and $t_{\bf b^{\prime}}\in \mathcal{O}_{\lambda}({\bf l})$ so that $\Gamma(f_{{\bf a}^{\prime}})=\Gamma(f_{{\bf a}^{\prime}})=0$. Since $|{\bf m}| \leq n^{\eta}$ and  $\lambda_{r} \geq \delta n + O(n^{\alpha})$, the tableau $t_{\bf a^{\prime}}$ can be chosen to have at most one modified box per column so the last condition is met (and similarly for $t_{\bf b^{\prime}}$). We also ensure that each of the modified columns of $t_{\bf a}$ (or $t_{\bf b}$) are also modified in $t_{\bf a^{\prime}}$ (or $t_{\bf b^{\prime}}$). The idea is to count the minimum number of `horizontal box shuffling' operations necessary 
in order to transform the Young tableau $t_{\bf a^{\prime}}$ into the tableau $t_{\bf a}$ and $t_{\bf b^{\prime}}$ into the given tableau $t_{\bf b}$ so that at the end $m_S^{\bf a} = m_S^{\bf b}$.

The horizontal box shuffling is described as follows; at each step we horizontally move one elementary column modifier $\kappa(i,j)$ of $t_{{\bf a}^{\prime}}$ (or $t_{{\bf b}^{\prime}}$) into an already modified 
column, with the aim of constructing $t_{\bf a}$ (or $t_{\bf b}$). Note that such operation increases $\Gamma(f_{{\bf a}^{\prime}}) + \Gamma(f_{{\bf b}^{\prime}})$ by one  and also changes 
$m_{S}^{{\bf a}^{\prime}}$ (or $m_{S}^{{\bf b}^{\prime}}$) in the following way. The multiplicities
$m_{\{(i,-),(j,+)\}}$ and $m_{S_{0}}$ decrease by one, and $m_{S_{0} +  \{(i,-),(j,+)\}}$ increases by one. Here $S_{0}$ is the set of modified entries for the column to which the box $(i,j)$ is moved. Hence the distance 
$\sum_{S} |m_S^{{\bf a}^{\prime}} - m_S^{{\bf b}^{\prime}}|$ 
decreases by at most three. Since initially this quantity was equal to  $\sum_{\ijr} | l_{i,j} - m_{i,j}|$ (because all the column modifiers were elementary in $t_{{\bf a}^{\prime}}$ and $t_{{\bf b}^{\prime}}$), we need at least $\sum_{\ijr} |{l}_{i,j} - {m}_{i,j}| / 3$ such operations before reaching our goal 
$m_S^{\bf a} = m_S^{\bf b}$. Hence we have $\Gamma(\Va) + \Gamma(\Vb) \geq |\mathbf{l} - \mathbf{m}| / 3$.

Finally using $\Gamma^b - \Gamma^a = |{\bf m}| - |{\bf l}|$,  we obtain $\Gamma^a \geq (|\mathbf{l} - \mathbf{m}| + 3 |\mathbf{l}| - 3 |\mathbf{m}|) / 6 $ and hence the desired result (\eqref{denom}).

\bigskip

\emph{Proof of \eqref{denom54}.}
Since $\mathbf{l} = \mathbf{m}$,  equations \eqref{cg0} and \eqref{prod_id0} prove that the bound \eqref{refgam} is saturated when $\Gamma^a = 0$, up to the error factor $\left(1 + O(n^{-1 + 2 \eta}/\delta)\right)$. Hence the remainder 
term due to the other $\Gamma$ consist in a geometric series with factor 
$\left(\frac{C|\mathbf{m}|^3}{n\delta^2}\right) = O(n^{1 - 3\eta}/\delta^{2})$. 

\end{proof}

\subsection{Non-orthogonality issues and construction of $V^r_{\lambda}$}
\label{preuve_quasi_orth}
In this subsection we consider the basis vectors $|\bf{m}_{\lambda}\ra$ and show that they are approximately orthogonal. We will use this result to construct the isometry $V^r_{\lambda}$ and hence the channel $T^r_{\lambda}$.

\begin{lemma}
\label{non-orth}
Let $\mathbf{m}_\lambda$ and $\mathbf{l}_\lambda$ be two basis vectors (corresponding to two semistandard Young tableaux with diagram $\lambda$) and let $|{\bf m}|:=\sum_{\ijr}m_{ij}$ and 
$|{\bf l}-{\bf m}| := \sum_{\ijr} | l_{i,j}-m_{ij}|$.

If 
$$
\sum_{j > i} m_{i,j} - \sum_{j < i} m_{j,i} \neq \sum_{j > i} l_{i,j} - \sum_{j < i} l_{j,i}
$$ for some $1\leq i \leq r$, then   
\[
\langle \mathbf{m}_\lambda | \mathbf{l}_\lambda \rangle  = 0.
\]

Otherwise, we derive an upper bound under the following conditions.
We assume that $\lambda_i - \lambda _{i+1} > \delta n $ for all $1\leq i \leq r-1$  and $\lambda _r > \delta n$, for some $\delta > 0$. Furthermore we assume 
$|\mathbf{l}| \leq |\mathbf{m}| \leq n^{\eta}$ for some $\eta < 1/ 3$ and that
 $Cn^{3\eta - 1}/\delta^{2} <1$ where $C=C(d)$ is a  constant.

Then: 
\begin{equation}\label{quasiorth1}
\left|\langle \mathbf{m}_ {\lambda} | \mathbf{l}_ {\lambda} \rangle\right| \leq
 (C^{\prime}n)^{-\eta( |{\bf m}|-|{\bf l}|)/4}\,
  (C^{\prime}n)^{(9\eta-2)|{\bf m} -{\bf l}|/12}\,
  \delta^{ (|\mathbf{m}| -|\mathbf{l}|)/2 - |\mathbf{m} - \mathbf{l}|  / 3}\,
   (1+O(n^{-1+3\eta} /\delta)).
\end{equation}
where $C^{\prime}=C^{\prime}(d,\eta)$ and the constant in the remainder term depends only on $d$. The right side is of order less than $n^{(9 \eta -2)|{\bf m} - {\bf l}| / 12 }$ and converges to zero for $\eta < 2 / 9$  when $n\to \infty$.

\end{lemma}

\begin{proof} 
We know that $\left\vert \mathbf{m} _{\lambda}\right\rangle$ is a linear combination 
of $n$-tensor product vectors in which the basis vector $f_i$ appears exactly 
$\lambda_i - \sum_{j > i} m_{i,j} + \sum_{j < i} m_{j,i}$ times for $1\leq i\leq r$. As two tensor basis vectors are orthogonal if they do not have the same number of $f_i$ in the decomposition, we get that $\langle \mathbf{m}_ {\lambda} | \mathbf{l}_ {\lambda} \rangle = 0$ if $\sum_{j > i} m_{i,j} - \sum_{j < i} m_{j,i} \neq \sum_{j > i} l_{i,j} - \sum_{j < i} l_{j,i}$ for any $1 \leq i \leq r$.

\medskip

Now consider the following expansion of the inner product,
\begin{align}
\langle \mathbf{m}_{\lambda} | \mathbf{l}_ {\lambda} \rangle &= \frac{\langle q_{\lambda} p_{\lambda} \Vm | q_{\lambda} p_{\lambda} \Vl\rangle }{\sqrt{\langle q_{\lambda} p_{\lambda} \Vm | q_{\lambda} p_{\lambda} \Vm\rangle \langle q_{\lambda} p_{\lambda} \Vl | q_{\lambda} p_{\lambda} \Vl\rangle}}\\
&=\frac{\langle  p_{\lambda} \Vm | q_{\lambda} p_{\lambda} \Vl\rangle }{\sqrt{\langle  p_{\lambda} \Vm | q_{\lambda} p_{\lambda} \Vm\rangle \langle  p_{\lambda} \Vl | q_{\lambda} p_{\lambda} \Vl\rangle}}
\end{align}

where we have used the fact that $q_{\lambda}$ is a projection (use (\ref{projection_formula}) and (\ref{csq1}) except that we have only $r$ rows in our Young diagrams), up to a constant factor. Now we decompose $p_{\lambda} f_{\bf m} $ and $p_{\lambda} f_{\bf l}$ 
on orbits as in \eqref{plambda.orbit.decomposition}. 
Since the multiplicity of the elements in the orbits are the same in numerator and denominator, we obtain:
\begin{align}
\langle \mathbf{m}_{\lambda} | \mathbf{l}_ {\lambda} \rangle = \frac{\langle \sum_{\Va \in \mathcal{O}_{\lambda}(\mathbf{m})} \Va 
| q_{\lambda} \sum_{\Vb \in \mathcal{O}_{\lambda}(\mathbf{l})} \Vb \rangle}
{\sqrt{\langle \sum_{\Va \in \mathcal{O}_{\lambda}(\mathbf{m})} \Va | 
q_{\lambda}\sum_{\Vap \in \mathcal{O}_{\lambda}(\mathbf{m}) } \Vap\rangle
\langle \sum_{\Vb \in \mathcal{O}_{\lambda}(\mathbf{l})} \Vb | q_{\lambda} \sum_{\Vbp \in \mathcal{O}_{\lambda}(\mathbf{l})} \Vbp \rangle}}.
\end{align}

We use \eqref{denom54} for the denominator:
\begin{eqnarray*}
&&
\left\langle \sum_{\Va \in \mathcal{O}_{\lambda}(\mathbf{m})} \Va \Bigg| 
q_{\lambda}\sum_{\Vap \in \mathcal{O}_{\lambda}(\mathbf{m}) } \Vap\right\rangle
\Bigg\langle \sum_{\Vb \in \mathcal{O}_{\lambda}(\mathbf{l})} \Vb \Bigg| q_{\lambda} \sum_{\Vbp \in \mathcal{O}_{\lambda}(\mathbf{l})} \Vbp \Bigg\rangle \\
&&=  
\prod_{\ijr}\frac{ (\lambda_i -\lambda_j)^{(m_{i,j} + l_{i,j}) / 2}}{\sqrt{m_{i,j}! \,l_{i,j}!}}
(1+O(n^{3\eta - 1}/\delta)),
\end{eqnarray*}
and the numerator is bounded as in \eqref{denom}. Then, under the assumption 
$|{\bf m} | \geq |{\bf l}|$ we have
\[
\left|\langle \mathbf{m}_ {\lambda} | \mathbf{l}_ {\lambda} \rangle\right| 
\leq 
(C |\mathbf{m}|)^{|\mathbf{m}| - |\mathbf{l}|} \left(\frac{C |\mathbf{m}|^3}{\delta ^2 n}\right)^{\Gamma_{min}} \cdot \prod_{\ijr} (\lambda_i - \lambda_j)^{(l_{i,j} - m_{i,j}) / 2} \sqrt{\frac{m_{i,j}!}{l_{i,j}!}} \cdot \left(1 +  O(n^{3\eta-1}/\delta) \right), 
\]
where $\Gamma_{min} = (   (|\mathbf{l} - \mathbf{m}| + 3 |\mathbf{l}| - 3 |\mathbf{m}|) / 6) \wedge 0$.

The factorials can be bounded as 
$$
\prod_{\ijr} \sqrt{\frac{m_{i,j}!}{l_{i,j}!}} \leq |\mathbf{m}|^{\sum (m_{i,j} - l_{i,j})_{+}/2} = 
|\mathbf{m}|^{(|\mathbf{m} - \mathbf{l}| + |\mathbf{m}| - |\mathbf{l}|) / 4 }.
$$ 
Since $|{\bf m}|\leq n^{\eta}$ and  $Cn^{3\eta-1}/\delta^{2}< 1$, we have
$$
\left(\frac{C |\mathbf{m}|^3}{\delta ^2 n}\right)^{\Gamma_{min}} \leq
\left(\frac{C n^{3\eta-1}}{\delta ^2}\right)^{  (|\mathbf{l} - \mathbf{m}| + 3 |\mathbf{l}| - 3 |\mathbf{m}|) / 6}.
$$
Since $n\delta <\lambda_{i}-\lambda_{j}<n$ we have
$$(\lambda_i-\lambda_j)^{(l_{i,j} - m_{i,j}) / 2}\leq n^{{(l_{i,j} - m_{i,j}) / 2}} \quad \text {  if }l_{i,j}\geq m_{i,j}$$ 
and 
$$(\lambda_i-\lambda_j)^{(l_{i,j} - m_{i,j}) / 2}\leq (n\delta)^{{(l_{i,j} - m_{i,j}) / 2}}\quad \text {  if }l_{i,j}< m_{i,j}$$ 

so that we have 
$$\prod_{\ijr} (\lambda_i -\lambda_j)^{(l_{i,j} - m_{i,j}) / 2} \leq
n^{ (|{\bf l}|-|{\bf m}|)/2}\delta^{\sum (m_{i,j} - l_{i,j})_{+}/2}\leq (n\delta)^{ (|{\bf l}|-|{\bf m}|)/2}.
$$

The last inequality follows since $\delta<1$ and by assumption $|{\bf l}|-|{\bf m}|\leq 0$.

The constant $C=C(d)$ can be replaced by another constant 
$C^{\prime}= C^{\prime}(d,\eta)$ such that all powers of $n$ appear in the form $(C^{\prime}n)^{\gamma}$. The above bound yields
\begin{eqnarray*}
\left|\langle \mathbf{m}_{\lambda} | \mathbf{l}_ {\lambda} \rangle\right| \leq
 \delta^{ (|\mathbf{m}| -|\mathbf{l}|)/2 - |\mathbf{m} - \mathbf{l}|  / 3}
 (C^{\prime}n)^{-\eta( |{\bf m}|-|{\bf l}|)/4} (C^{\prime}n)^{(9\eta-2)|{\bf m} -{\bf l}|/12} (1+O(n^{-1+3\eta} /\delta))
\end{eqnarray*}
\end{proof}

We have the following lemma which shows that for a given $\mathbf{m}_{\lambda}$, its overlap with $\mathbf{l}_{\lambda}$s is uniformly small.
\begin{lemma}
\label{gqo}
Let $\eta < 2 / 9$ and let $\mathbf{m}_{\lambda}$ be such that $|\mathbf{m}| \leq n^{\eta}$. Assume as in Lemma \ref{non-orth} that $\lambda_i - \lambda _{i+1} > \delta n $ for all $1\leq i \leq r-1$  and $\lambda _r > \delta n$, for some $\delta > 0$, and that
 $Cn^{3\eta - 1}/\delta^{2} <1$ where $C=C(d)$ is a constant.

Then there exists a constant $C^{\prime\prime}= C^{\prime\prime}(d,\eta)$ such that
\begin{equation}\label{bound.sum.inner.prod}
\sumtwo{|\mathbf{l}| \leq n^{\eta}}{\mathbf{l} \neq \mathbf{m}} 
\left| \langle \mathbf{m}_{\lambda} | \mathbf{l}_{\lambda} \rangle  \right| \leq (C^{\prime\prime} n)^{(9 \eta - 2) / 12}\delta^{-1/3}.
\end{equation}
\end{lemma}

\begin{proof}
We break the sum into two parts ($|{\bf l}| \leq |{\bf m}|$ and $|{\bf l}| >  |{\bf m}|$), and we note it suffices to prove the statement under the additional condition 
$|{\bf l}| \leq |{\bf m}|$ (otherwise we use \eqref{quasiorth1} with the roles of $\bf{l}$ and $\bf{m}$ exchanged).
  
We use \eqref{quasiorth1} neglecting the terms containing $|{\bf m}|-|{\bf l}|$ in the exponent which are less than $1$. Then the expression on the left side of 
\eqref{bound.sum.inner.prod} is bounded from above by
$$
2 \sum_{k\geq 1} N(k) \left[ (C^{\prime}n)^{(9\eta-2)/12} \delta^{-1/3} \right]^{k}
$$
where $N(k)$ is the number of ${\bf l}$'s for which $|{\bf m} - {\bf l}|= k$. 

Since there are at most $rd$ pairs $\ijr$, there are at most 
$(k+1)^{rd}$ different choices for the  values $\{| l_{i,j} - m_{i,j}| : \ijr\}$ satisfying 
$\sum  |l_{i,j} - m_{i,j}| = k$. Moreover, there are $2^{rd}$ sign choices which fix ${\bf l}= \{l_{i,j}\}$ completely. Thus $N(k) \leq (2(k+1))^{rd} \leq c^{k}$ for some constant $c$. Thus we have a geometric series with the first term corresponding to the RHS of (\ref{bound.sum.inner.prod}) and hence the desired result is proved.

\end{proof}
We use this quasi-orthogonality to build an isometry
$V^r_{\lambda}:\mathcal{H}_{\lambda} \to \mathcal{F}^r$ which maps the relevant 
finite-dimensional vectors $\left\vert {\bf m}_ \lambda\right\rangle$ `close' to their Fock counterparts $\left\vert {\bf m}\right\rangle$, with $| \mathbf{m} \rangle \in \mathcal{F}^{r}$. This is the aim of following two lemmas.

\begin{lemma}
\label{isocomp}
Let $A$ be a contraction (i.e. $A^* A \leq \bf{1}$) from a finite space $\mathcal{H}$  to an infinite space $\mathcal{K}$. Then there is an $R:\mathcal{H}\to \mathcal{K}$ such that $A + R$ is an isometry and ${\rm Range}(A) \perp {\rm Range}(R)$.

As a consequence, for any unit vector $\phi$, we have $\| R \phi \|^2 = 1 - \| A \phi \|^2$.
\end{lemma}

\begin{proof}
See \cite{Kahn&Guta}.
\end{proof}


\begin{lemma}{\label{V_approx}}
Let $\eta < 2/9$. Suppose that $\lambda _i - \lambda _{i+1} \geq \delta n$ for all 
$1\leq i\leq r$, with the convention $\lambda _{r+1} = 0$. Then for 
$n>n_{0}(\eta,\delta,d)$ there exists an isometry $V^r_{\lambda}: \mathcal{H}_{\lambda}\to \mathcal{F}^r$ such that, 
$V^r_{\lambda}|{\bf 0}_\lambda \rangle = |{\bf 0}\rangle $ and for $|\mathbf{m}| \leq n^{\eta}$, 
\[
\langle  \mathbf{m}| V^r_{\lambda}  = \frac1{\sqrt{1 + (\tilde{C}n)^{(9\eta -2) / 12}/\delta ^{1/3}}} \left\langle \mathbf{m}_\lambda \right|
\]
where $\tilde{C}= \tilde{C}(\eta,d)$ is a particular constant. More precisely, 
$n_{0}$ can be taken of the form $\left(\frac{C(d)}{\delta^{2}}\right)^{1/(1-3\eta)}$.
\end{lemma}
\begin{proof}
Let $A^r_{\lambda}:\mathcal{H}_{\lambda}\to \mathcal{F}^r$ be defined by
\[
A^r_{\lambda} := \frac1{\sqrt{1 + (Cn)^{(9\eta -2) / 12}/\delta ^{1/3}}} \sum_{|\mathbf{l}| \leq n^{\eta}} | \mathbf{l} \rangle \left\langle \mathbf{l}_{\lambda} \right|.
\]

where $| \mathbf{l} \rangle \in \mathcal{F}^{r}$.
Then,  
\begin{align*}
A_{\lambda}^{r*} A^r_{\lambda} & =  \frac1{1 + (Cn)^{(9\eta -2) / 12}/\delta ^{1/3}} \sum_{|\mathbf{l}| \leq n^{\eta}} \left| \mathbf{l}_{\lambda} \right\rangle \left\langle \mathbf{l}_{\lambda} \right| \\
& \leq \mathbf{1}_{\mathcal{H}_{\lambda}}. 
\end{align*}
where the last inequality follows from the following argument. 
It is enough to show that all eigenvalues of $A_{\lambda}^{r*} A^r_{\lambda}$ are smaller than $1$. Let $\sum_{\bf m} c_{\bf m} \left| \mathbf{m}_ {\lambda} \right\rangle$ be an eigenvector of $A_{\lambda}^{r*} A^r_{\lambda}$, and $a$ the corresponding eigenvalue. Then by the linear independence of $\left\vert \mathbf{m}_ {\lambda} \right\rangle $  we get that for each~${\bf l}$
$$
 \frac1{1 + (Cn)^{(9\eta -2) / 12}/\delta ^{1/3}} 
\sum_{|\mathbf{m}|\leq n^{\eta}}  
\left\langle \mathbf{l}_\lambda  | \mathbf{m}_ {\lambda} \right\rangle  c_{\bf m} 
 = 
a   c_{\bf l}.
$$
If ${\bf l}_{0}$ is an index for which $| c_{\bf l}|$ is maximum, then by taking absolute values on both sides and dividing by $|c_{\bf l_0}|$, we obtain
$$
a \leq  \frac1{1 + (Cn)^{(9\eta -2) / 12}/\delta ^{1/3}} \sum_{|\mathbf{m}|\leq n^{\eta}}  
 \left\vert \langle \mathbf{l}_\lambda  | \mathbf{m}_ {\lambda}\rangle  \right\vert \leq 1
$$ 

where the last inequality follows from Lemma \ref{gqo}.
Now may apply Lemma \ref{isocomp}, and find an $R^r_{\lambda}$ such that 
$A^r_{\lambda} + R^r_{\lambda}$ is an isometry, and ${\rm Range}(R^r_{\lambda}) \perp 
{\rm Range}(A^r_{\lambda})$, so that $\langle  \mathbf{m}| R^r_{\lambda} = 0$. 
Defining $ V^r_{\lambda} : = A^r_{\lambda} + R^r_{\lambda}$, we have
\begin{eqnarray*}
\langle  \mathbf{m}| V^r_{\lambda} & = & \langle  \mathbf{m}| (A^r_{\lambda} + R^r_{\lambda}) \\
& = & \langle  \mathbf{m}| A^r_{\lambda} \\
& = &  \frac1{\sqrt{1 + (Cn)^{(9\eta -2) / 12}/\delta ^{1/3}}} \langle  \mathbf{m}| \sum_{|\mathbf{l}| \leq n^{\eta}} |  \mathbf{l}\rangle \left\langle \mathbf{l}_ {\lambda} \right| \\
& = & \frac1{\sqrt{1 + (Cn)^{(9\eta -2) / 12}/\delta ^{1/3}}} \left\langle \mathbf{m}_{\lambda} \right|.
\end{eqnarray*}

\end{proof}
\subsection{Proof of the main lemmas}
\begin{proof}[Proof of Lemma \ref{approx_classical}]
To obtain upper bounds for the following expression
$$\sup_{\theta \in \Theta}||\sum_{\lambda \in \Lambda_1}b_{\lambda}^{u,r,n}-\cal{N}_{r-1}(u,V_{\mu})||_1$$
one first approximates $p_{\lambda}^{u,r,n}$ by a $r$ dimensional multinomial distribution and then uses the classical channel $\tau_{\lambda}^{r,n}$ to map the multinomial distribution to a distribution in $\bb{R}^{r-1}$ which lies close to the multivariate normal. Thus we obtain the following variant of equation (7.35) of \cite{Kahn&Guta} 
\begin{align}
\label{multientre}
\left\| \mathcal{N}_{r-1}(u,V_{\mu}) -
\sum_{\lambda\in \Lambda_1}
b_{\lambda}^{u,r,n}\right\|_1 &\leq 
\left\| 
p^{u,r,n} -  M^n_{\mu^{u,n}_1,\dots,\mu^{u,n}_r}\right\|_1
\nonumber\\
&+ 
\left\| \mathcal{N}_{r-1}(u,V_{\mu}) -
\sum_{\lambda \in \Lambda_1}  M^n_{\mu^{u,n}_1,\dots,\mu^{u,n}_r}(\lambda) \tau_{\lambda}^{r,n}
\right\|_1, 
\end{align}
where $ M^n_{\mu^{u,n}_1,\dots,\mu^{u,n}_r}$ is the $r$-multinomial with coefficients $\mu_i^{u,n}$. One then shows that both the terms in the RHS is $O(n^{-1/2+\gamma}+n^{-1/4+\epsilon})$ for any $\epsilon>0$ and finally obtain:
\begin{align*}
\sup_{\theta \in \Theta}||\sum_{\lambda \in \Lambda_1}b_{\lambda}^{u,r,n}-\cal{N}_{r-1}(u,V_{\mu})||_1\leq C_r(n^{-1/2+\gamma}+n^{-1/4+\epsilon})
\end{align*}
where $C_r$ is a constant which depends only on the rank $r$. Since the proof is given in \cite{Kahn&Guta} we omit the details.
\end{proof}
\begin{proof}[Proof of Lemma \ref{extreme_rep}]
Note that $$p_{\lambda}^{u,r,n}=\prod_{1\leq i \leq r}(\mu^{u,n}_i)^{\lambda_i}\sum_{\bm\in \cal{M}^{\lambda}_r}\prod_{1\leq i<j\leq r}\left(\frac{\mu^{u,n}_j}{\mu^{u,n}_i}\right)^{m_{i,j}}\times \mathrm{dim}(\mathcal{K}_{\lambda}).$$
where (see \cite{MR1153249})
$$\mathrm{dim}(\mathcal{K}_{\lambda})=\left(\substack{n\\\\\lambda_1,\lambda_2,\ldots,\lambda_r}\right)\prod_{i=1}^r\frac{\lambda_l !\prod_{k=l+1}^r\lambda_l-\lambda_k+k-l}{(\lambda_l+r-l)!}.$$
The rest of the proof proceeds as in \cite{Kahn&Guta}. One checks that 
$$\prod_{i=1}^r\frac{\lambda_l !\prod_{k=l+1}^r\lambda_l-\lambda_k+k-l}{(\lambda_l+r-l)!}\leq 1$$
and  
$$\left(\frac{\mu^{u,n}_j}{\mu^{u,n}_i}\right)^{m_{i,j}}\leq 1, \quad \sum_{m \in \mathcal{M}^{\lambda}_r}\left(\frac{\mu^{u,n}_j}{\mu^{u,n}_i}\right)^{m_{i,j}}\leq n^{r^2},$$
so that
$$\sum_{\lambda\in \Lambda_1\cap\Lambda_2^c}p_{\lambda}^{u,r,n}\leq n^{r^2}\sum_{\lambda\in \Lambda_1\cap\Lambda_2^c}\left(\substack{n\\\\\lambda_1,\lambda_2,\ldots,\lambda_r}\right)\prod_{1\leq i \leq r}(\mu^{u,n}_i)^{\lambda_i}.$$
Also note that 
$$\Lambda_1\cap\Lambda_2^c=\{\lambda:\lambda_k=0 \text { for } k>r \text { and } \exists i \text{ such that } |\lambda_i-n\mu_i|>n^{\alpha} \}$$
Thus we are summing a multinomial probability distribution over extreme values. Since each coordinate of the multinomial is a binomial random variable one uses Hoeffding inequality for Bernoulli random variables and union bound to obtain
$$\sum_{\lambda\in \Lambda_1\cap\Lambda_2^c}p_{\lambda}^{u,r,n}\leq 2rn^{r^2}\exp(-n^{2\alpha-1}/2)$$
for large enough $n$.
\end{proof}
\begin{proof}[Proof of lemma \ref{comparison_unshifted}]
We start by considering the states

$$\rho^{0,u,r,n}_{\lambda}=\frac{1}{N^{u,r,n}_{\lambda}}\sum_{\{m_i\}}\prod_{i=1}^r (\mu^{u,n}_{i})^{\lambda_i}\prod_{j=i+1}^r \left(\frac{\mu^{u,n}_{j}}{\mu^{u,n}_{i}}\right)^{m_{i,j}}\sum_{m \in \cal{M}^{\lambda}_r\cap \{m_i\}}|\bar{\bm}_{\lambda}\rangle\langle\bar{\bm}_{\lambda}|
 $$
\begin{equation}\phi^{0,r}=\sum_{\mathbf{m}_1\in \bb{N}^{r(r-1)/2}}\prod_{1\leq i<j\leq r}\frac{\mu_i-\mu_j}{\mu_i}\left(\frac{\mu_j}{\mu_i}\right)^{m_{i,j}}|\mathbf{m}_1\ra\la \mathbf{m}_1|\otimes |\mathbf{0}_2\ra\la \mathbf{0}_2|.
\end{equation}
Note that although $|\mathbf{m}_\lambda\ra$ are eigenvectors of $\rho^{0,u,r,n}_{\lambda}$, they are not orthogonal and hence the decomposition is not straightforward.
 Let \begin{equation}
     \cal{H}^r_{\lambda}=\text{ span } \{|\mathbf{m}_\lambda\rangle: \bm\in \mathcal{M}_r^{\lambda}\}
 \end{equation}
 We split $\cal{H}^r_{\lambda}$ into direct sum of two spaces 
i.e.
 $$\cal{H}^r_{\lambda}=\cal{H}^r_{\lambda,\eta}\bigoplus \bar{\cal{H}}^r_{\lambda,\eta}$$
 where 
$$\cal{H}^r_{\lambda,\eta}=\text{ span } \{|\bm_\lambda\rangle \in \cal{H}^r_{\lambda}: |\bm|\leq n^{\eta}\}.$$
Also decompose $\mathcal{F}^r=\mathcal{F}^r_{\eta}\bigoplus \mathcal{F}^{r\perp}_{\eta}$ with $\{|\bm_1\rangle\otimes |\mathbf{0}_2\rangle:|\bm_1|\leq n^{\eta}\}$ spanning the set $\mathcal{F}^r_{\eta}$. Define the projectors into the subspaces $\cal{H}^r_{\lambda,\eta},\mathcal{F}^r_{\eta},$ and $\mathcal{F}^{r\perp}_{\eta}$
 as $P^r_{\lambda,\eta},P^r_{\eta},$ and $P^{r\perp}_{\eta}$ respectively. One can show that the following decomposition holds:
\begin{align}||T^r_{\lambda}(\rho^{0,u,r,n}_{\lambda})-\phi^{0,r}||_1\leq 2||T^r_{\lambda}(P^{r}_{\lambda,\eta}\rho^{0,u,r,n}_{\lambda}P^{r}_{\lambda,\eta})-P^{r}_{\eta}\phi^{0,r}P^{r}_{\eta}||_1+2||P^{r\perp}_{\eta}\phi^{0,r}P^{r\perp }_{\eta}||_1 \label{low_excite_decomp}.
\end{align}
Indeed consider Hermitian trace class operators $A,B,A_1,B_1,A_2,B_2$ such that
$$A=A_1+A_2,\quad B=B_1+B_2.$$
Further assume that $A,A_1,B_1,B_2$ are positive semidefinite operators and $\Tr {A}=\Tr {B}=1$. Now for a  Hermitian operator trace class operator $X$ define
$$|X|=\sqrt{X^*X}, \quad X_+=(|X|+X)/2, \quad X_-=(|X|-X)/2$$
so that $X=X_+-X_-$, $|X|=X_++X_-$ and $||X||_1=\Tr |X|$.

Using this decomposition into positive and negative parts we have
\begin{align*}
    A-B&=A_1-B_1+A_2-B_2\\
    &=(A_1-B_1)_+-(A_1-B_1)_-+(A_2-B_2)_+-(A_2-B_2)_-.
\end{align*}
Since $\Tr(A-B)=0$ we obtain
\begin{align}
\label{positive_negative_decomp}
\Tr(A_2-B_2)_+&=\Tr(A_2-B_2)_-+\Tr(A_1-B_1)_--\Tr(A_1-B_1)_+\\
&\leq \Tr(A_2-B_2)_-+||A_1-B_1||_1. \nonumber
\end{align}
We obtain the following upper bound
\begin{align}\label{low_excit_predecomp}
    ||A-B||_1&\leq ||A_1-B_1||_1+||A_2-B_2||_1\\
    & = ||A_1-B_1||_1+\Tr(A_2-B_2)_++\Tr(A_2-B_2)_-\nonumber\\
    & \leq 2||A_1-B_1||_1+2\Tr(A_2-B_2)_-\nonumber\\
    & \leq  2||A_1-B_1||_1+2||B_2||_1\nonumber.
\end{align}

For the second inequality we have used the relation (\ref{positive_negative_decomp}) while for the
last inequality we have used the relation
$$\Tr(A_2-B_2)_-=\Tr(B_2-A_2)_+\leq \Tr(B_2)_+\leq ||B_2||_1.$$
Let $A=T^r_{\lambda}(\rho^{0,u,r,n}_{\lambda})$, $B=\phi^{0,r}$, $A_1=T^r_{\lambda}(P^{r}_{\lambda,\eta}\rho^{0,u,r,n}_{\lambda}P^{r}_{\lambda,\eta})$, $B_1=P^{r}_{\eta}\phi^{0,r}P^{r}_{\eta}$, and $B_2=P^{r\perp}_{\eta}\phi^{0,r}P^{r\perp }_{\eta}$. Then we use (\ref{low_excit_predecomp}) to obtain (\ref{low_excite_decomp}).

Being a geometric sum the second term of relation (\ref{low_excite_decomp}) can be shown to be $\max_{1\leq i<j \leq r}(\mu_i/\mu_j)^{n^{\eta}}$\newline$=O(\exp(-c_1 n^{\eta}))$ for some constant $c_1$. To deal with the first term note that $|\bm|\leq n^{\eta}$ and we have by Lemma \ref{V_approx}
$$P(\{m_i\})=\sum_{\bm \in \cal{M}^{\lambda}_r\cap \{m_i\}}|\bar{\bm}_{\lambda}\rangle\langle\bar{\bm}_{\lambda}|=\frac{1}{1+Cn^{(9\eta-2)/12}}\sum_{\bm \in \cal{M}^{\lambda}_r\cap \{m_i\}}|\bm_{\lambda}\ra\la \bm_{\lambda}|+E(\{m_i\})$$
and the positive remainder $E(\{m_i\})$ has trace norm
$$\Tr(E(\{m_i\})=O(n^{(9\eta-2)/12}).\text{dim}(\mathcal{H}(\{m_i\}))$$
where $\mathcal{H}(\{m_i\})$ is the space on which $P(\{m_i\})$ projects.
Multiplying by the operator $P^r_{\lambda,\eta}$ on both sides we obtain 
\begin{equation}P^r_{\lambda,\eta}\rho^{0,u,r,n}_{\lambda}P^r_{\lambda,\eta}=\frac{1}{1+Cn^{(9\eta-2)/12}}\breve{\rho}^{0,u,r,n}_{\lambda}+\frac{1}{N^{u,r,n}_{\lambda}}\sum_{\{m_i\}}\prod_{i=1}^r (\mu^{u,n}_{i})^{\lambda_i}\prod_{j=i+1}^r \left(\frac{\mu^{u,n}_{j}}{\mu^{u,n}_{i}}\right)^{m_{i,j}}E(\{m_i\})\label{approx_state_decomp}.\end{equation}
where the approximate state $\breve{\rho}^{0,u,r,n}_{\lambda}$ is given by
$$\breve{\rho}^{0,u,r,n}_{\lambda}=\frac{1}{N^{u,r,n}_{\lambda}}\sum_{\{m_i\}}\prod_{i=1}^r (\mu^{u,n}_{i})^{\lambda_i}\prod_{j=i+1}^r \left(\frac{\mu^{u,n}_{j}}{\mu^{u,n}_{i}}\right)^{m_{i,j}}\sum_{\substack{\bm \in \cal{M}^{\lambda}_r\cap \{m_i\}\\|\bm|\leq n^{\eta}}}|\bm_{\lambda}\ra\la \bm_{\lambda}|.$$
Trace norm of the second term in (\ref{approx_state_decomp}) is given by $$\Tr\left(\frac{1}{N^{u,r,n}_{\lambda}}\sum_{\{m_i\}}\prod_{i=1}^r (\mu^{u,n}_{i})^{\lambda_i}\prod_{j=i+1}^r \left(\frac{\mu^{u,n}_{j}}{\mu^{u,n}_{i}}\right)^{m_{i,j}}E(\{m_i\})\right)=O(n^{9\eta-2/12})$$
and since $T^r_{\lambda}$ is contractive we have 
$$||T^r_{\lambda}(P^r_{\lambda,\eta}\rho^{0,u,r,n}_{\lambda}P^r_{\lambda,\eta}-\tilde{\rho}^{0,u,r,n}_{\lambda})||_1\leq||P^r_{\lambda,\eta}\rho^{0,u,r,n}_{\lambda}P^r_{\lambda,\eta}-\tilde{\rho}^{0,u,r,n}_{\lambda}||_1=O(n^{9\eta-2/12}).$$
Thus we need to give an upper bound to  $||T^r_{\lambda}(\breve{\rho}^{0,u,r,n}_{\lambda})-P^{r}_{\eta}\phi^{0,r}P^{r}_{\eta}||_1$ in order to give an upper bound to the RHS of (\ref{low_excite_decomp}).
Since $\mu_i^{u,n}=\mu_i+O(n^{-1/2+\gamma})$ for $|m|\leq n^{\eta}$ we have
$$\left(\frac{\mu^{u,n}_{j}}{\mu^{u,n}_{i}}\right)^{m_{ij}}=\left(\frac{\mu_i}{\mu_j}\right)^{m_{ij}}(1+O(n^{-1/2+\gamma+\eta}))$$ and hence the from the definition of $N^{u,r,n}_{\lambda}$ we have
\begin{align*}
    \left[\frac{1}{N^{u,r,n}_{\lambda}}\prod_{i=1}^r (\mu^{u,n}_{i})^{\lambda_i}\right]^{-1}=&\sum_{\bm\in \mathcal{M}^r_{\lambda},|\bm|\leq n^{\eta}}\prod_{1\leq i<j \leq r}\left(\frac{\mu^{u,n}_{j}}{\mu^{u,n}_{i}}\right)^{m_{ij}}+\sum_{\bm\in \mathcal{M}^r_{\lambda},|\bm|> n^{\eta}}\prod_{1\leq i<j \leq r}\left(\frac{\mu^{u,n}_{j}}{\mu^{u,n}_{i}}\right)^{m_{ij}}\\
    = &\sum_{\bm\in \mathcal{M}^r_{\lambda}} \prod_{1\leq i<j \leq r}\left(\frac{\mu_{j}}{\mu_{i}}\right)^{m_{ij}}(1+O(n^{-1/2+\gamma+\eta}))\\
    & -\sum_{\bm\in \mathcal{M}^r_{\lambda},|\bm|> n^{\eta}} \prod_{1\leq i<j \leq r}\left(\frac{\mu_{j}}{\mu_{i}}\right)^{m_{ij}}(1+O(n^{-1/2+\gamma+\eta}))\\
    &+ \sum_{\bm\in \mathcal{M}^r_{\lambda},|\bm|> n^{\eta}}\prod_{1\leq i<j \leq r}\left(\frac{\mu^{u,n}_{j}}{\mu^{u,n}_{i}}\right)^{m_{ij}}.\\
\end{align*}
Again by a geometric sum argument the last two terms can easily be seen to be of the order 
$O(\exp(-c_2 n^{\eta}))$ for some constant $c_2$ and  hence we obtain
\begin{align}\left[\frac{1}{N^{u,r,n}_{\lambda}}\prod_{i=1}^r (\mu^{u,n}_{i})^{\lambda_i}\right]^{-1}&=\sum_{\bm\in \mathcal{M}^r_{\lambda}} \prod_{1\leq i<j \leq r}\left(\frac{\mu_{j}}{\mu_{i}}\right)^{m_{ij}}(1+O(n^{-1/2+\gamma+\eta}))\nonumber\\
&=\prod_{1\leq i<j \leq r} \frac{\mu_i}{\mu_i-\mu_j}(1+O(n^{-1/2+\gamma+\eta})) \label{approx_constant}
\end{align}
By definition of $\mathcal{M}^r_{\lambda}$ $m_{i,j}=0$ unless $i,j\leq r$. Using the fact that $\|  |\psi\rangle\langle \psi| - |\psi^{\prime} \rangle \langle \psi^{\prime}| \|_{1}= 2\sqrt{ 1 - |\langle\psi
|\psi^{\prime}\rangle|^2}$ (see (2))  holds for unit vectors ($\psi,\psi^{\prime}$) in conjunction with Lemma \ref{V_approx} we obtain
\begin{align*}
||T^r_{\lambda}(|\mathbf{\bm}_{\lambda}\ra\la\mathbf{\bm}_{\lambda}|)-|\mathbf{m}_1\ra\la \mathbf{m}_1|\otimes |\mathbf{0}_2\ra\la \mathbf{0}_2|||_1& =||V^r_{\lambda}|\mathbf{m}_{\lambda}\ra\la\mathbf{m}_{\lambda}|V^{r*}_{\lambda}-|\mathbf{m}_1\ra\la \mathbf{m}_1|\otimes |\mathbf{0}_2\ra\la \mathbf{0}_2|)||_1\\
&=O(n^{9\eta-2/24}).
\end{align*}
Note that we have ignored the $\delta$ since the minimum gap between the eigenvalues is fixed. Consequently
\begin{align*}T^r_{\lambda}(\tilde{\rho}^{0,u,r,n}_{\lambda})&=\frac{1}{N^{u,r,n}_{\lambda}}\sum_{\{m_i\}}\prod_{i=1}^r (\mu^{u,n}_{i})^{\lambda_i}\prod_{j=i+1}^r \left(\frac{\mu^{u,n}_{j}}{\mu^{u,n}_{i}}\right)^{m_{i,j}}\sum_{\substack{\bm \in \cal{M}^{\lambda}_r\cap \{m_i\}\\|\bm|\leq n^{\eta}}}T^r_{\lambda}(|\bm_{\lambda}\ra\la \bm_{\lambda}|)\\
&= \frac{1}{N^{u,r,n}_{\lambda}}\prod_{i=1}^r (\mu^{u,n}_{i})^{\lambda_i}\sum_{\{m_i\}} \prod_{1\leq i<j \leq r} \left(\frac{\mu^{u,n}_{j}}{\mu^{u,n}_{i}}\right)^{m_{i,j}}\sum_{\substack{m \in \cal{M}^{\lambda}_r\cap \{m_i\}\\|\bm|\leq n^{\eta}}}T^r_{\lambda}(|\bm_{\lambda}\ra\la \bm_{\lambda}|)\\
&= \sum_{|\bm_1|\leq n^{\eta}}\prod_{1\leq i<j \leq r} \frac{\mu_i-\mu_j}{\mu_i}\left(\frac{\mu_{j}}{\mu_{i}}\right)^{m_{i,j}}|\mathbf{m}_1\ra\la \mathbf{m}_1|\otimes |\mathbf{0}_2\ra\la \mathbf{0}_2|+O(n^{-1/2+\gamma+\eta},n^{9\eta-2/24}) \\
&= P^r_{\eta}\phi^{0,r}P^r_{\eta} + O(n^{-1/2+\gamma+\eta},n^{9\eta-2/24}) 
\end{align*}
which was to be shown.

\end{proof}

\begin{proof}[Proof of Lemma \ref{comparison_shift_operator}]
Similar to the lemmas in subsections D.1 and D.2 this lemma will be presented from the original Lemma 6.4 of \cite{Kahn&Guta} \textit{mutatis mutandis} for the convenience of the reader. Recall that $|\mathbf{0}\ra\la \mathbf{0}|$ is the vacuum in the multimode Fock space $\mathcal{F}^r$ where the number of modes is $r(r-1)/2+r(d-r)$. Let us define
$$D^{\mathfrak{z}_1+\bz,\mathfrak{z}_2}(|\mathbf{0}\ra\la \mathbf{0}|)=
|\mathfrak{z}+\bz\ra \la \mathfrak{z}+\bz|.$$
Although $\mathfrak{z}$ and $\bz$ are of different dimension we can identify $\bz$ as $(\bz,0)$, so that
$|\mathfrak{z}+\bz\ra$ makes sense. Also note that $|\mathfrak{z}+\bz\ra \la \mathfrak{z}+\bz|$ is a pure state in the multimode Fock space $\mathcal{F}^r$.

Similarly define $|\mathfrak{z}+\bz,\lambda\ra = U^r_{\lambda}((\mathfrak{z}+\bz)/\sqrt{n})|\mathbf{0}_{\lambda}\ra$ and hence 
$$T^r_{\lambda}\Delta_{\lambda}^{\mathfrak{z}+\bz,r,n}T^{r*}_{\lambda}(|\mathbf{0}\ra\la \mathbf{0}|)=V^r_{\lambda}|\mathfrak{z}+\bz,\lambda\ra\la\mathfrak{z}+\bz,\lambda| V^{r*}_{\lambda}.$$
According to Lemma \ref{V_approx}, the coordinates of 
$V^r_{\lambda} \vert  \frak{z} + \bz, \lambda \rangle$ 
in the Fock basis are described  by:
\begin{align}
\label{etape1}
\langle {\bf m} | V^r_{\lambda} \vert \frak{z}+\bz,\lambda \rangle =
\left\{ \begin{array}{l} 
\langle {\bf m}_{\lambda}| U^r_{\lambda}((\frak{z}+\bz)/\sqrt{n}) |{\bf 0}_{\lambda}\rangle 
(1 + O(n^{(9\eta -2) / 12}\delta ^{-1/3}))
\mbox{ if }
|{\bf m}| \leq n^{\eta};\vspace{2mm}\\
\mbox{something not important if } |{\bf m}| > n^{\eta} .
\end{array} \right.
\end{align}
To show that $$\sup_{\theta \in \Theta,||\bz||\leq n^{\beta}}\sup_{\lambda\in \Lambda_1\cap \Lambda_2}||(T^r_{\lambda}\Delta_{\lambda}^{\mathfrak{z}+\bz,r,n}T^{r*}_{\lambda}-D^{\mathfrak{z}+\bz})(|\mathbf{0}\ra\la \mathbf{0}|)||_1 \rightarrow 0$$
 and $$\sup_{\theta \in \Theta,||\bz||\leq n^{\beta}}\sup_{\lambda\in \Lambda_1\cap \Lambda_2}||(D^{\bz,0}-T^r_{\lambda}\Delta_{\lambda}^{\bz,r,n}T^{r*}_{\lambda})(|\mathbf{0}\ra\la \mathbf{0}|)||_1\rightarrow 0$$
 it is enough to show that the first limit goes to $0$ since the second one is a special case of the first. Recalling  
(2) we observe that it is enough to show
$$\sup_{||\bz||\leq n^{\beta}}\sup_{\mathfrak{z}\in \Theta}\sup_{\lambda \in \Lambda_1\cap \Lambda_2} 1- |\la\mathfrak{z}+\bz|V^r_{\lambda}|\mathfrak{z}+\bz,\lambda\ra|=R(n)^2$$ 
since $V^r_{\lambda}|\mathfrak{z}+\bz,\lambda\ra\la\mathfrak{z}+\bz,\lambda| V^*_{\lambda}$ and $|\mathfrak{z}+\bz\ra\la\mathfrak{z}+\bz|$ are pure states. Here $R(n)$ goes to $0$ (we will give an explicit expression of this term later).

Recall that $\mathbf{m}=\{m_{i,j}\}_{1\leq i\leq r,i<j\leq d}$ and also note that we need the full $\mathbf{m}$ here in sharp contrast with $\bf{m}_1$ (i.e. $\bf{m}\in \mathcal{M}^r_{\lambda}$) used in Lemma \ref{comparison_unshifted}. First we decompose the inner product as follows
\begin{equation}\la\mathfrak{z}+\bz|V^r_{\lambda}|\mathfrak{z}+\bz,\lambda\ra=\sum_{m}\la\mathfrak{z}+\bz|\bf{m}\ra\la \bf{m} |V^r_{\lambda}|\mathfrak{z}+\bz,\lambda\ra\label{decomp_inner_prod}\end{equation}

Consider the following set
\begin{equation}
\label{adaptedm}
\mathcal{M}:= \{{\bf m} : m_{i,j}\leq |(\frak{z} + z)_{i,j}|^2 n^{\epsilon} 
, \quad 1\leq i\leq r,i<j\leq d \}.
\end{equation} 
In particular, since $2\beta+ \epsilon<\eta$ we have $\mathcal{M}\subset \{{\bf m} : |{\bf m}| \leq n^{\eta}\}$.

 \medskip 
We can upper bound the overlap of the coherent state with $|\bf{m}\ra$s where $m\notin \mathcal{M}$ as follows:
\begin{equation}
\sum_{\mathbf{m} \not \in \mathcal{M}}  | \langle\frak{z} + \bz | {\bf m}\rangle |^2\leq
\sum_{\ijr}  \exp(-x_{i,j}) \sum_{k> x_{i,j} n^{\epsilon}} \frac{x_{ij}^{k}}{k!}\leq rdn^{-\epsilon}, \qquad \text{ where }
x_{i,j}=  | \frak{z} + \bz |^2_{i,j}.\label{tail}
\end{equation}
The last inequality follows since each of the terms in the sum is a tail of Poisson distribution and is bounded by 
$ n^{-\epsilon}$.

\medskip

 From the first line of \eqref{etape1} we get
\begin{align*}
\langle {\bf m}| V^r_{\lambda}  U^r_{\lambda}((\frak{z}+\bz)/\sqrt{n}) |{\bf
0}_{\lambda}\rangle 
& = \frac{ \langle y_{\lambda} \Vm | y_{\lambda}
U^r_{\lambda}((\frak{z}+\bz)/\sqrt{n}) | f_{\bf 0}\rangle}{\sqrt{\langle y_{\lambda} \Vzero |
y_{\lambda}\Vzero\rangle}\sqrt{\langle y_{\lambda} \Vm |
y_{\lambda}\Vm\rangle}} (1 + O(n^{(9\eta -2) / 12}\delta ^{-1/3}))
\\
& = \frac{\langle p_{\lambda} \Vm  | q_{\lambda} U^r_{\lambda}((\frak{z}
+
\bz)/\sqrt{n})
 |f_{\bf 0}\rangle }{\sqrt{\langle p_{\lambda} \Vm |
q_{\lambda} p_{\lambda} \Vm \rangle}} (1 + O(n^{(9\eta -2) / 12}\delta ^{-1/3}) ).
\end{align*} 
The first line follows since $y_{\lambda}$ and $U^r_{\lambda}((\frak{z}
+
\bz)/\sqrt{n})$ commutes and for the second line we have used (\ref{csq1}), (\ref{eq.norm.0.lambda}) and (\ref{fcoherent}).

We recall that  $ \mathcal{O}_{\lambda}({\bf m})$ is the orbit in $
(\mathbb{C}^{d})^{\otimes n}$ of $\Vm$ under $\mathcal{R}_{\lambda}$ and that we have the decomposition
$$
p_{\lambda} \Vm = \sum_{\Va\in \mathcal{O}_{\lambda}({\bf m})}
\frac{\# \mathcal{R}_{\lambda}}{\# \mathcal{O}_{\lambda}({\bf m})} \Va.
$$

Using \eqref{numer} and \eqref{denom54}, we can write
\begin{align}
\label{dev}
 \langle {\bf m}| V^r_{\lambda}  U^r_{\lambda}((\frak{z}+\bz)/\sqrt{n}) |{\bf 0}_{\lambda}\rangle 
&
= \frac{\sum_{f_{\bf{a}}\in \mathcal{O}_{\lambda} ({\bf m})}\langle f_{\bf{a}} | q_{\lambda}
U^r_{\lambda}((\frak{z} + \bz)/\sqrt{n}) f_{\mathbf{0}}\rangle}{\sqrt{\sum_{f_{\bf{a}},f_{\bf{b}} \in
\mathcal{O}_{\lambda} ({\bf m})} \langle  f_{\bf{a}} | q_{\lambda} f_{\bf{b}} \rangle}}(1 + O(n^{(9\eta -2) / 12}\delta ^{-1/3}) \\
&
= e^{ -\|\frak{z} + \bz\|^2_2/2}\prod_{\ijr}\frac{(\frak{z} + z)_{i,j}^{m_{i,j}}}{\sqrt{m_{i,j}!}}\left(\frac{n(\mu_i - \mu_j)}{\lambda_i - \lambda_j}\right)^{m_{i,j}/2}r(n). \notag
\end{align} 
The corresponding remainder term is
\begin{multline*}
r(n) =  1 + O\left(n^{(9\eta -2) / 12}\delta ^{-1/3},n^{-1 + 2\beta + \eta}\delta^{-1}, n^{-1/2 + 3\beta+2\epsilon} \delta^{-3/2}, n^{-1 + \alpha + 2\beta}\delta^{-1},\right.\\
\left.n^{-1 + \alpha + \eta}\delta^{-1}, n^{-1 + 3\eta}\delta^{-1} \right).
\end{multline*}
Since $\lambda \in \Lambda_2$ and the eigenvalues are separated by $\delta$ we have $\left(\frac{n(\mu_i - \mu_j)}{\lambda_i - \lambda_j}\right)^{m_{i,j}/2} = 1 + O(n^{\alpha - 1 + \eta} / \delta) $ and the error can be absorbed in $r(n)$.

Thus, for $\mathbf{m}$ satisfying  \eqref{adaptedm}, we have:
\[
 \langle
{\bf m} | V^r_{\lambda} U^r_{\lambda}((\frak{z}+\bz)/\sqrt{n}) |{\bf 0}_{\lambda}\rangle = 
\langle {\bf m} | \frak{z} + \bz \rangle r(n) .
\]
If $a_{\bf m}$ and $b_{\bf m}$ are the two sets of coefficients, such that 
$\sum_{\bf m} |a_{\bf m}|^2 = \sum_{\bf m} |b_{\bf m}|^2 = 1$, then  
\begin{equation}
\label{boundsubset}
 1 - \left| \sum_{\bf m} a_{\bf m} b_{\bf m} \right|   \leq  1- \left| \sum_{{\bf m}\in\mathcal{M}} a_{\bf m} b_{\bf m} \right|  + 
\left| \sum_{{\bf m}\notin \mathcal{M}} a_{\bf m} b_{\bf m}\right|  \leq 2 
\left(1-  \left|\sum_{{\bf m}\in\mathcal{M}} a_{\bf m} b_{\bf m}\right| \right).
\end{equation}
We use \eqref{decomp_inner_prod} in conjunction with 
\eqref{tail} and \eqref{boundsubset}, we get
\begin{eqnarray}\label{final.ip.approx}
1-\left|\langle \frak{z}+\bz| V^r_{\lambda}  U^r_{\lambda}((\frak{z}+\bz)/\sqrt{n}) 
|{\bf 0}_{\lambda}\rangle \right|& = &  
O\left( 1- r(n),  \sum_{\mathbf{m} \not\in \mathcal{M}}  |\langle {\bf m} | \frak{z} + \bz \rangle |^2\right) =  R_2(n),
\end{eqnarray}
with
\begin{multline*}
R_2(n) =  O\left(n^{(9\eta -2) / 12}\delta ^{-1/3},n^{-1 + 2\beta + \eta}\delta^{-1}, n^{-1/2 + 3\beta+2\epsilon} \delta^{-3/2}, n^{-1 + \alpha +2 \beta}\delta^{-1}, \right.\\
\left. n^{-1 + \alpha + \eta}\delta^{-1}, n^{-1 + 3\eta}\delta^{-1}, n^{-\epsilon}\right).
\end{multline*}
Note that we are interested in the square root of the LHS of (\ref{final.ip.approx}). Setting $R(n)=\sqrt{R_2(n)}$ and ignoring $\delta$ in the final statement (since it is a fixed constant) we obtain the desired result.
\end{proof}

\begin{proof}[Proof of Lemma \ref{group_structure}]
Note that although $SU(d)$ forms a group the unitaries $U^r$ do not form a subgroup. Instead we view the unitaries as elements of $SU(d)$ and the corresponding $\Delta^{\bz,r,n}_{\lambda}$ as $\Delta^{\bz,n}_{\lambda}$
and directly use Lemma 6.5 of \cite{Kahn&Guta} which states that
$$\sup_{||\bz||\leq n^{\beta}}\sup_{\theta \in \Theta_{n,\beta,\gamma}}\sup_{\lambda\in \Lambda_{n,\alpha}}||(\Delta_{\lambda}^{\mathfrak{z}+\bz,n}-\Delta_{\lambda}^{\mathfrak{z},n}\Delta_{\lambda}^{\bz,n})(|\mathbf{0}_{\lambda}\ra\la \mathbf{0}_{\lambda}|)||_1=R(n)$$
Note that in our case also $T^{r*}_{\lambda}(|\mathbf{0}\ra\la \mathbf{0}|)=|\mathbf{0}_{\lambda}\ra\la \mathbf{0}_{\lambda}|$. Finally we identify $\Delta^{\bz,r,n}_{\lambda}$s as $\Delta^{\bz,n}_{\lambda}$s as mentioned and since we are taking  supremum over restricted sets (for the coordinates of $\theta$ with $i,j\geq r$ is $0$ and $\lambda$ is further restricted to $\Lambda_1$) we obtain
\begin{align*}&\sup_{\theta \in \Theta,||\bz||\leq n^{\beta}}\sup_{\lambda\in \Lambda_1\cap \Lambda_2}||(\Delta_{\lambda}^{\mathfrak{z}+\bz,r,n}-\Delta_{\lambda}^{\mathfrak{z},r,n}\Delta_{\lambda}^{\bz,r,n})T^{r*}_{\lambda}(|\mathbf{0}\ra\la \mathbf{0}|)||_1\\
& \leq \sup_{||\bz||\leq n^{\beta}}\sup_{\theta \in \Theta_{n,\beta,\gamma}}\sup_{\lambda\in \Lambda_{n,\alpha}}||(\Delta_{\lambda}^{\mathfrak{z}+\bz,n}-\Delta_{\lambda}^{\mathfrak{z},n}\Delta_{\lambda}^{\bz,n})(|\mathbf{0}_{\lambda}\ra\la \mathbf{0}_{\lambda}|)||_1=R(n).
\end{align*}
\end{proof}
\begin{proof}[Proof of Lemma \ref{lower_bound_error}]
Define $u_r=\sum_{i=1}^{r-1}u_i$ and $\hat{u}_r=\sum_{i=1}^{r-1}\hat{u}_i$
.
\begin{align*}\int_{\cal{U}^c
    }E_u[\cal{L}(u,\hat{u})]d\pi_1(u)&= \int_{\cal{U}^c
    }E_u\left[\sum_{i=1}^{r-1}(u_i-\hat{u}_i)^2+(\sum_{i=1}^{r-1}(u_i-\hat{u}_i))^2\right]d\pi_1(u)\\
    &\leq 2 \sum_{i=1}^{r}\int_{\cal{U}^c
    }E_u\left[(u^2_i+\hat{u}^2_i)\right]d\pi_1(u)\\
    & \leq  2 \sum_{i=1}^{r}
    \sqrt{E(u^4_i)}\sqrt{\pi_1(\cal{U}^c)}+rn^{2\gamma}\pi_1(\cal{U}^c).\\
    \end{align*}
In the first part we have used Cauchy-Schwartz inequality and in the second part we have used the fact that since $\hat{u}\in \cal{U}$ each component $|\hat{u}_i| \leq n^{\gamma}$ for $i=1,\ldots,r-1$ and hence $|\hat{u}_r|\leq rn^{\gamma}$.

Note that for $i=1,\ldots,r-1$,$u_i\sim \mathcal{N}(0,n^{2\gamma_0})$ and hence $Eu_i^4=3n^{4\gamma_0}$. It can be easily seen that $u_r\sim \mathcal{N}(0,(r-1)n^{2\gamma_0})$ and so $Eu_r^4=3(r-1)^2n^{2\gamma_0}$. Also $P(|u_i|\geq n^{\gamma})\leq c\exp(-n^{2\gamma}/2n^{2\gamma_0})$ for $i=1,\ldots,r-1$, so that by union bound we have
$\pi(\cal{U}^c)\leq r c\exp(-n^{2\gamma}/2n^{2\gamma_0})$. Thus we have
$$\int_{\cal{U}^c
    }E_u[\cal{L}(u,\hat{u})]d\pi_1(u)\leq 2 \sum_{i=1}^{r}
    \sqrt{3r^2n^{4\gamma_0}}\sqrt{r c\exp(-n^{2\gamma}/2n^{2\gamma_0}}+rn^{2\gamma}r c\exp(-n^{2\gamma}/2n^{2\gamma_0})$$
    
and the RHS of the above inequality tends to $0$ since $\gamma_0<\gamma$.

For the quantum part we have
$$\int_{\cal{Z}^c}E_{\xi}[\cal{L}(\xi,\hat{\xi})]d\pi_2(\xi)=\int_{\cal{Z}^c}E_{\xi}\sum_{\substack{{1\leq i\leq r}\\{i<j\leq d}}}(\mu_i-\mu_j)|\xi_{ij}-\hat{\xi}_{ij}|^2d\pi_2(\xi).$$
It is enough to show that  $\int_{\cal{Z}^c}E_{\xi}||\xi_{ij}-\hat{\xi}_{ij}||^2d\pi_2(\xi)\rightarrow 0$ for each pair of $i,j$. As in the classical part we have the following inequalities:
\begin{align*}
    \int_{\cal{Z}^c}E_{\xi}|\xi_{ij}-\hat{\xi}_{ij}|^2d\pi_2(\xi)&\leq 2\int_{\cal{Z}^c}E_{\xi}\left[||\xi_{ij}||^2+||\hat{\xi}_{ij}||^2\right]d\pi_2(\xi)\\
    & \leq2(\sqrt{E|(\xi_{ij})_1|^4}+\sqrt{E|(\xi_{ij})_2|^4})\sqrt{\pi_2(\cal{Z}^c)}+4n^{2\beta}\pi_2(\cal{Z}^c)\\
    & \leq 4\sqrt{3n^{4\beta_0}}\sqrt{rd c\exp(-n^{2\beta}/2n^{2\beta_0}}+4n^{2\beta}rd c\exp(-n^{2\beta}/2n^{2\beta_0}).
\end{align*}
Again the RHS goes to $0$ as $\beta_0<\beta$.
\end{proof}
\begin{proof}[Proof of Lemma 5.3]
Consider the first order approximation of the state
$\rho_{\theta,r}$ i.e. consider the expression
in (\ref{rho.theta.tilde}). Then we have,
\[||\tilde{\rho}_{\theta^{(1)},r}-\tilde{\rho}_{\theta^{(2)},r}||^2_2=\sum_{i=1}^r(u^{(1)}_i-u^{(2)}_i)^2+2\sum_{\substack{{1\leq i\leq r}\\{i<j\leq d}}}|\zeta^{(1)}_{ij}-\zeta^{(2)}_{ij}|^2\]
Replacing $u_i$ by $u_i/\sqrt{n}$ and $\zeta_{ij}$ by $\mathfrak{z}_{ij}\sqrt{\frac{(\mu_i-\mu_j)}{n}}$ we obtain the desired result with the error term $O\left(\frac{||\theta^{(1)}||^3,||\theta^{(2)}||^3}{n^{3/2}}\right)$ corresponding to the error in first order approximation. 
\end{proof}

\end{document}